\documentclass[onefignum,onetabnum]{siamonline220329}

\usepackage{amsmath}
\usepackage{mathtools}
\usepackage{amsfonts,amssymb}
\usepackage{empheq}
\usepackage{graphicx}
\usepackage{epstopdf}
\usepackage{algorithmic}
\usepackage{subcaption}  %subfigure
\usepackage{float}

\usepackage{multicol}
\usepackage{rotating}
\usepackage{comment}

\makeatletter
\newcommand\Label[1]{&\refstepcounter{equation}(\theequation)\ltx@label{#1}&}
\makeatother

\allowdisplaybreaks
\ifpdf
  \DeclareGraphicsExtensions{.eps,.pdf,.png,.jpg}
\else
  \DeclareGraphicsExtensions{.eps}
\fi
\usepackage[numbers, sort]{natbib}
\usepackage{bm}

\usepackage{xcolor, framed}

\newenvironment{myframedeq}[1][\linewidth]{\FrameSep=4pt\abovedisplayskip=0pt\belowdisplayskip=0pt
\framed\hsize=#1\leftskip=\dimexpr(\textwidth-#1)/2\relax}
{\endframed}

\newcommand{\rulesep}{\unskip\ \vrule\ }

\newcommand{\R}{\mathbb{R}}
\newcommand{\F}{\mathbb{F}}
\newcommand{\EE}{\mathcal{E}}
\newcommand{\PP}{\mathcal{P}}
\newcommand{\MM}{\mathcal{M}}
\newcommand{\TT}{\mathcal{T}}
\newcommand{\QQ}{\mathcal{Q}}
\newcommand{\BB}{\mathcal{B}}

\DeclareMathOperator*{\tr}{tr }
\DeclareMathOperator*{\asym}{Asym}
\DeclareMathOperator*{\sym}{Sym}

\DeclareMathOperator*{\tmp}{tmp}
\DeclareMathOperator*{\tot}{tot}
\DeclareMathOperator*{\dist}{dist}
\DeclareMathOperator*{\jump}{Jump}
\DeclareMathOperator*{\imp}{Imp}
\DeclareMathOperator*{\argmin}{arg\,min}

\newcommand{\pluseq}{\mathrel{{+}{=}}}
\newcommand{\minuseq}{\mathrel{{-}{=}}}

\usepackage{enumitem}
\setlist[enumerate]{leftmargin=.5in}
\setlist[itemize]{leftmargin=.5in}

\newsiamremark{remark}{Remark}
\newsiamremark{hypothesis}{Hypothesis}
\crefname{hypothesis}{Hypothesis}{Hypotheses}
\newsiamthm{claim}{Claim}

\headers{Lie Group Variational Collision Integrators for a Class of Hybrid Systems}{K. Tran and M. Leok}

\title{Lie Group Variational Collision Integrators for a Class of Hybrid Systems\thanks{Submitted to the editors 10/20/2023.
\funding{This work was supported in part by NSF under grants DMS-1345013, DMS-1813635, CCF-2112665, DMS-2307801, by AFOSR under grant FA9550-18-1-0288, FA9550-23-1-0279, and by the DoD under grant 13106725 (Newton Award for Transformative Ideas during the COVID-19 Pandemic).}}}

\author{Khoa Tran\thanks{Department of Mathematics, University of California at San Diego, La Jolla, CA 
  (\email{k2tran@ucsd.edu}).}
\and Melvin Leok\thanks{Department of Mathematics, University of California at San Diego, La Jolla, CA  
  (\email{mleok@ucsd.edu}).}
}

\usepackage{amsopn}

\newfloat{alg}{luc}{Algorithm}

\ifpdf
\hypersetup{
  pdftitle={Lie Group Variational Collision Integrators for a Class of Hybrid Systems},
  pdfauthor={K.Tran, M. Leok}
}
\fi

\begin{document}

\maketitle

% REQUIRED
\begin{abstract}
The problem of 3-dimensional, convex rigid-body collision over a plane is fully investigated; this includes bodies with sharp corners that is resolved without the need for nonsmooth convex analysis of tangent and normal cones. In particular, using nonsmooth Lagrangian mechanics, the equations of motion and jump equations are derived, which are largely dependent on the collision detection function. Following the variational approach, a Lie group variational collision integrator (LGVCI) is systematically derived that is symplectic, momentum-preserving, and has excellent long-time, near energy conservation. Furthermore, systems with corner impacts are resolved adeptly using $\epsilon$-rounding on the sign distance function (SDF) of the body. Extensive numerical experiments are conducted to demonstrate the conservation properties of the LGVCI.
\end{abstract}

% REQUIRED

% REQUIRED

\section{Introduction}
\label{sect:Intro}
	Hybrid systems are dynamical systems that exhibits both continuous and discrete dynamics. The state of a hybrid system changes either continuously by the flow described by differential equations or discretely following some jump conditions.  A canonical example of a hybrid system is the bouncing ball, imagined as a point mass, over a horizontal plane. The extension of this problem to 3-dimensions, wherein the bouncing body is rigid and convex, is rather complex, especially in the case of sharp corner impacts; in fact, these systems have unilateral constraints that describe the collision surface. We study such problems with perfectly elastic collisions and the Lie group variational collision integrators (LGVCI) are derived following the approaches introduced in \citep{fetecau2003nonsmooth} and \citep{lee2007lie}. The advantage of these frameworks is that they yield a global description of the system, in contrast to local representations such as Euler angles \citep{markeev2008,polukoshko2013}. Furthermore, in high-precision physics engine and graphics dynamics, the integrator becomes a foundation, and its extensions with inelastic collisions and friction can be derived to fully actualize the engine. This is also naturally applicable to problems in optimal control with similar nonlinear manifold constraints \citep{oberblobaum2011DMOC, pekarek2012variationalnonsmooth, pekarek2013projected, pekare2014projection, leyendecker2010optimal}. In particular, these constraints and optimal control problems arise in robotics \citep{pekarek2007discrete, duruisseaux2023LieFVIN, werner2017optimal, werner2018structure, leyendecker2013structure} and multi-body dynamics \citep{leyendecker2007dmocConstrainedMulti, johnson2014discontinuous}.
	
	There is an extensive literature on various extensions of the bouncing ball example. In fact, it is a subset of the broader, classical field of rigid-body dynamics, which has a strong emphasis on collisions, contact, and friction. Due to its practical importance and various theoretical challenges, there have been extensive studies which have been summarized in textbooks \citep{routh1897dynamics, brach2001mechanical, goldsmith1960impact, johnson1985contact, stronge2000impact, PfeifferGlocker1996MultiBody, christoph2001nonsmooth, brogliato2016nonsmooth}, some of which are considered classical references in the field. However, there is a notable absence in the literature on global descriptions of rigid-body dynamics with collisions, e.g., configurations of the bodies via Lie groups. This is due to the fact that systems with collisions are not continuous because there is an instantaneous jump in momenta after each collision. Consequently, researchers have focused mostly on discrete impact solutions of systems with rigid-bodies, dating back to Brach's  work in 1989 \citep{brach1984friction,brach1989rigid}. 
	
	There are, of course, more general studies of impacts for rigid-bodies, for example, the case of extremely small deformations (also known as ``hard'' collisions) at the contacting points; the research on this case stem from three main theories regarding the compression and restitution phases brought forth by Newton in 1833 \citep{newton1833philosophiae}, Poisson in 1838 \citep{poisson1838traite}, and Stronge as early as 1990 \citep{stronge1989rigid, stronge1991unraveling, stronge1994planar, stronge2000contact}. Actually, these ``extremely small'' deformations are one of the three main categories of collision problems indicated by Chatterjee and Ruina in \citep{chatterjee1998algebraic} and Najafabadi et al. in \citep{najafabadi2008generalization}. The first other category is ``small'' deformation collision which can be resolved using compliant contact modeling such as Hertz's model \citep{hertz1896miscellaneous} and non-linear damping model introduced by Hunt and Crossley \citep{hunt1975damping}; the second is ``large'' deformation collisions which require tools from continuum mechanics, e.g., finite element methods \citep{sha2012nonlinear,dintwa2008finite, kitamura2002FEM,wei1995finite}. In short, the literature and history of rigid-body and contact dynamics modeling are extensive, so one may refer to the following surveys and reviews in \citep{gilardi2002literature, fetecau2003nonsmooth, skrinjar2018review, flores2022review} for a more complete picture. 
	
	Our approach to the collision problem for a convex rigid-body is based on the variational methodology and integrators developed in Fetecau et al.~\citep{fetecau2003nonsmooth}. They specifically develop the theory for nonsmooth Lagrangian mechanics, which automatically gives a symplectic-momentum preserving integrator.  Furthermore, near impacts,  a collision point and time are determined to solve for the next configurations using the variational method as well. This approach was extended to develop collision algorithms for dissipative systems \citep{limebeer2020variational, portillo2017energy, chen2022stochastic} that take advantage of the near-energy preserving properties of the variational integrators in the absence of dissipation in order to more reliably track the energy decay of dissipative systems. The case of nonsmooth field theories was considered in \citep{demoures2016multisymplectic}, which is built on multisymplectic field theories \citep{gotay1997momentum} and multisymplectic variational integrators \citep{marsden1998multisymplectic}. 
	
	This paper, however, extends the work of Fetecau et al. to the 3-dimensional case and  explicitly uses the special Euclidean group to give a complete description of the system away from and during impacts. In addition, it investigates the equations of motion and jump conditions at impact for a class of hybrid systems, in which a convex rigid-body is bouncing elastically over a horizontal/tilted plane. The corresponding Lie group variational collision integrators (LGVCI) are derived, and extensions of the algorithm are developed for rigid-bodies with sharp corners, drawing from the tools of solid geometry. In particular, the signed distance functions \citep{osher2003signed} will be utilized to cleverly regularize our hybrid systems at corner impacts; consequently, nonsmooth analysis and differential inclusions may be avoided entirely. This work provides the foundation for future directions involving dissipation, multi-body, and articulated rigid-body collisions.
		
	%----------------------------------------------%
	% SECTION 1.1
	%----------------------------------------------%
	\subsection{Contributions}
	\label{sect:IntroContributions}
	%----------------------------------------------%
	We first investigate an ellipsoid bouncing elastically on the horizontal plane. The equations of motion and the jump conditions during impacts are derived, and they are expressed in terms of the signed distance function between the ellipsoid and plane. Furthermore, the signed distance function allows us to detect collisions and it has the necessary regularity for us to describe the jump conditions. For the collision response, jump conditions give a unique, instantaneous configuration after each impact time based solely on the instantaneous configuration before and the tangent space to the configuration at impact. 
		
	The paper also develops the LGVCI for our hybrid system. The integrators adopt the usual discrete flow for configurations away from the impact points, and the discrete equations are modified to describe the discrete flow near and at the impacts. Since the integrators are based on variational integrators, they are symplectic, momentum-preserving, and exhibit excellent long-time, near energy conservation. In addition, they respect the Lie group structure of the configuration space. Numerical simulations of the triaxial ellipsoid are presented, and we discuss the \textit{Zeno phenomenon}, wherein a hybrid system makes an infinite number of jumps -- collisions in this case -- within a finite time. 

	We demonstrate how to extend the model problem, by considering tilted planes and the unions and/or intersections of convex rigid-bodies. We further develop a sensible and practical regularization for the collision response of convex rigid-bodies with sharp corners that avoids the need for nonsmooth convex analysis, tangent, and normal cones. Since the tangent space of the configuration is not well-defined for corner impacts, the method introduces a regularization by smoothing the boundary of the bodies by a small $\epsilon$-parameter to handle such collisions. We provide numerical results for case of tilted planes, unions of two ellipsoids, and cube. 
	%----------------------------------------------%
	% SECTION 1.2
	%----------------------------------------------%
	\subsection{Organization}
	\label{sect:IntroOrganization}
	%----------------------------------------------%
	The paper is organized as follows: In Section \ref{sect:Problem}, background material and a description of the model problem is given. The theory of nonsmooth Lagrangian mechanics and the corresponding collision variational discretization are presented in Section \ref{sect:Backg}. In Section \ref{sect:Dynamics}, we derive the full equations of motion with jump conditions at the point of collision, and then the full variational integrators and algorithms are derived in Section \ref{sect:LGVCI}. We provide the extension to titled planes and more general rigid-bodies in Section \ref{sect:Extensions}. Finally, numerical experiments for four different hybrid systems are explored, and further discussions of the algorithm are given in Section \ref{sect:NumericalExp}. 

%----------------------------------------------%
%----------------------------------------------%
	%SECTION 2: THE PROBLEM
%----------------------------------------------%
%----------------------------------------------%
\section{The Problem: Dynamics of a Bouncing Ellipsoid}
\label{sect:Problem}
	We want to analyze the dynamics of an ellipsoid bouncing elastically on the plane under the effect of gravity. Hereafter, we will refer to this as the \textit{dynamics of a bouncing ellipsoid}. Some relevant background will be introduced to describe our system, which will provide the necessary foundation to develop our proofs, computations, and generalizations of our theory. 
	
	%----------------------------------------------%
	% SECTION 2.1
	%----------------------------------------------%
	\subsection{Notation}
	\label{sect:ProblemNotation}
	%----------------------------------------------%
	The notation we adopt for linear algebra and the configuration space of Special Euclidean group $SE(3)$ is presented here.
	
		%----------------------------------------------%
		% SECTION 2.1.1
		%----------------------------------------------%
		\subsubsection{Skew Map, Trace, and Inner products}
		\label{sect:ProblemNotation1}
		%----------------------------------------------%
		Recall that the Lie algebra $\mathfrak{so}(3)$ of the rotation group $SO(3)$ is the set of skew-symmetric matrices. Consider the \textit{skew map} $S : \R^3 \to \mathfrak{so}(3)$ defined by 
		$
			S(\bm{x}) =
			\begin{bsmallmatrix}
				0		&	-x_3	&	x_2		\\
				x_3		&	0		&	-x_1	\\
				-x_2	&	x_1		&	0
			\end{bsmallmatrix},
		$
		where $S(\bm{x})\bm{y} = \bm{x} \times \bm{y} $ for any $\bm{x},\bm{y} \in \R^3$. Also, one can show that the followings hold: 
		
		\begin{subequations}
		\label{eqn:skew_props}
		\noindent
			\begin{minipage}{0.5\textwidth}
				\begin{align}
					S(\bm{x})^T 
						&= -S(\bm{x}),	\label{eqn:skew1}\\
					S(\bm{x})^2
						&= \bm{x}\bm{x}^T - \|\bm{x}\|^2 I_3,	\label{eqn:skew2}
				\end{align}
			\end{minipage}%
			\begin{minipage}{0.5\textwidth}
				\begin{align}
					S(\bm{x} \times \bm{y} )
						&= S(\bm{x})S(\bm{y} ) - S(\bm{y} )S(\bm{x}), \label{eqn:skew3}	\\
					S(R\bm{x})
						&= RS(\bm{x})R^T,	\label{eqn:skew4}
				\end{align}
			\end{minipage}
			\vskip1em
		\end{subequations}
		\noindent for all $\bm{x},\bm{y} \in \R^3, R \in SO(3)$, where $I_3 \in \R^{3 \times 3}$ is the identity matrix and $\| \cdot \|$ is the Euclidean norm. Note that $S$ is an isomorphism with the inverse $S^{-1}(\Omega)^T = (\Omega_{32},\Omega_{13},\Omega_{21})$ for any $\Omega \in \mathfrak{so}(3)$.
	
		We introduce the following maps: $\asym : \R^{3\times 3} \to \R^{3 \times 3}$ and $\sym : \R^{3 \times 3} \to \R^{3 \times 3}$, which are defined respectively by $\asym(A) = A - A^T$ and $\sym(A) = A + A^T$ for any $A \in \R^{3 \times 3}$. The trace of a matrix is denoted by $\tr [A] = \sum_{i=1}^3 A_{ii}$, and satisfies the following property: 
		\begin{proposition}
		\label{prop:traceProp}
		For all skew-symmetric matrices $\Omega \in \R^{3 \times 3}$, $\tr [\Omega^T P] = 0$ if and only if $P \in \R^{3 \times 3}$ is symmetric, i.e., $\asym(P) = 0$.
		\end{proposition}
		This fact provides an insight into the usual inner product of $\R^{3 \times 3}$ where one of the matrices is skew-symmetric. Recall that the inner product $\langle \cdot, \cdot \rangle : \R^{3 \times 3} \times \R^{3 \times 3} \to \R$ is defined by $\langle A, B \rangle = \tr[B^T A]$. Suppose that $\Omega$ is skew-symmetric and $A \in \R^{3 \times 3}$. Since $A = \frac{1}{2}\sym(A) + \frac{1}{2}\asym(A)$, we have that by Proposition \ref{prop:traceProp},
		\begin{equation}
		\label{eqn:trIPtoSkewIP}
			\langle A, \Omega \rangle = \frac{1}{2} \tr[\Omega^T \asym(A)].
		\end{equation}
		 Note that $\asym(A)$ is also skew-symmetric, so this naturally leads us to the inner product on $\mathfrak{so}(3)$. In fact, the Lie algebra is a linear space with the inner product $\langle \cdot, \cdot \rangle_S : \mathfrak{so}(3) \times \mathfrak{so}(3) \to \R$, which is induced by the standard inner product on $\R^3$ via the skew map,
		\begin{equation}
		\label{eqn:lieinnerprod}
			\langle \Omega_1, \Omega_2 \rangle_S = \bm{\omega}_2^T \bm{\omega}_1,
		\end{equation}
		where $\Omega_1,\Omega_2 \in \mathfrak{so}(3)$ and $\bm{\omega}_1,\bm{\omega}_2 \in \R^3$ such that $S(\bm{\omega}_i) = \Omega_i$. It can be shown that $\langle \Omega_1, \Omega_2 \rangle_S =\tfrac{1}{2} \tr [\Omega_2^T \Omega_1]$, and by \eqref{eqn:trIPtoSkewIP}, the inner products on $\R^{3 \times 3}$ and $\mathfrak{so}(3)$ are related by $\langle A, \Omega \rangle = \langle \asym(A), \Omega \rangle_S$.
	
		%----------------------------------------------%
		% SECTION 2.1.2
		%----------------------------------------------%
		\subsubsection{Special Euclidean Group \texorpdfstring{$SE(3)$}{SE3}}
		\label{sect:ProblemNotation2}
		%----------------------------------------------%
	Given our main goal, it is natural to describe the translation and rotation of the ellipsoid using the \textit{Special Euclidean group} $SE(3)$ as our configuration space. Recall that $SE(3)$ is the Lie group defined by
		\begin{align*}
			SE(3) 
			= \{ (\bm{x},R) \in \R^3 \times GL(3) \mid R^TR = RR^T = I \text{ and } \det(R) = 1 \}
			= \R^3 \rtimes SO(3),
		\end{align*}
		where $\rtimes$ is the semidirect product. The semidirect product structure of $SE(3)$ can be conveniently encoded in terms of homogeneous transformations
			$G =
			\begin{bsmallmatrix}
				R	&	\bm{x}	\\
				0	&	1
			\end{bsmallmatrix}$
		where the group operation is defined by the usual matrix multiplication. This allows $SE(3)$ to be viewed as an embedded submanifold in $GL_4(\R)$. Furthermore, its action on $\R^3$ is given by matrix-vector product once we embed $\R^3\hookrightarrow\R^3 \times \{1\} \subset \R^4$:
			\[
				\begin{bmatrix}
					R_2	&	\bm{x_2}	\\
					0	&	1
				\end{bmatrix}
				\begin{bmatrix}
					R_1	&	\bm{x_1}	\\
					0	&	1
				\end{bmatrix}
				=
				\begin{bmatrix}
					R_2R_1	&	R_2\bm{x_1}	+ \bm{x_2}\\
					0		&	1
				\end{bmatrix},
				\quad \text{and} \quad
				\begin{bmatrix}
					R	&	\bm{x}	\\
					0	&	1
				\end{bmatrix}
				\begin{bmatrix}
					\bm{z}	\\
					1
				\end{bmatrix}
				=
				\begin{bmatrix}
					R \bm{z} + \bm{x}	\\
					1
				\end{bmatrix},
			\]
		where $R, R_1, R_2 \in SO(3)$ and $\bm{z}, \bm{x}, \bm{x_1}, \bm{x_2} \in \R^3$. As a result, $(\bm{x},R) \in SE(3)$ represents a configuration of the rigid-body where $\bm{x}$ is the location of the origin of the body-fixed frame relative to the inertial frame and $R$ as the attitude of the body. In particular, if $\bm{\rho} \in \R^3$ is a vector expressed in the body-fixed frame, then $\bm{x} + R\bm{\rho} $ is the same vector expressed in the reference frame.
	
		Furthermore, the Lie algebra $\mathfrak{se}(3)$ is given by $\mathfrak{se}(3) = \{ (\bm{y},\Omega) \mid \bm{y} \in \R^3 \text{ and } \Omega \in \mathfrak{so}(3) \}$, and its corresponding homogeneous representation is
		$			
			V
			=
			\begin{bsmallmatrix}
				\Omega	&	\bm{y}	\\
				0		&	0
			\end{bsmallmatrix}.
		$
		It has an induced inner product from the standard inner product of $\R^3$: For all $(\bm{y}_1,\Omega_1), (\bm{y}_2,\Omega_2) \in \mathfrak{se}(3)$,
		\begin{equation}
			(\bm{y}_1,\Omega_1) \cdot_S	 (\bm{y}_2,\Omega_2) 
			= \bm{y}_2^T \bm{y}_1 + \langle \Omega_1, \Omega_2 \rangle_S
			= \bm{y}_2^T \bm{y}_1 + \tfrac{1}{2}\tr[\Omega_2^T\Omega_1].
		\end{equation}
		
	%----------------------------------------------%
	% SECTION 2.2
	%----------------------------------------------%
	\subsection{Distance between an arbitrary ellipsoid and the plane}
	\label{sect:ProblemDistanceForm}
	%----------------------------------------------%
	In order to simulate the dynamics of a bouncing ellipsoid, it is crucial to be able to perform \textit{collision detection} between the ellipsoid and plane for each integration step. If the distance between the two is greater than zero, the next integration step is considered; if it is less than zero, we discard the current step and utilize a rootfinder to find the integration step that advances the solution to the impact point and time. Then, we use the variational collision integrator at the impact point to apply the discrete jump conditions, and then take the remainder of the integration step.
	
	 We briefly introduce the notation for the ellipsoid and plane, and then the formula for the collision detection function will be given. This formula can be derived using inversive geometry and/or constrained optimization. For practicality, the plane is fixed as the horizontal plane defined by $\PP = \{ \bm{z} \in \R^3 \mid \bm{n}^T\bm{z} + 0 = 0 \}$ where $\bm{n}^T = (0,0,1)$. Now, suppose that $a$, $b$, $c > 0$ and consider $f_\EE : \R^3 \to \R$ defined by
	\begin{equation}
	\label{eqn:fE}
		f_\EE(\bm{z}) 
		= \frac{z_1^2}{a^2} + \frac{z_2^2}{b^2} + \frac{z_3^2}{c^2}
		= \bm{z}^T (I_\EE^{-1})^2 \bm{z},
		\qquad \text{where } I_\EE = \text{diag}(a,b,c).
	\end{equation}
	Let us write the standard ellipsoid centered at the origin as $\EE(a,b,c) = \{ \bm{z} \in \R^3 \mid f_\EE(\bm{z}) \leq 1 \}$ with the shorthand notation $\EE$ when the lengths of the semiaxes are understood. Let $(\bm{x},R) \in SE(3)$ denote the configuration of the ellipsoid, where $\bm{x}$ is the center of mass and $R$ is the attitude of the ellipsoid. Consider the map $T_{(\bm{x},R)} : \R^3 \to \R^3$ defined by $T_{(\bm{x},R)}(\bm{z}) = R\bm{z} + \bm{x}$, which represents the action of $SE(3)$ on $\R^3$. Then, the arbitrary ellipsoid is simply the image of the map, denoted by $\EE' = T_{(\bm{x},R)}(\EE)$.
	
	For a strictly positive distance between the plane and ellipsoid, their intersection is empty. Thus, $\EE' \cap \PP = \emptyset$ and $x_3 > 0$ since the center of mass of the ellipsoid is always above the horizontal plane for our system. In the case that the distance is zero, their intersection must be a singleton since $\PP$ is closed and convex and $\EE'$ is compact and strictly convex. Essentially, this is one of the key assumptions of our system to reduce complexity and ensure uniqueness of the ellipsoid's trajectory based on the variational approach. The empty intersection condition gives us the distance formula and a condition:
	\begin{equation}
	\label{eqn:BouncingEllipsoidRaw}
		d_2(\EE',\PP) = \min \{ |\bm{n}^T \bm{x}| \pm \|I_\EE R^T \bm{n} \| \}
		\qquad \text{and} \qquad
		\|I_\EE R^T \bm{n} \| < |\bm{n}^T \bm{x}|,
	\end{equation}
	Note that the inequality $\|I_\EE R^T \bm{n} \| < |\bm{n}^T \bm{x}|$ is equivalent to the condition $\EE' \cap \PP = \emptyset$. Furthermore, given $\bm{n}^T \bm{x} = x_3 > 0$ with the inequality, the minimum value is chosen with the minus sign. As a result, we have the following proposition for our hybrid system:
	\begin{proposition}
	\label{prop:distanceFunction}
		Let $\Phi : SE(3) \to \R$ be the \textbf{collision detection function}, which is defined by
		\begin{equation}
		\label{eqn:distBouncingEllipsoid}
			\Phi(\bm{x},R) =\bm{n}^T \bm{x} - \|I_\EE R^T \bm{n} \|,
		\end{equation}
		where $\bm{n}$ is the normal vector of the plane $\PP$ and $I_\EE = \text{diag}(a,b,c)$. Let $(\bm{x},R) \in SE(3)$, so that $\EE' = T_{(\bm{x},R)}(\EE)$ is the arbitrary ellipsoid. Assume that $\EE' \cap \PP = \emptyset$ and $x_3 > 0$, then $d_2(\EE',\PP) = \Phi(\bm{x},R)$.
	\end{proposition}	
	The collision detection function $\Phi$ allows us to characterize the admissible set of configurations for the bouncing ellipsoid system. Namely, if the center of mass is below the plane ($x_3 < 0$), we get $\Phi(\bm{x},R) < 0$. In the case that the ellipsoid intersects the plane, the inequality becomes $\|I_\EE R^T \bm{n} \| \geq x_3$, and so $\Phi(\bm{x},R) \leq 0$. Of course, equality here implies that the ellipsoid makes an impact on the plane without \textit{interpenetration} -- a description where no body element of the ellipsoid crosses the plane. Therefore, the signed distance function satisfies $\Phi(\bm{x},R) \geq 0$ for all the ``allowable'' configurations $(\bm{x},R) \in SE(3)$, which we will define more rigorously next. 
	
	%----------------------------------------------%
%----------------------------------------------%
	%SECTION 3: BACKGROUND
%----------------------------------------------%
%----------------------------------------------%
\section{Background}
\label{sect:Backg}
	Before the variational collision integrators are derived for the bouncing ellipsoid, we will first give an overview of the main ideas and techniques in both the continuous-time and discrete-time setting following the approach developed in \citep{fetecau2003nonsmooth,pekarek2010variational}. Specific results for the ellipsoid dynamics will be stated in this section, and this will provide us with the tools to construct the necessary integrators in the subsequent sections.
	
	%----------------------------------------------%
	% SECTION 3.1
	%----------------------------------------------%
	\subsection{Continuous-Time Model.}
	\label{BackgCont}
	%----------------------------------------------%
	Let $Q$ be the configuration manifold. Let the submanifold $C \subset Q$ be the \textit{admissible set}, which consists of all the possible configurations where no contact occurs. The \textit{contact set} $\partial C$ consists of all the points at which contact occurs without any interpenetration.
		
		Consider a Lagrangian $L : TQ \to \R$. In the nonautonomous approach, we introduce a parameterized variable $\tau \in [0,1]$ for the time and trajectories in $Q$ with mappings $c_t(\tau)$ and $c_q(\tau)$, respectively. Assume that there is one contact point occurring at $\tau_i \in (0,1)$ for simplicity, but the theory can be easily extended for multiple contacts. Now, let us consider the smooth manifold \textit{path space} in \citep{fetecau2003nonsmooth}: $\MM = \TT \times \QQ([0,1], \tau_i, \partial C,Q)$ where
	\begin{align*}
		\TT
		&= \{ c_t \in C^\infty([0,1],\R) \mid c_t' > 0 \text{ in } [0,1] \},	\\
		\QQ([0,1], \tau_i, \partial C,Q)
		&= \{ c_q : [0,1] \to Q \mid c_q \in C^0, \text{ piecewise } C^2,	\\
			& \qquad \text{ has one singularity at } \tau_i, \text{ and } c_q(\tau_i) \in \partial C \}.
	\end{align*}
	Then $c \equiv (c_t,c_q) \in \MM$, and the \textit{associated curve} is $q : [c_t(0), c_t(1)] \to Q$ given by $q(t) \coloneqq c_q(c_t^{-1}(t))$. As a result, we are essentially considering paths on an \textit{extended configuration manifold} $Q_e = \R \times Q$. Now, the \textit{extended action map} $\tilde{\mathfrak{S}} : \MM \to \R$ is given by
	\begin{equation}
	\label{eqn:extActionMap}
		\tilde{\mathfrak{S}}(c)= \int_0^1 \tilde{L}(c(\tau),c'(\tau)) \, d \tau,
	\end{equation}
	where $\tilde{L} : TQ_e \to \R$ is defined as
	\begin{equation}
	\label{eqn:extLMap}
		\tilde{L}(c(\tau),c'(\tau)) = L\left(c_q(\tau), \frac{c_q'(\tau)}{c_t'(\tau)}\right)c_t'(\tau).
	\end{equation}
	The factor of $c_t'$ appears due to the Jacobian from the change of coordinates $s = c_t(\tau)$ for the usual action of the associated curve $\mathfrak{S}(q) = \int_{c_t(0)}^{c_t(1)} L(q(s),\dot{q}(s)) \, ds$.
	
	Lastly, we introduce
	$
		\ddot{Q}_e = \left\{ \frac{d^2 c}{d \tau^2} (0) \in T(TQ_e) \mathrel{\bigg|} c \in C^2([0,1],Q_e)\right\}
	$
	as the the \textit{extended configuration manifold} of $Q_e$. This will help us derive the equations of motion and jump conditions after taking variations of the action:
	\begin{theorem}
	\label{thm:ContinuousThm}
		Given $C^k$ Lagrangian $L$, $k \geq 2$, there exist a unique $C^{k-2}$ \textbf{Euler--Lagrange derivative} $EL : \ddot{Q}_e \to T^{*}Q_e$ and a unique $C^{k-1}$ \textbf{Lagrangian one-form} $\Theta_L$ on $TQ_e$ such that for all variations $\delta c \in T_c \MM$, the variation of the extended action is given by
		\begin{equation}
		\label{eqn:actionVariation}
			d\tilde{\mathfrak{S}} \cdot \delta c
			= \int_0^{\tau_{i}} EL(c'') \cdot \delta c \, d \tau + \int_{\tau_{i}}^1 EL(c'') \cdot \delta c \, d \tau + \Bigl.\Theta_L(c') \cdot \hat{\delta} c \Bigr \lvert_{0}^{\tau_{i}^-} + \Bigl.\Theta_L(c') \cdot \hat{\delta} c \Bigr \lvert_{\tau_{i}^+}^{1},
		\end{equation}
	where, in coordinates,
	\begin{subequations}
	\label{eqn:ELFormality}
		\begin{align}
			EL(c'')
				&= \left[ \frac{\partial L}{\partial q}c_t' - \frac{d}{d \tau} \left(\frac{\partial L}{\partial \dot{q}} \right) \right]dc_q + \left[\frac{d}{d \tau}\left( \frac{\partial L}{\partial \dot{q}} \frac{c_q'}{c_t'} - L  \right) \right]dc_t,	\label{eqn:EL}\\
			\Theta_L(c')
				&= \left[ \frac{\partial L}{\partial \dot{q}} \right]dc_q - \left[ \frac{\partial L}{\partial \dot{q}}\frac{c_q'}{c_t'} - L \right]dc_t,	\label{eqn:ThetaL}\\
			\hat{\delta}c(\tau)	
				&= \left( \left( c(\tau),\frac{\partial c}{\partial \tau}(\tau) \right), \left(\delta c (\tau), \frac{\partial \delta c}{\partial \tau}(\tau) \right) \right).
		\end{align}
	\end{subequations}
	\end{theorem}
		
	Hamilton's principle states that the possible trajectories of the system are the critical points of the action map. Therefore, all of the solution paths $c \in \MM$ must satisfy $d\tilde{\mathfrak{S}} \cdot \delta c = 0$ for all $\delta c \in T_c \MM$, which vanish at the boundary points $\tau = 0$ and $\tau = 1$. Then, if $c$ is a solution, it satisfies
	\begin{equation}
	\label{eqn:actionVariationwBoundaries}
		d\tilde{\mathfrak{S}} \cdot \delta c
		= \int_0^{\tau_{i}} EL(c'') \cdot \delta c \, d \tau + \int_{\tau_{i}}^1 EL(c'') \cdot \delta c \, d \tau + \Bigl.\Theta_L(c') \cdot \hat{\delta} c \Bigr \lvert_{\tau_{i}^-} ^{\tau_{i}^+} = 0.
	\end{equation}
	The above equation holds if and only if the Euler--Lagrange derivative is zero on the smooth portions $[0,\tau_i) \cup (\tau_i,1]$, and the Lagrangian one-form has a zero jump at $\tau_i$. The first gives us the \textit{extended Euler--Lagrange equations} expressed in the time domain as 
	
	\begin{subequations}
	\label{eqn:extELEquations}
	\noindent
		\begin{minipage}{0.5\textwidth}
			\begin{equation}
				\left(\frac{\partial L}{\partial q} - \frac{d}{dt} \left(\frac{\partial L}{\partial \dot{q}}\right)\right) \cdot \delta q = 0,
			\label{eqn:extEL1}
			\end{equation}
		\end{minipage}%
		\begin{minipage}{0.5\textwidth}
			\begin{equation}
			\frac{d}{dt}\left( \frac{\partial L}{\partial \dot{q}} \dot{q} - L \right) = 0,
			\label{eqn:extEL2}
			\end{equation}
		\end{minipage}
		\vskip1em
	\end{subequations}
	\noindent for all $ \delta q \in T_{q(t)}C$ and $t \in [t_0,t_i) \cup (t_i,t_1]$ where $t_0 = c_t(0), t_1 = c_t(1)$, and $t_i = c_t(\tau_i)$. There is some redundancy in this set of equations because \eqref{eqn:extEL2} is actually a consequence of \eqref{eqn:extEL1}: Note that $\frac{d}{dt} \left(\frac{\partial L}{\partial \dot{q}}\right) = \frac{\partial L}{\partial q}$ for all $\delta q \in T_{q(t)}C$, so 
	\begin{align*}
		\frac{d}{dt}\left( \frac{\partial L}{\partial \dot{q}} \dot{q} - L \right) 
		= \frac{d}{dt}\left( \frac{\partial L}{\partial \dot{q}} \right) \dot{q} + \frac{\partial L}{\partial \dot{q}}\ddot{q} - \frac{dL}{dt}
		= \left[ \frac{\partial L}{\partial q}\dot{q} + \frac{\partial L}{\partial \dot{q}}\ddot{q} \right] - \frac{dL}{dt}	
		= 0,
	\end{align*}
	since the time derivative of the Lagrangian is the expression in the square brackets. In addition, \eqref{eqn:extEL2} implies energy conservation for the autonomous system where $E = \frac{\partial L}{\partial \dot{q}} \dot{q} - L$, which is unsurprising, since the standard Euler--Lagrange equation already preserves energy. Hence, it suffices to consider the standard Euler--Lagrange equation for the equations of motion.
		
	For the Lagrangian one-form, one write it compactly as $\Theta_L = \frac{\partial L}{\partial \dot{q}} \, dq - E \, dt$ in the time domain. Then, having a zero jump at $\tau_i$ implies 
	
	\begin{subequations}
	\label{eqn:impactConserved}
	\noindent
		\begin{minipage}{0.5\textwidth}
			\begin{equation}
				\left(\biggl.\frac{\partial L}{\partial \dot{q}} \biggr\lvert_{t = t_i^+} - \biggl.\frac{\partial L}{\partial \dot{q}}\biggr\lvert_{t = t_i^-} \right)\cdot \delta q = 0,
			\label{eqn:impactLinearMomentumConserved}	
			\end{equation}
		\end{minipage}%
		\begin{minipage}{0.5\textwidth}
			\begin{equation}
			\Bigl.E(q,\dot{q}) \Bigr\lvert_{t = t_i^+} - \Bigl.E(q,\dot{q})\Bigr \lvert_{t = t_i^-} = 0,
			\label{eqn:impactEnergyConserved}
			\end{equation}
		\end{minipage}
		\vskip1em
	\end{subequations}
	\noindent for all $\delta q \in T_{q(t_i)} \partial C$. This set of equations is known as the \textit{Weierstrass--Erdmann} type conditions for impact. While \eqref{eqn:impactLinearMomentumConserved} indicates that the momentum is conserved in the tangential direction of $\partial C$, \eqref{eqn:impactEnergyConserved} indicates the conservation of energy during an elastic impact. 
	
	Together, \eqref{eqn:impactConserved} gives the solution to $q(t_i^+)$, and it is clear that an obvious solution would be $q(t_i^+) = q(t_i^-)$. However, this is not allowed as the trajectory will no longer stay within the admissible set $C$. For a contact set $\partial C$ that is a codimension-one smooth submanifold, the solution to \eqref{eqn:impactConserved} exists and is locally unique \citep{fetecau2003nonsmooth}.
	
		%----------------------------------------------%
		% SECTION 3.1.1
		%----------------------------------------------%
		\subsubsection{Legendre Transform}
		\label{BackgContLegendre}
		%----------------------------------------------%
		Although the proof of Theorem \ref{thm:ContinuousThm} is omitted, its results, namely the Euler--Lagrange derivative, were obtained by taking the variation with respect to the tangent bundle $TQ$. Hence, it is also natural to derive Hamilton's equations by taking variations in terms of momenta $\F L(q,\dot{q}) \in T^*Q$ where,
		\begin{equation}
		\label{eqn:LegendreTrans}
			\F L(q,\dot{q}) \cdot \delta \dot{q} = \biggl.\frac{d}{d\epsilon}\biggr\lvert_{\epsilon = 0} L(q,\dot{q}+\epsilon \delta\dot{q}).
		\end{equation}
		The map $\F L$ is the \textit{Legendre Transform} or \textit{fibre derivative}, and $T^*Q$ denotes the cotangent bundle. From this transformation, it can be shown that Hamilton's equations are equivalent to the Euler--Lagrange equations and they provide an alternative description of our system \citep{MaRa1999}. 
	
	%----------------------------------------------%
	% SECTION 3.2
	%----------------------------------------------%
	\subsection{Discrete-Time Model}
	\label{sect:BackgDiscrete}
	%----------------------------------------------%
	
	Ideally, we would like to introduce the discrete-time model analogously to the continuous case using the nonautonomous approach. However, there is no guarantee that there exists a varying timestep solution to the extended discrete Euler--Lagrange equations as discussed in \cite{kane1999symplectic} and \cite{lew2003variational}. Therefore, we shall only introduce the discrete-time equations using the autonomous approach with a fixed timestep $h \in \R$ instead. 
	
	In the discrete setting, consider a discrete Lagrangian $L_d : Q \times Q \to \R$ that approximates a segment of the autonomous action integral 
	\[
		L_d(q_k,q_{k+1},h) \approx \int_{t_k}^{t_{k+1}} L(q(t),\dot{q}(t)) \, dt,
	\]
	where $q_k = q(t_k)$, $q_{k+1} = q(t_{k+1})$, and $h = t_{k+1} - t_k$. In general, the action integral is approximated using discrete fixed timesteps: $t_k = kh$ for $k = 0, 1, \ldots, N$. Let $\tilde{\alpha} \in [0,1]$ such that $\tilde{\tau} = t_{i-1} + \tilde{\alpha}h$ is the fixed impact time corresponding to the parameterized variable. Denote the actual impact time by $\tilde{t} = t_{i-1} + \alpha h$, where $\alpha = t_d(\tilde{\alpha})$ with $t_d$ being a strictly increasing function from the closed unit interval onto the closed unit interval. Now, define the \textit{discrete path space} to be $\MM_d = \TT_d \times \QQ_d(\tilde{\alpha}, \partial C, Q)$ where
	\begin{align*}
		\TT_d 
			&= \{ t_d(\tilde{\alpha}) = \alpha \mid t_d \in C^\infty([0,1],[0,1]), t_d \text{ onto}, t_d' > 0 \},	\\
		\QQ_d(\tilde{\alpha}, \partial C, Q)
			&= \{ q_d : \, \{ t_0, \ldots, t_{i-1}, \tilde{\tau}, t_i, \ldots, t_N\} \to Q \mid q_d(\tilde{\tau}) \in \partial C \}.
	\end{align*}
	Moreover, we remark that $\TT_d$ is actually the closed unit interval $[0,1]$ given all the possible strictly increasing, surjective functions $t_d$.
	
	For a more convenient notation, identify the discrete trajectory with its image 
	\[
		(\alpha,q_d) \leftrightarrow (\alpha, \{ q_0, \ldots, q_{i-1},\tilde{q}, q_i, \ldots, q_N \}),
	\]
	where $q_k = q_d(t_k)$, $\tilde{q} = q_d(\tilde{\tau})$, and $\alpha = t_d(\tilde{\alpha})$. In fact, $\MM_d$ is isomorphic to $[0,1] \times Q \times \cdots \times \partial C \times \cdots \times Q$, so it can be viewed as a smooth manifold. Then, for $q_d \in \QQ_d(\tilde{\alpha}, \partial C, Q)$, we have the tangent space
	\[
		T_{q_d}\QQ_d(\tilde{\alpha}, \partial C, Q) = \{ v_{q_d} : \{ t_0, \ldots, t_{i-1}, \tilde{\tau}, t_i, \ldots, t_N\} \to Q \mid v_{q_d}(\tilde{\tau}) \in T_{\tilde{q}}\partial C\}.
	\]
	Now, for $(\alpha,q_d) \in \MM$, also identify $(\delta \alpha,v_{q_d}) \in T_{(t_d,q_d)}\MM_d$ with its image
	\[
		(\delta \alpha,v_{q_d}) \leftrightarrow (\delta \alpha, \delta q_d)= (\delta \alpha, \{ \delta q_0, \ldots, \delta q_{i-1},\delta \tilde{q}, \delta q_i, \ldots, \delta q_N \}),
	\]
	where $\delta q_k = v_{q_d}(t_k)$, $\delta \tilde{q} = v_{q_d}(\tilde{\tau})$, and $\delta\alpha = v_{t_d}(\tilde{\alpha})$.
	
	The \textit{discrete action map} $\mathfrak{S}_d : \MM_d \to \R$ is defined by
	\begin{equation}
	\label{eqn:discreteactionmap}
		\mathfrak{S}_d(\alpha,q_d)
		= \sum_{\substack{k=0 \\ k \neq {i-1}}}^{N-1} L_d(q_k,q_{k+1},h)
		+ L_d(q_{i-1},\tilde{q},\alpha h) + L_d(\tilde{q},q_i, (1-\alpha)h),
	\end{equation}
	on which we will take the variations with respect to the discrete path and parameter $\alpha$. Lastly, we define the \textit{discrete second-order manifold} to be
	$
		\ddot{Q}_d = Q \times Q \times Q,
	$
	which is locally isomorphic to $\ddot{Q}_e$.
	
	\begin{theorem}
	\label{thm:discreteThm}
		Given a $C^k$ discrete Lagrangian $L_d : Q \times Q \times \R \to \R$, $k \geq 1$, there exist a unique $C^{k-1}$ mapping $EL_d : \ddot{Q}_d \to T^{*}Q$ and one-forms $\Theta_{L_d}^{-}$ and $\Theta_{L_d}^{+}$ on the discrete Lagrangian phase space $Q \times Q$ such that for all variations $(\delta \alpha, \delta q_d) \in T_{(t_d,q_d)} \MM_d$ of $(t_d,q_d)$, the variation of the discrete action is given by
		\begin{align}
		\label{eqn:deltaDiscreteAction}
			\begin{split}
			d\mathfrak{S}_d(\alpha, q_d) \cdot (\delta \alpha, \delta q_d)
				&= \sum_{k=1}^{i-2} EL_d(q_{k-1},q_k,q_{k+1}) \cdot \delta q_k + \sum_{k=i+1}^{N-1} EL_d(q_{k-1},q_k,q_{k+1}) \cdot \delta q_k	\\
					& + \Theta_{L_d}^{+}(q_{N-1},q_N) \cdot (\delta q_{N-1}, \delta q_N) - \Theta_{L_d}^{-}(q_0,q_1) \cdot (\delta q_0, \delta q_1)	\\
					& + [D_2 L_d(q_{i-2},q_{i-1},h) + D_1 L_d(q_{i-1},\tilde{q},\alpha h)] \cdot \delta q_{i-1}	\\
					& + h[D_3 L_d(q_{i-1},\tilde{q},\alpha h) - D_3L_d(\tilde{q},q_i, (1-\alpha)h)] \cdot \delta \alpha	\\
		 			& + i^*(D_2 L_d(q_{i-1},\tilde{q},\alpha h) + D_1 L_d(\tilde{q},q_i, (1-\alpha)h)) \cdot \delta \tilde{q}	\\
		 			& + [D_2 L_d(\tilde{q},q_i, (1-\alpha)h) + D_1 L_d(q_i,q_{i+1},h)] \cdot \delta q_i,
			\end{split}
		\end{align}
	where $i^{*} : T^{*}Q \to T^{*}\partial C$ is the \textbf{cotangent lift} of the embedding $i : \partial C \to Q$.
	
	The map $EL_d$ is called the \textbf{discrete Euler--Lagrange derivative} and the one-forms $\Theta_{L_d}^{-}$ and $\Theta_{L_d}^{+}$ are the \textbf{discrete Lagrangian one-forms}. These are the coordinate expressions,
	\begin{equation}
	\label{eqn:discreteELequation}
		EL_d(q_{k-1},q_k,q_{k+1}) = [D_2 L_d(q_{k-1},q_k,h) + D_1 L_d(q_k,q_{k+1},h)] \, dq_k,
	\end{equation}
	for $k \in \{ 1, \ldots, i-2, i, \ldots, N-1 \}$ and 
	\begin{subequations}
	\label{eqn:discreteLagrangian}
	\begin{align}
		\Theta_{L_d}^{+}(q_k,q_{k+1})
			&= D_2 L_d (q_k,q_{k+1},h) \, dq_{k+1},	\label{eqn:discreteLagrangianp}\\
		\Theta_{L_d}^{-}(q_k,q_{k+1})
			&= -D_1 L_d (q_k,q_{k+1},h) \, dq_{k}.	\label{eqn:discreteLagrangianm}
	\end{align}
	\end{subequations}
	\end{theorem}
	
	Note that $D_i$ denotes the partial derivative with respect to the $i$-th argument of the discrete Lagrangian. Now, by the discrete Hamilton's principle, the possible discrete, stationary paths $(t_d,q_d)$ are critical points of the discrete action map. Therefore, the solution,
	\[
		(\alpha, \{ q_0, \ldots, q_{i-1},\tilde{q}, q_i, \ldots, q_N \}),
	\]
	%by identification 
	satisfies $d\mathfrak{S}_d(\alpha,q_d) \cdot (\delta \alpha, \delta q_d) = 0$ for all variations $(\delta \alpha, \delta q_d) \in T_{(\alpha,q_d)} \MM_d$ that vanish at index $0$ and $N$. 
	
	Using equation \eqref{eqn:deltaDiscreteAction}, we conclude that the discrete Euler--Lagrange derivative vanishes,
	\begin{equation}
	\label{eqn:discreteEL}
		[D_2 L_d(q_{k-1},q_k,h) + D_1 L_d(q_k,q_{k+1},h)] \cdot \delta q_k = 0, 
	\end{equation}
	for all $\delta q_k \in T_{q_k}{C}$ and all $k \in \{1, 2, \ldots, i-2, i+1, \ldots, N - 1 \}$. This is know as the \textit{discrete Euler--Lagrange equations}, which describes the system away from the impact point.
	
	Then, prior to the impact, we have this system of equations,
	\begin{subequations}
	\label{eqn:discreteImpact1}
	\begin{align}
		[D_2 L_d(q_{i-2},q_{i-1},h) + D_1 L_d(q_{i-1},\tilde{q},\alpha h)] \cdot \delta q_{i-1}
			&= 0,	\label{eqn:discreteImpact1pt1}\\
		\tilde{q}
			&\in \partial C,	\label{eqn:discreteImpact1pt2}
	\end{align}	
	\end{subequations}
	which can be used to solve for $\alpha$ and $\tilde{q}$, the impact point. Next, we have the discrete impact condition,
	\begin{subequations}
	\label{eqn:discreteImpact2}
	\begin{align}
		D_3 L_d(q_{i-1},\tilde{q},\alpha h) - D_3L_d(\tilde{q},q_i, (1-\alpha)h)
			&= 0,	\label{eqn:discreteImpact2pt1}\\
		i^*(D_2 L_d(q_{i-1},\tilde{q},\alpha h) + D_1 L_d(\tilde{q},q_i, (1-\alpha)h)) \cdot \delta \tilde{q}
			&= 0,	\label{eqn:discreteImpact2pt2}
	\end{align}
	\end{subequations}
	where we solve for $q_i$. To provide an interpretation for the discrete equations above, define the \textit{discrete energy} $E_d : Q \times Q \to \R$ by 
	\begin{equation}
	\label{eqn:discreteEnergy}
		E_d(q_k,q_{k+1},h) = -D_3L_d(q_k,q_{k+1},h).
	\end{equation}	 
	This definition is motivated by the fact that it yields the exact Hamiltonian if we apply it to the \textit{exact discrete Lagrangian} $L_d^E$. The exact discrete Lagrangian $L_d^E$ is equal to action integral of the exact solution of the Euler--Lagrange equations that satisfy the prescribed boundary conditions, and is related to Jacobi's solution of the Hamilton--Jacobi equation. Hence, equation \eqref{eqn:discreteImpact2pt1} represents the conservation of discrete energy while equation \eqref{eqn:discreteImpact2pt2} expresses the conservation of the discrete momentum projected onto $T^* \partial C$. Lastly, we obtain $q_{i+1}$ by solving the last system of equations,
	\begin{equation}
	\label{eqn:discreteImpact3}
		[D_2 L_d(\tilde{q},q_i, (1-\alpha)h) + D_1 L_d(q_i,q_{i+1},h)] \cdot \delta q_i = 0.
	\end{equation}
	Equations \eqref{eqn:discreteEL}-\eqref{eqn:discreteImpact2} and \eqref{eqn:discreteImpact3} describe the complete trajectory of the discrete path including the impact point and time. We will rely on these equations to construct our algorithm to simulate the dynamics of the bouncing ellipsoid.
	\begin{figure}
	\centering
	\includegraphics{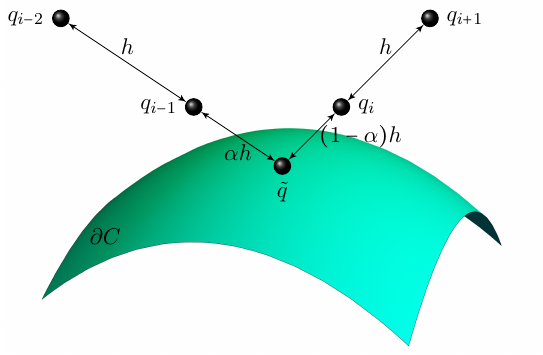}
	\caption{The discrete variational principle for collisions}
	\label{fig:varPrincipleCollisions}
	\end{figure}
		%----------------------------------------------%
		% SECTION 3.2.1
		%----------------------------------------------%
		\subsubsection{Discrete Legendre Transforms}
		\label{sect:BackgDiscreteLengenre}
		%----------------------------------------------%
		
		In the discrete-time model, we may also write the variational collision integrator in Hamiltonian form via the discrete analogue of the Legendre transform, known as the \textit{discrete Legendre transforms} or \textit{discrete fiber derivatives} $\F^\pm L_d : Q \times Q \to T^*Q$. They are defined by
		\begin{subequations}
		\label{eqn:discreteLegendreTrans}
		\begin{align}
			\mathbb{F}^- L_d(q_k,q_{k+1})	\cdot \delta q_k &= -D_1L_d(q_k,q_{k+1}) \cdot \delta q_k,	\label{eqn:discreteLegendreTrans1}\\
			\mathbb{F}^+ L_d(q_k,q_{k+1})	\cdot \delta q_{k+1} &= D_2L_d(q_k,q_{k+1}) \cdot \delta q_{k+1}.	\label{eqn:discreteLegendreTrans2}
		\end{align}
		\end{subequations}
		This allows us to introduce momenta as images of the discrete Legendre transform,
		\begin{subequations}
		\label{eqn:discreteLTransp}
		\begin{align}
			p^+_{k,k+1} = p^+(q_k,q_{k+1},h) = \mathbb{F}^+ L_d(q_k,q_{k+1}),	\label{eqn:discreteLTransp1} \\
			p^-_{k,k+1} = p^-(q_k,q_{k+1},h) = \mathbb{F}^- L_d(q_k,q_{k+1}).	\label{eqn:discreteLTransp2} 
		\end{align}
		\end{subequations}
		The discrete Euler--Lagrange equations \eqref{eqn:discreteEL} may be rewritten as $\mathbb{F}^+L_d(q_{k-1},q_k) \cdot \delta q_k  = \mathbb{F}^- L_d(q_k,q_{k+1}) \cdot \delta q_k$. This translates to $p^+_{k-1,k} = p^-_{k,k+1}$, which shows that the discrete momentum expressed on the time interval $[t_{k-1},t_k]$ is the same as on the time interval $[t_k,t_{k+1}]$. This allows us to interpret \eqref{eqn:discreteImpact2pt2} as the conservation of discrete momentum projected onto $T^*\partial C$ during the impact. 
		 
%----------------------------------------------%
%----------------------------------------------%
	%SECTION 4: Dynamics of the Bouncing Ellipsoid
%----------------------------------------------%
%----------------------------------------------%

\section{Dynamics of the Bouncing Ellipsoid}
\label{sect:Dynamics}
	In this section, we derive the continuous equations of motion and the jump conditions for the bouncing ellipsoid. However, we first construct the Lagrangian based on the approach described in \cite{lee2007lie}.
	
	Consider the configuration space $Q = SE(3)$. Denote the set of body elements of a rigid-body by $\mathcal{B}$, namely the ellipsoid, and let $(\bm{x},R) \in SE(3)$ describe its configuration. Then, the inertial position of a body element of $\mathcal{B}$ is $\bm{x} + R \bm{\rho}$, where $\bm{\rho} \in \R^3$ is the position of the body element relative to the origin of the body-fixed frame. Thus, the kinetic energy is written as
	$
		T 
		= \frac{1}{2} \int_\BB \| \dot{\bm{x}} + \dot{R} \bm{\rho} \|^2 \, dm
		= \frac{1}{2} \int_\BB \| \dot{\bm{x}} \|^2 \, dm + \int_\BB \dot{\bm{x}}^T \dot{R} \bm{\rho}  \, dm + \frac{1}{2} \int_\BB \| \dot{R} \bm{\rho} \|^2 \, dm.		
	$
	Recall that $\int_\BB \bm{\rho} \, dm = 0$ since the origin of the body-fixed frame is the center of mass of the rigid-body. Expanding $ \| \dot{R} \bm{\rho} \|^2 = \tr[\dot{R} \bm{\rho} \bm{\rho}^T \dot{R}^T]$, we obtain $T =  \frac{1}{2} m \| \dot{\bm{x}} \|^2 + \frac{1}{2} \tr[\dot{R}J_d\dot{R}^T]$ where $m$ is the total mass and $J_d = \int_\mathcal{B} \bm{\rho} \bm{\rho}^T \, dm$ is a nonstandard moment of inertia matrix. By property \eqref{eqn:skew2} of the skew map, $J_d$ is related to the standard moment of inertia matrix $J = \int_\mathcal{B} S(\bm{\rho})^T S(\bm{\rho}) \, dm$ in the following ways: $J = \tr [J_d] I_3 - J_d$, $J_d = \frac{1}{2} \tr [J] I_3 - J$, and
	\begin{equation}
	\label{eqn:SJOmega}
		S(J\bm{\Omega}) = S(\bm{\Omega})J_d + J_dS(\bm{\Omega}),
	\end{equation}
	for any $\bm{\Omega} \in \R^3$ (see Appendix A.2 in \cite{lee2008thesis}). Since the ellipsoid is under the influence of a uniform, constant gravitational field in the $z$-direction, the potential energy is written as $U = mg\bm{e_3}^T\bm{x}$ where $g$ is the gravitational acceleration. Hence, the Lagrangian $L : TSE(3) \to \R$ is given by
	\begin{equation}
	\label{eqn:ellipsoidLagrangian}
		L(\bm{x},R,\dot{\bm{x}},\dot{R})
		= \frac{1}{2} m \| \dot{\bm{x}} \|^2 + \frac{1}{2} \tr[\dot{R}J_d\dot{R}^T] - mg\bm{e_3}^T\bm{x}. 
	\end{equation}
	This form of the Lagrangian is useful for computing the equations of motion and the jump conditions directly from Theorem \ref{thm:ContinuousThm}. However, a modified Lagrangian expressed in terms of identifications with the Lie algebra, which is a linear space, is useful for computing variations; this computation gives the equations of motion in \cite{lee2007lie}. The modified Lagrangian is introduced below for later use in Section~\ref{sect:LGVCI}.
	
	Recall that when $R \in SO(3)$, $\dot{R} = R S(\bm{\Omega})$ for some $\bm{\Omega} \in \R^3$ by left-trivialization. Since $SE(3) = \R^3 \rtimes SO(3)$ is defined by a semidirect product, we are careful with the identification on the tangent bundle $TSE(3) = SE(3) \times \mathfrak{se}(3)$ where $SE(3)$ is associated with $\R^3 \times SO(3)$ and $\mathfrak{se}(3)$ with $\R^3 \times \mathfrak{so}(3)$. Furthermore, by the skew map, $\mathfrak{so}(3) \simeq \R^3$. With these identifications, we can introduce a \textit{modified Lagrangian} $\tilde{L} : \R^3 \times SO(3) \times \R^3 \times \R^3 \to \R$. Denote the position, attitude, linear velocity, and angular velocity of the ellipsoid by $(\bm{x},R,\dot{\bm{x}},\bm{\Omega})$, respectively. Then, the modified Lagrangian is defined by
	\begin{subequations}
	\label{eqn:ellipsoidmLagrangian}
	\begin{align}
		\tilde{L} (\bm{x},R,\dot{\bm{x}},\bm{\Omega}) 
		&= \frac{1}{2}m \|\dot{\bm{x}}\|^2 + \frac{1}{2}\tr[S(\bm{\Omega}) J_d S(\bm{\Omega})^T] - mg \bm{e_3}^T \bm{x}	\label{eqn:ellipsoidmLagrangian1}\\
		&= \frac{1}{2}m \|\dot{\bm{x}}\|^2 + \frac{1}{2}\bm{\Omega}^T J \bm{\Omega} - mg \bm{e_3}^T \bm{x}, \label{eqn:ellipsoidmLagrangian2}
	\end{align}
	\end{subequations}
	see Section 2.3.2 in \cite{lee2008thesis}.
	
	Here, the admissible set is $C = \left\{ (\bm{x},R) \in SE(3) \mid \Phi(\bm{x},R) > 0 \right\} \subset Q$ where $\Phi$ is the collision detection function discussed in Section \ref{sect:ProblemDistanceForm}. Furthermore, this means that the boundary $\partial C = \left\{ (\bm{x},R) \in SE(3) \mid \Phi(\bm{x},R) = 0 \right\} = {\Phi}^{-1}(0)$ is the zero level set of the function $\Phi$, where $0$ is a regular value.
	%----------------------------------------------%
	% SECTION 4.1
	%----------------------------------------------%
	\subsection{Equations of Motion: Euler--Lagrange Equations}
	\label{sect:eqnOfMotions}
	%----------------------------------------------%
	We compute the equations of motion directly from the Euler--Lagrange equations in \eqref{eqn:extEL1}, which is $\left[ \frac{\partial L}{\partial q} - \frac{d}{dt} \left(\frac{\partial L}{\partial \dot{q}}\right) \right] \cdot \delta q = 0$ on the time interval $[t_0,t_f]$. We write $\delta q = (\delta \bm{x}, \delta R)$, where each term can be calculated individually. The variation on translation is simple: $\delta \bm{x}: [t_0,t_f] \to \R^3$, which vanishes at the initial time and final time. For the other, vary the rotation matrix in $SO(3)$ by expressing it as $R^{\epsilon} = Re^{\epsilon \eta}$ where $\epsilon \in \R$ and $\eta \in \mathfrak{so}(3)$. Then, the variation is given by $\delta R = \biggl.\frac{d}{d \epsilon} \Big\lvert_{\epsilon=0} R^\epsilon = R \eta$. 
		
	Now, using the Lagrangian in \eqref{eqn:ellipsoidLagrangian}, the partial derivatives are
	
	\begin{subequations}
	\label{eqn:Lpartials}
	\noindent
		\begin{minipage}{0.5\textwidth}
			\begin{equation}
				\frac{\partial L}{\partial q} 
				= \left( \frac{\partial L}{\partial \bm{x}}, \frac{\partial L}{\partial R} \right)
				= (-mg\bm{e_3},0),
			\label{eqn:Lpartials1} 
			\end{equation}
		\end{minipage}%
		\begin{minipage}{0.5\textwidth}
			\begin{equation}
				\frac{\partial L}{\partial \dot{q}} 
				= \left( \frac{\partial L}{\partial \dot{\bm{x}}}, \frac{\partial L}{\partial \dot{R}} \right)
				= (m\dot{\bm{x}},\dot{R}J_d),
			\label{eqn:Lpartials2} 
			\end{equation}
		\end{minipage}
		\vskip1em
	\end{subequations}
	\noindent where matrix derivatives are involved. Since $\delta q = (\delta \bm{x}, R\eta)$, the Euler--Lagrange equations become $\left[ (-mg\bm{e_3},0) - \frac{d}{dt}(m\dot{\bm{x}},\dot{R}J_d) \right]	\cdot (\delta \bm{x}, R\eta) = 0$, and so we have $-\delta \bm{x}^T \left\{ m \ddot{\bm{x}} + m g \bm{e_3} \right\} - \tr[\eta^T\left\{ R^T \ddot{R}J_d \right\}] = 0$ for all variations $\delta \bm{x}$ and $\eta$. For this to be true, the expressions in the curly brackets must vanish. Hence, the first equation of motion is $\ddot{\bm{x}} = -g \bm{e}_3$. The second expression vanishes if it is a symmetric matrix by Proposition~\ref{prop:traceProp}, so $\asym(R^T \ddot{R}J_d) = 0$. Note that 
	$
		R^T \ddot{R}J_d
			= R^T(\dot{R}S(\bm{\Omega}) + RS(\dot{\bm{\Omega}}))J_d
			= S(\bm{\Omega})^2 J_d + S(\dot{\bm{\Omega}})J_d,
	$
	where $\dot{R} = RS(\Omega)$. Then, $\asym(R^T \ddot{R}J_d) = 0$ expands into 
	\begin{align*}
		0
		&= \asym(S(\bm{\Omega})^2 J_d + S(\dot{\bm{\Omega}})J_d)
		= \left[S(\bm{\Omega})^2 J_d - J_d S(\bm{\Omega})^2\right] + \left[ S(\dot{\bm{\Omega}})J_d + J_d S(\dot{\bm{\Omega}}) \right]	\\
		&= S(\bm{\Omega} \times J \bm{\Omega}) +  S(J \dot{\bm{\Omega}})
		=  S(\bm{\Omega} \times J \bm{\Omega} + J \dot{\bm{\Omega}}),
	\end{align*}
	where \eqref{eqn:skew3} and \eqref{eqn:SJOmega} are used. Since the skew map is an isomorphism, $0 = J \dot{\bm{\Omega}}  + \bm{\Omega} \times J \bm{\Omega} $, we obtain the full set of \textit{continuous equations of motion in Lagrangian form} for the bouncing ellipsoid:
	\begin{myframedeq}
	\begin{equation}
	\label{eqn:eomLagrangian}
		\dot{\bm{v}} = - g \bm{e_3}, \quad
		J \dot{\bm{\Omega}} = J \bm{\Omega} \times \bm{\Omega}, \quad
		\dot{\bm{x}} = \bm{v}, \quad
		\dot{R} = R S(\bm{\Omega}),
	\end{equation}
	\end{myframedeq}
	\noindent where $\bm{v} \in \R^3$ is the translational velocity defined as $\bm{v} = \dot{\bm{x}}$. In particular, this describes the motion of the ellipsoid in the admissible set $C$.
	
		%----------------------------------------------%
		% SECTION 4.1.2
		%----------------------------------------------%
		\subsubsection{Hamilton's Equations}
		\label{sect:eqnOfMotionsM3}
		%----------------------------------------------%
		We consider the Legendre transformation for our modified Lagrangian $\F \tilde{L} : SE(3) \times \mathfrak{se}(3) \to SE(3) \times \mathfrak{se}^*(3)$, where $\mathfrak{se}^*(3)$ is identified with $\mathfrak{se}(3)$ by the Riesz representation. Using \eqref{eqn:LegendreTrans} and \eqref{eqn:ellipsoidmLagrangian1}, we have
		\[
			\F \tilde{L}(\bm{x},R,\dot{\bm{x}},S(\bm{\Omega})) \cdot_S (\delta \dot{\bm{x}},\eta)
			= (m\dot{\bm{x}},S(J \bm{\Omega})) \cdot_S (\delta \dot{\bm{x}},\eta),
		\]
		where the identity \eqref{eqn:SJOmega} was used. From the Legendre transform $\F \tilde{L} : (\bm{x},R,\dot{\bm{x}}, S(\bm{\Omega})) \mapsto (\bm{x},R,m\dot{\bm{x}},S(J \bm{\Omega}))$, the linear and angular momentum are written as $\bm{\gamma} = m \dot{\bm{x}}$ and $\bm{\Pi} = J\bm{\Omega}$, respectively. Hence, we arrive at the \textit{continuous equations of motion in Hamiltonian form}
		\begin{myframedeq}
		\begin{equation}
			\dot{\bm{\gamma}} = - mg \bm{e_3}, \quad
			\dot{\bm{\Pi}} = \bm{\Pi} \times \bm{\Omega}, \quad
			\dot{\bm{x}} = \dfrac{1}{m}\bm{\gamma}, \quad
			\dot{R} = R S(\bm{\Omega}).
		\end{equation}
		\end{myframedeq}
	%----------------------------------------------%
	% SECTION 4.2
	%----------------------------------------------%
	\subsection{Jump Conditions}
	\label{sect:jumpConds}
	%----------------------------------------------%
	We derive the jump conditions for our system using \eqref{eqn:impactConserved}. For convenience, let $q(t_i) = (\bm{x}_i,R_i) \in \partial C$, and let $(\dot{\bm{x}}_{i}^{\pm},\dot{R}_{i}^{\pm}) = \lim_{t \to t_i^{\pm}}(\dot{\bm{x}},\dot{R})$. The first jump condition \eqref{eqn:impactLinearMomentumConserved} gives 
	$
		\left( m (\dot{\bm{x}}_{i}^{+} - \dot{\bm{x}}_{i}^{-}), (\dot{R}_{i}^{+} - \dot{R}_{i}^{-})J_d \right) \cdot \delta q = 0,
	$
	for all $\delta q \in T_{q(t_i)}\partial C$. Now, write $\dot{R}_{i}^{\pm} = R_i S(\bm{\Omega}_{i}^{\pm})$ since $\lim_{t \to t_{i}^{\pm}} (\bm{x},R) = (\bm{x}_i, R_i)$, so 
	\begin{equation}
	\label{eqn:BEJumpCondition1}
		\left( m (\dot{\bm{x}}_{i}^{+} - \dot{\bm{x}}_{i}^{-}), R_i S(\bm{\Omega}_{i}^{+} - \bm{\Omega}_{i}^{-})J_d \right) \cdot \delta q = 0.
	\end{equation}
	One immediate solution to this condition is $(\dot{\bm{x}}_{i}^{+},\bm{\Omega}_{i}^{+})  = (\dot{\bm{x}}_{i}^{-},\bm{\Omega}_{i}^{-} )$. However, this implies that the system leaves the admissible set $C$, so we look for other solutions by considering the possible variations on the tangent space at the boundary point $q(t_i)$. In order to accomplish this, we consider a local representation of the boundary $\partial C = \Phi^{-1}(0)$, where $0$ is a regular value of the collision detection function $\Phi$. From the Submersion Level Set Theorem, 
	\begin{equation}
	\label{eqn:tangentSpacepartialC}
		T_{q(t_i)} \partial C = \left\{ \left( \delta \bm{x},  R_i \eta \right) \mathrel{\bigg|}  \left( \frac{\partial \Phi_i}{\partial \bm{x}}, \frac{\partial \Phi_i}{\partial R} \right) \cdot (\delta \bm{x},  R_i \eta) = 0  \right\},
	\end{equation}
	where
	$
		\left( \frac{\partial \Phi_i}{\partial \bm{x}}, \frac{\partial \Phi_i}{\partial R} \right) = \left.\left( \frac{\partial \Phi}{\partial \bm{x}}, \frac{\partial \Phi}{\partial R} \right) \right|_{q(t_i)}.
	$
	Then, $ \left( \frac{\partial \Phi_i}{\partial \bm{x}}, \frac{\partial \Phi_i}{\partial R} \right) \cdot (\delta \bm{x},  R_i \eta) = \delta \bm{x}^T \frac{\partial \Phi_i}{\partial \bm{x}}  + \tr\left[\eta^T R_i^T \frac{\partial \Phi_i}{\partial R}\right]$. Applying the argument used to obtain \eqref{eqn:trIPtoSkewIP}, we have
	$
		\delta \bm{x}^T \frac{\partial \Phi_i}{\partial \bm{x}}  + \frac{1}{2} \tr\left[\eta^T \asym\left( R_i^T \frac{\partial \Phi_i}{\partial R}  \right)\right] = 0,
	$
	for any $\left( \delta \bm{x},  R_i \eta \right) \in T_{q(t_i)} \partial C$.
	
	The tangent space can be further identified as a hyperplane in $\R^6$. This involves finding $\bm{\chi}_i \in \R^3$ such that $S(\bm{\chi}_i) = \asym\left(R_i^T \frac{\partial \Phi_i}{\partial R}\right)$. Suppose that $R_i^T\frac{\partial \Phi_i}{\partial R}$ is given, then $\bm{\chi}_i$ can be computed using the inverse of the skew map. It can also be defined using the rows of $R_i$ and $\frac{\partial \Phi_i}{\partial R}$: Let $\bm{r}_{i_1},\bm{r}_{i_2},\bm{r}_{i_3} \in S^2$ and $\bm{\phi}_{i_1},\bm{\phi}_{i_2},\bm{\phi}_{i_3} \in \R^3$ be the successive columns of the rotation matrix $R_i^T$ and $\frac{\partial \Phi_i}{\partial R}^T$, respectively. Then,
	\[
		S(\bm{\chi}_i)
		= R_i^T \frac{\partial \Phi_i}{\partial R} - \frac{\partial \Phi_i}{\partial R}^T R_i 
		= S(\bm{\phi}_{i_1} \times \bm{r}_{i_1} + \bm{\phi}_{i_2} \times \bm{r}_{i_2} + \bm{\phi}_{i_3} \times \bm{r}_{i_3}),
	\]
	where \eqref{eqn:skew3} is used. Since $S$ is an isomorphism,
	\begin{equation}
	\label{eqn:chi}
		\bm{\chi}_i = \bm{\phi}_{i_1} \times \bm{r}_{i_1} + \bm{\phi}_{i_2} \times \bm{r}_{i_2} + \bm{\phi}_{i_3} \times \bm{r}_{i_3}.
	\end{equation}
	\begin{lemma}
	\label{lem:tangentSpaceAsym2}
		Let $q(t_i) = (\bm{x}_i,R_i) \in \partial C$ and $\bm{\chi}_i$ be given by \eqref{eqn:chi}. Then 
		\begin{equation}
			T_{q(t_i)} \partial C = \left\{ (\delta \bm{x},R_i S(\bm{\zeta})) \mathrel{\Big|}  \delta \bm{x}^T  \frac{\partial \Phi_i}{\partial \bm{x}} + \bm{\zeta}^T \bm{\chi}_i = 0 \right \}.
		\end{equation}
	\end{lemma}
	
	\begin{proof}
		Express the elements of the tangent space as $(\delta \bm{x}, R_i \eta) = (\delta \bm{x}, R_i S(\bm{\zeta}))$ where $\bm{\zeta} \in \R^3$. Since $\asym\left( R_i^T \frac{\partial \Phi_i}{\partial R}  \right) = S(\bm{\chi}_i)$, we write $\frac{1}{2} \tr\left[\eta^T \asym\left( R_i^T \frac{\partial \Phi_i}{\partial R}  \right)\right] = \frac{1}{2}\tr[S(\bm{\zeta})^T S(\bm{\chi}_i)] = \bm{\zeta}^T \bm{\chi}_i$, where we used the induced inner product of $\R^3$ given in \eqref{eqn:lieinnerprod}. 
	\end{proof}
	
	\begin{remark} 
	\label{remark:hyperplane}
		$T_{q(t_i)} \partial C$ can be identified as a hyperplane $\left\{z \in \R^6 \mathrel{\Big|} \left( \tfrac{\partial \Phi_i}{\partial \bm{x}}^T \; \bm{\chi}_i^T \right) z = 0 \right\}$.
	\end{remark}
	\begin{theorem}
	\label{thm:jumpC1}
		Suppose $q(t_i) \in \partial C$. Then $\left( m (\dot{\bm{x}}_{i}^{+} - \dot{\bm{x}}_{i}^{-}), R_i S(\bm{\Omega}_{i}^{+} - \bm{\Omega}_{i}^{-})J_d \right) \cdot \delta q = 0$, for all $\delta q \in T_{q(t_i)} \partial C$, if and only if the first jump conditions, 
		\begin{subequations}
		\label{eqn:jumpC1}
			\begin{align}
				m(\dot{\bm{x}}_{i}^{+} - \dot{\bm{x}}_{i}^{-}) 
					&= \lambda  \frac{\partial \Phi_i}{\partial \bm{x}}	\label{eqn:jumpC11},	\\
				J(\bm{\Omega}_{i}^{+} - \bm{\Omega}_{i}^{-})
					&= \lambda \bm{\chi}_i	\label{eqn:jumpC12},
			\end{align}
		\end{subequations}
		are satisfied for some $\lambda \in \R \setminus \{0\}$.
	\end{theorem}

	\begin{proof}
	Let $\delta q =  (\delta \bm{x}, R_i S(\bm{\zeta}))$. We compute $\left( m (\dot{\bm{x}}_{i}^{+} - \dot{\bm{x}}_{i}^{-}), R_i S(\bm{\Omega}_{i}^{+} - \bm{\Omega}_{i}^{-})J_d \right) \cdot \delta q$, which becomes $\delta \bm{x}^T \left\{ m (\dot{\bm{x}}_{i}^{+} - \dot{\bm{x}}_{i}^{-}) \right\} +  \tr[S(\bm{\zeta})^T S(\bm{\Omega}_{i}^{+} - \bm{\Omega}_{i}^{-})J_d]$. The second term, following \eqref{eqn:trIPtoSkewIP}, turns into $\tfrac{1}{2}\tr[S(\bm{\zeta})^T \asym(S(\bm{\Omega}_{i}^{+} - \bm{\Omega}_{i}^{-})J_d)]$. Using \eqref{eqn:SJOmega}, $\asym(S(\bm{\Omega}_{i}^{+} - \bm{\Omega}_{i}^{-})J_d)) =  S(J(\bm{\Omega}_{i}^{+} - \bm{\Omega}_{i}^{-}))$, so the right-hand side becomes $\delta \bm{x}^T \left\{ m (\dot{\bm{x}}_{i}^{+} - \dot{\bm{x}}_{i}^{-}) \right\}  + \bm{\zeta}^T \left\{ J(\bm{\Omega}_{i}^{+} - \bm{\Omega}_{i}^{-}) \right\}$. If the jump conditions hold, the expression becomes $\lambda \left(\delta \bm{x}^T  \frac{\partial \Phi_i}{\partial \bm{x}} + \bm{\zeta}^T \bm{\chi}_i\right)$, which vanishes by Lemma \ref{lem:tangentSpaceAsym2}. If we assume that the LHS expression vanishes, the curly brackets must be a nonzero multiple of the normal vector $\left( \tfrac{\partial \Phi_i}{\partial \bm{x}}^T \; \bm{\chi}_i^T \right) \in \R^6$ by Remark \ref{remark:hyperplane}, which yield the jump conditions.
	\end{proof}
	
	We consider the second jump condition \eqref{eqn:impactEnergyConserved}, which is a statement of the conservation of energy. Recall that $E = \frac{\partial L}{\partial \dot{q}} \cdot \dot{q} - L$ and $\frac{\partial L}{\partial \dot{q}} = (m \dot{\bm{x}}, \dot{R}J_d)$, so $\frac{\partial L}{\partial \dot{q}} \cdot \dot{q} = (m \dot{\bm{x}}, \dot{R}J_d) \cdot (\dot{\bm{x}}, \dot{R}) = m\|\dot{\bm{x}}\|^2 + \tr[\dot{R}J_d\dot{R}^T]$. Therefore, the energy may be written as
	\begin{subequations}
	\label{eqn:contEnergy}
	\begin{align}
		E
		&= \frac{1}{2}m \|\dot{\bm{x}}\|^2 + \frac{1}{2}\tr[\dot{R} J_d \dot{R}^T] + mg\bm{e_3}^T\bm{x}, 	\label{eqn:contEnergy1}\\
		&= \frac{1}{2}m \|\dot{\bm{x}}\|^2 + \frac{1}{2}\tr[S(\bm{\Omega}) J_d S(\bm{\Omega})^T] + mg\bm{e_3}^T\bm{x}	\label{eqn:contEnergy2}	\\
		&= \frac{1}{2}m \|\dot{\bm{x}}\|^2 + \frac{1}{2}\bm{\Omega}^T J \bm{\Omega} + mg\bm{e_3}^T\bm{x}	.\label{eqn:contEnergy3}
	\end{align}
	\end{subequations}
	The second jump condition for our system is given by $0 = E(q(t_i^+),\dot{q}(t_i^+)) - E(q(t_i^-),\dot{q}(t_i^-))$. Using \eqref{eqn:contEnergy3}, the full set of jump conditions become 
	\begin{myframedeq} 
			$F_{\jump} : (\bm{x}_i,R_i,\dot{\bm{x}}_{i}^{-},\bm{\Omega}_{i}^{-}) \to (\lambda, \dot{\bm{x}}_{i}^{+},\bm{\Omega}_{i}^{+})$
	\begin{subequations}
	\label{eqn:jumpCondF}
		\begin{align}
			m\dot{\bm{x}}_{i}^{+}
				&= m \dot{\bm{x}}_{i}^{-}  + \lambda  \frac{\partial \Phi_i}{\partial \bm{x}}	\label{eqn:jumpCondF1},	\\
			J\bm{\Omega}_{i}^{+}
				&= J \bm{\Omega}_{i}^{-} + \lambda \bm{\chi}_i	\label{eqn:jumpCondF2},	\\
			0 
				&= \frac{1}{2}m(\|\dot{\bm{x}}_{i}^{+}\|^2 - \|\dot{\bm{x}}_{i}^{-}\|^2) + \frac{1}{2}\left( \bm{\Omega}_{i}^{+T} J \bm{\Omega}_{i}^{+} - \bm{\Omega}_{i}^{-T} J \bm{\Omega}_{i}^{-} \right).	\label{eqn:jumpCondF3}
		\end{align}
	\end{subequations}
	\end{myframedeq}
	Denote the solution to $(\dot{\bm{x}}_{i}^{+},\bm{\Omega}_{i}^{+})$ with $\lambda \neq 0$ as a discrete map $F_{\jump}$ with the necessary arguments above. In particular, $\lambda$ is obtained by substituting $(\dot{\bm{x}}_{i}^{+},\bm{\Omega}_{i}^{+})$ from \eqref{eqn:jumpCondF1} and \eqref{eqn:jumpCondF2} into \eqref{eqn:jumpCondF3}, which gives a quadratic equation for the variable $\lambda$. One root will always be $\lambda = 0$, which is omitted. Then, there is a unique nonzero root $\lambda$, which gives $(\dot{\bm{x}}_{i}^{+},\bm{\Omega}_{i}^{+})$.
	
		%----------------------------------------------%
		% SECTION 4.2.1
		%----------------------------------------------%
		\subsubsection{Jump Conditions: Hamiltonian Form}
		\label{sect:JumpCondHamiltonian}
		%----------------------------------------------%
		It is actually natural to express the jump conditions on the Hamiltonian side since both conservation of energy and conservation of momentum can easily be described on the cotangent bundle. We still write $(\bm{x}_i,R_i) \in \partial C$ as the configuration at impact. Denote the instantaneous linear and angular momentum before and after impact as $\bm{\gamma}^{\pm} = m \dot{\bm{x}}^{\pm}$ and $\bm{\Pi}^{\pm} = J\bm{\Omega}^{\pm}$, respectively. Suppose that $J$ is invertible, then we obtain the following result.
		\begin{corollary}
		\label{cor:jumpCondsH}
			Let $(\bm{x}_i,R_i) \in \partial C$ and $(\bm{\gamma}^{-},\bm{\Pi}^{-})$ be given; let $\bm{\chi}_i$ be defined by \eqref{eqn:chi}. There is a unique $\lambda \in \R \setminus \{0\}$ and $(\bm{\gamma}^{+},\bm{\Pi}^{+})$ satisfying
			\pagebreak
			\begin{myframedeq}
				$\tilde{F}_{\jump} : (\bm{x}_i,R_i,\bm{\gamma}^{-},\bm{\Pi}^{-}) \to (\lambda,\bm{\gamma}^{+},\bm{\Pi}^{+})$
			\begin{subequations}
			\label{eqn:jumpCondsHamiltonian}
				\begin{align}
					\bm{\gamma}^{+}
						&= \bm{\gamma}^{-} + \lambda \frac{\partial \Phi_i}{\partial \bm{x}}, \label{eqn:jumpCondsHamiltonian1}	\\
					\bm{\Pi}^{+} 
						&= \bm{\Pi}^{-} + \lambda \bm{\chi}_i, \label{eqn:jumpCondsHamiltonian2}	\\
					0
						&= \frac{1}{2m}(\| \bm{\gamma}^{+} \|^2 - \| \bm{\gamma}^{-} \|^2) + \frac{1}{2}({\bm{\Pi}^{+}}^{T}J^{-1}\bm{\Pi}^{+} - {\bm{\Pi}^{-}}^{T}J^{-1}\bm{\Pi}^{-}).\label{eqn:jumpCondsHamiltonian3}
				\end{align}
			\end{subequations}
			\end{myframedeq}
		\end{corollary}

%----------------------------------------------%
%----------------------------------------------%
	%SECTION 5: Lie Group Variational  Collision Integrator
%----------------------------------------------%
%----------------------------------------------%

\section{Lie Group Variational Collision Integrators for the Bouncing Ellipsoid}
\label{sect:LGVCI}
	For the discrete setting, we also follow the approach described in \cite{lee2007lie}. This involves the construction of the discrete Lagrangian by approximating a segment of the action integral via the trapezoidal rule. However, we first approximate the linear and angular velocity for a segment of the action integral. We introduce the auxiliary variable $F_k \in SO(3)$ so that $R_{k+1} = R_k F_k$. Note that $F_k$ represents the relative attitude between $R_k$ and $R_{k+1}$, and it is guaranteed that the attitude evolves on $SO(3)$ since $F_k \in SO(3)$. Now, using $\dot{R} = RS(\bm{\Omega})$, approximate the $k$-th angular velocity as
	\begin{equation}
	\label{eqn:Fk}
		S(\bm{\Omega}_k) = R_k^T \dot{R}_k \approx R_k^T \frac{R_{k+1} - R_k}{h} = \frac{1}{h}(F_k - I_3).
	\end{equation}
	The linear velocity $\dot{\bm{x}}_k$ is approximated by $(\bm{x}_{k+1} - \bm{x}_{k})/h$. Substitute the approximations above into the modified Lagrangian \eqref{eqn:ellipsoidmLagrangian2}, so the approximation of the kinetic term becomes 
	\[
		T(\dot{\bm{x}}_k,S(\bm{\Omega}_k)) 
			\approx T((\bm{x}_{k+1} - \bm{x}_{k})/h, (F_k - I_3)/h) 
			= \frac{1}{2h^2}\| \bm{x}_{k+1} - \bm{x}_{k} \|^2 + \frac{1}{h^2}\tr\left[(I_3 - F_k)J_d \right].
	\]
	We obtain the discrete Lagrangian by approximating the action integral using the trapezoidal rule:
	\begin{equation}
	\label{eqn:discreteL}
			L_d (\bm{x}_k,R_k,\bm{x}_{k+1},F_k)
			= \frac{1}{2h}m\| \bm{x}_{k+1} - \bm{x}_{k} \|^2 + \frac{1}{h}\tr[(I_3 - F_k)J_d]
			-\frac{1}{2}mgh\bm{e_3}^T(\bm{x}_{k+1} + \bm{x}_{k}).
	\end{equation}
	%----------------------------------------------%
	%Section 5.1
	%----------------------------------------------%
	\subsection{Discrete Equations of Motion}
	\label{sect:discreteEqofM}
	%----------------------------------------------%
	
	We will obtain the discrete equations of motion away from the point of impact directly from the first result \eqref{eqn:discreteEL} of Theorem \ref{thm:discreteThm}. More specifically, the result is obtained by taking the variations of the discrete variables on the discrete action sum and applying the discrete Hamilton's principle.
	
		%----------------------------------------------%
		%Section 5.1.1
		%----------------------------------------------%
		\subsubsection{Lagrangian Form}
		\label{sect:discreteEqofMLF}
		%----------------------------------------------%
	Let $q_k = (\bm{x}_k,R_k)$ for $k \in \{0, \ldots, i-2, i+1, \ldots, N\}$. Consider the following variations of the discrete variables. Namely, the variation $\delta \bm{x}_k \in \R^3$ of $\bm{x}_k$ which vanishes at $k =0$ and $k=N$; the variation of $R_k$ is given by $\delta R_k = R_k \eta_k$ where $\eta_k \in \mathfrak{so}(3)$ and also vanishes at $k=0$ and $k=N$.

	Recall that $F_k = R_k^TR_{k+1}$ and write $\delta q_k = (\delta \bm{x}_k, R_k \eta_k)$. The following identities are derived using \eqref{eqn:trIPtoSkewIP} and matrix derivatives: 
	\begin{align}
		D_2 L_d(q_{k-1},q_k,h)\cdot \delta q_k
		&= \left( \tfrac{1}{h}m(\bm{x}_k - \bm{x}_{k-1}) - \tfrac{1}{2} m g h \bm{e_3}, - \tfrac{1}{h}\asym(F_{k-1}^T J_d) \right) \cdot_S (\delta \bm{x}_k,\eta_k),	\label{eqn:discreteLdPartial1} \\
		D_1 L_d(q_k,q_{k+1},h)\cdot \delta q_k
		&= \left( \tfrac{1}{h}m(\bm{x}_k - \bm{x}_{k+1}) - \tfrac{1}{2} m g h \bm{e_3}, - \tfrac{1}{h}\asym(F_k J_d) \right) \cdot_S (\delta \bm{x}_k,\eta_k). \label{eqn:discreteLdPartial2}
	\end{align}
	Note that $-\asym(F_{k-1}^TJ_d) = \asym(J_dF_{k-1})$. Using the discrete Euler--Lagrange equation \eqref{eqn:discreteEL}, their sum vanishes
 	\[
		0 = [D_2L_d(q_{k-1},q_k,h) + D_1L_d(q_{k},q_{k+1},h)] \cdot \delta q_k,
	\]
	giving $0 = -\delta\bm{x}^T\left\{ \tfrac{1}{h}m(\bm{x}_{k+1} - 2 \bm{x}_{k} + \bm{x}_{k-1}) + mgh\bm{e_3} \right\} + \frac{1}{2}\tr \left[ \eta_k^T\left\{ \tfrac{1}{h}\asym(J_d F_{k-1} - F_k J_d) \right\} \right]$. This equation holds for all $k = 1, \ldots, i-2, i+1, \ldots, N$, if and only if the expressions in the curly brackets vanish. We arrive at the \textit{discrete equations of motion in Lagrangian form}, denoted as $FL_d$ in Table \ref{tab:LGVCI_F_Ld}.
	
	The \textit{discrete Lagrangian map} $FL_d$ has the following parameters $[h,h]$: The first $h$ indicates the timestep for the interval $[t_{k-1},t_k]$ with the corresponding configurations $(\bm{x}_{k-1},R_{k-1})$ at time $t_{k-1}$ and $(\bm{x}_k,R_k)$ at time $t_k$; the second $h$ indicates the timestep for $[t_k,t_{k+1}]$ with its corresponding configurations $(\bm{x}_k,R_k)$ at $t_k$ and $(\bm{x}_{k+1},R_{k+1})$ at $t_{k+1}$.  To compute the map, solve for the $\bm{x}_{k+1}$ first. $F_k$ is obtained next using the second, implicit equation where $F_{k-1} = R_{k-1}^TR_k$ (see Section 3.4 in \cite{Lee2007orbital}). In fact, similar implicit equations show up in all of the discrete maps that follow. Finally, $R_{k+1}$ is updated after obtaining $F_k$.
		%----------------------------------------------%
		%Section 5.1.2
		%----------------------------------------------%
		\subsubsection{Hamiltonian Form}
		\label{sect:discreteEqofMHF}
		%----------------------------------------------%
	Using the discrete Legendre transforms, we will arrive at a set of discrete equations of motion based on discrete positions and momenta, which we will collectively call \textit{states}. Denote the linear and angular momentum by $p_k = (\bm{\gamma}_k,S(\bm{\Pi}_k))$. Recall Theorem \ref{thm:discreteThm} and \eqref{eqn:discreteLegendreTrans} for the discrete Legendre transforms:
	\begin{align}
		(\bm{\gamma}_k,S(\bm{\Pi}_k)) \cdot_S (\delta \bm{x}_k,\eta_k) 
		&= -D_1 L_d(q_k,q_{k+1},h) \cdot \delta q_k, \label{eqn:pkEllipsoid} \\
		(\bm{\gamma}_{k+1},S(\bm{\Pi}_{k+1})) \cdot_S (\delta \bm{x}_{k+1},\eta_{k+1}) 
		&= D_2 L_d(q_k,q_{k+1},h) \cdot \delta q_{k+1}. \label{eqn:pkp1Ellipsoid}
	\end{align}
	\eqref{eqn:pkEllipsoid} gives $(\bm{\gamma}_k,S(\bm{\Pi}_k)) = \left( \tfrac{1}{h}m(\bm{x}_{k+1} - \bm{x}_{k}) + \tfrac{h}{2}mg\bm{e_3},\tfrac{1}{h}\asym(F_{k} J_d) \right)$, and \eqref{eqn:pkp1Ellipsoid} produces $(\bm{\gamma}_{k+1},S(\bm{\Pi}_{k+1})) = \left( \tfrac{1}{h}m(\bm{x}_{k+1} - \bm{x}_{k}) - \tfrac{h}{2}mg\bm{e_3},\tfrac{1}{h}\asym(J_d F_{k}) \right)$. Hence, $\bm{x}_{k+1}$ and $\bm{\gamma}_{k+1}$ can be expressed in terms of $\bm{x}_k$ and $\bm{\gamma}_{k}$ from the equations arising from the first components of \eqref{eqn:pkEllipsoid} and \eqref{eqn:pkp1Ellipsoid}, respectively. Since $S(\bm{\Pi}_k) = \frac{1}{h}\asym(F_{k} J_d)$,
	\[
		S(\bm{\Pi}_{k+1}) 
		= \tfrac{1}{h}\asym(J_d F_{k})
		= F_{k}^T \tfrac{1}{h}\asym(F_{k} J_d) F_{k}
		= F_{k}^T S(\bm{\Pi}_k) F_{k}
		= S(F_k^T \bm{\Pi}_k),
	\]
	using property \eqref{eqn:skew4} of the skew map. As a result, the \textit{discrete equations of motion in Hamiltonian form} are given as $\tilde{F}_{L_d}$ in Table \ref{tab:LGVCI_F_tilde_Ld}.
	The \textit{discrete Hamiltonian map} $\tilde{F}_{L_d}$ has the following parameter $[h]$: This $h$ indicates the timestep for the interval $[t_{k},t_{k+1}]$ with the corresponding states $(\bm{x}_k,R_k,\bm{\gamma}_k,\bm{\Pi}_k)$ at time $t_k$ and $(\bm{x}_{k+1},R_{k+1},\bm{\gamma}_{k+1},\bm{\Pi}_{k+1})$ at time $t_{k+1}$. Similar to the Lagrangian form, $\bm{x}_{k+1}$ and $\bm{\gamma}_{k+1}$ can be computed first. Compute $F_k$ from the third, implicit equation, which is used to update $\bm{\Pi}_{k+1}$ and $R_{k+1}$.
	%----------------------------------------------%
	%Section 5.2
	%----------------------------------------------%
	\subsection{Impact Point and Time}
	\label{sect:discreteImpactPointTime}
	%----------------------------------------------%
	Recall the definition of the collision detection function $\Phi$ from \eqref{eqn:distBouncingEllipsoid}, which allows us to detect collisions in the system. For each integration step discussed in Section \ref{sect:discreteEqofM}, $(\bm{x}_{k+1},R_{k+1})$ are computed. Hence, one may check for interpenetration after each integration by evaluating $\Phi(\bm{x}_{k+1},R_{k+1})$. 
	
	If the signed distance is positive, then we proceed to the next integration step. If the signed distance is zero, then the current configuration and time is the impact point and time, and we will have to apply the discrete jump conditions.
	If the evaluation is negative, then interpenetration has occurred, the current integration step is discarded, and so we consider \eqref{eqn:discreteImpact1} of Theorem \ref{thm:discreteThm}, and we attempt to resolve the impact point and time.
	
		%----------------------------------------------%
		%Section 5.2.1
		%----------------------------------------------%
		\subsubsection{Lagrangian Form}
		\label{sect:discreteImpactLF}
		%----------------------------------------------%
	Note that the impact point would occur at time $\tilde{t} = t_{i-1} + \alpha h$ for some $\alpha \in (0,1)$. Similarly, we rewrite \eqref{eqn:discreteImpact1}
	\[
		0 = [D_2 L_d(q_{i-2},q_{i-1},h) + D_1 L_d(q_{i-1},\tilde{q},\alpha h)]\cdot \delta q_{i-1}
	\]
	as the discrete map ${F_{L_d}^\text{Imp}}[h,\alpha h] $ in Table \ref{tab:LGVCI_F_Ld}.
	We compute the solution using the bisection method to solve for $\alpha \in (0,1)$. There is a unique $\alpha$ in the open interval such that $\Phi(\tilde{\bm{x}},\tilde{R}) = 0$. This is the case because $\Phi$ is defined to be positive when the ellipsoid is above the plane and negative when the ellipsoid is interpenetrating or below the plane. Once $\alpha$ is solved, $(\tilde{\bm{x}},\tilde{R})$ and $F_{i-1}$ can be computed following the steps in $FL_d[h,h]$.
	
		%----------------------------------------------%
		%Section 5.2.2
		%----------------------------------------------%
		\subsubsection{Hamiltonian Form}
		\label{sect:discreteImpactHF}
		%----------------------------------------------%
	On the Hamiltonian side, we compute the following discrete Legendre transforms,
	\begin{align*}
		(\bm{\gamma}_{i-1},S(\bm{\Pi}_{i-1})) \cdot_S (\delta \bm{x}_{i-1},\eta_{i-1}) 
		&= -D_1 L_d(q_{i-1},\tilde{q},\alpha h) \cdot \delta q_{i-1},	\\
		(\tilde{\bm{\gamma}},S(\tilde{\bm{\Pi}})) \cdot_S (\delta \tilde{\bm{x}},\tilde{\eta}) 
		&= D_2 L_d(q_{i-1},\tilde{q},\alpha h) \cdot \delta \tilde{q}.
	\end{align*}
	We obtain the equations for the impact point in Hamiltonian form, $\tilde{F}_{L_d}^\text{Imp}[\alpha h]$ in Table \ref{tab:LGVCI_F_tilde_Ld}. The $\alpha \in (0,1)$ and $(\tilde{\bm{x}},\tilde{R},\tilde{\bm{\gamma}},\tilde{\bm{\Pi}})$ are solved for using the bisection method, as before.
	
	%----------------------------------------------%
	%Section 5.3
	%----------------------------------------------%
	\subsection{Single Impact}
	\label{sect:SingleImpact}
	%----------------------------------------------%	
	We consider the next integration step, which would give us the next discrete configuration after the impact configuration. In this subsection, assume that there is one collision in the time interval $(t_{i-1},t_i)$ occurring at time $\tilde{t} = t_{i-1} + \alpha h$. Therefore, our next discrete configuration occurs at time $t_i = \tilde{t} + (1-\alpha)h$. 
	
		%----------------------------------------------%
		%Section 5.3.1
		%----------------------------------------------%
		\subsubsection{Lagrangian Form}
		\label{sect:discreteConfigiLF}
		%----------------------------------------------%
	The equation \eqref{eqn:discreteImpact2} of Theorem \ref{thm:discreteThm} is used for our next integration step. Recall the discrete energy from \eqref{eqn:discreteEnergy}, which gives
	\begin{equation}
		E_d(q_k,q_{k+1},h)
		= \frac{1}{2h^2}m\| \bm{x}_{k+1} - \bm{x}_k \|^2 + \frac{1}{h^2}\tr[(I_3 - F_k)J_d] + \frac{1}{2}mg\bm{e_3}^T(\bm{x}_{k+1} + \bm{x}_k),
	\end{equation}
	and so \eqref{eqn:discreteImpact2pt1} can be easily written. In addition, the conservation of discrete momentum described with the cotangent lift in \eqref{eqn:discreteImpact2pt2} is written in terms of a local representation of the contact set $\partial C = \Phi^{-1}(0)$. This is the same constraint formulation as in the continuous case, so we write our next set of equations after computing
	\begin{equation*}
		0 = i^*\left(D_2 L_d(q_{i-1},\tilde{q},\alpha h) + D_1 L_d(\tilde{q},q_i, (1-\alpha)h)\right) \cdot \delta \tilde{q}.
	\end{equation*}

	On the Lagrangian side, we obtain $F_{L_d}^i[\alpha h, (1-\alpha)h][\lambda] $ in Table \ref{tab:LGVCI_F_Ld}. Observe that $\lambda \neq 0$,  and the partial derivatives are
	\begin{equation}
		\left( \frac{\partial \tilde{\Phi}}{\partial \bm{x}}, \frac{\partial \tilde{\Phi}}{\partial R} \right) 
		= \left( \frac{\partial \Phi}{\partial \bm{x}}, \biggl.\frac{\partial \Phi}{\partial R} \right) \biggr \lvert_{(\tilde{\bm{x}},\tilde{R})} 
		=  \left( \bm{n}, -\frac{\bm{n}\bm{n}^T \tilde{R} I_\epsilon^2}{\| I_\epsilon \tilde{R}^T \bm{n} \|} \right).
	\label{eqn:ellipsoidpartials}
	\end{equation}
	Furthermore, $[\lambda]$ indicates the requirement to compute it first. In fact, this is the same $\lambda$ as in Theorem \ref{thm:jumpC1} in the continuous case. However, note that determining $\lambda$ on the Lagrangian side can be difficult since we do not have information on the instantaneous linear and angular velocity $(\dot{\bm{x}}^{-}, \bm{\Omega}^{-})$ before the impact. Therefore, the solution of $\lambda$ is discussed on the Hamiltonian side, which follows easily from Corollary \ref{cor:jumpCondsH}. Once $\lambda$ is determined, $(\bm{x}_{i},R_{i})$ can be solved similarly to the previous discrete Lagrangian maps.
	
	Finally, we write the last set of equations from \eqref{eqn:discreteImpact3} by first computing 
	\begin{equation*}
		0 = [D_2L_d(\tilde{q},q_i,(1-\alpha)h) + D_1L_d(q_i,q_{i+1},h)] \cdot \delta q_i.
	\end{equation*}
	In Lagrangian form, we get $F_{L_d}^{i+1}[(1-\alpha)h,h]$ in Table \ref{tab:LGVCI_F_Ld}.
	
		%----------------------------------------------%
		%Section 5.3.2
		%----------------------------------------------%
		\subsubsection{Hamiltonian Form}
		\label{sect:discreteConfigiHF}
		%----------------------------------------------%
	Again, we compute the discrete Legendre transforms,
	\begin{align*}
		(\tilde{\bm{\gamma}},S(\tilde{\bm{\Pi}})) \cdot_S (\delta \tilde{\bm{x}},\tilde{\eta}) 
			&= -D_1 L_d(\tilde{q},q_i,(1-\alpha) h) \cdot \delta \tilde{q},	\\
		(\bm{\gamma}_i,S(\bm{\Pi}_i)) \cdot_S (\delta \bm{x}_i, \eta_i) 
			&= D_2 L_d(\tilde{q},q_i,(1-\alpha) h)) \cdot \delta q_i.
	\end{align*}
	We obtain the equations in Hamiltonian form as $\tilde{F}_{L_d}^\text{Imp}[\alpha h]$ in Table \ref{tab:LGVCI_F_tilde_Ld}. Note that 
		$
			S(\tilde{\bm{\chi}}) = \asym\left(\tilde{R}^T \frac{\partial \tilde{\Phi}}{\partial R}\right),
		$
	where $\tilde{\bm{\chi}}$ can be computed by either using $S^{-1}$ or \eqref{eqn:chi}. Similarly, $[\lambda]$ indicates that it needs to be solved first; to solve for $\lambda$, invoke Corollary \ref{cor:jumpCondsH} by setting $(\tilde{\bm{x}},\tilde{R}) \in \partial C$ as the configuration at impact and letting $(\tilde{\bm{\gamma}},\tilde{\bm{\Pi}})$ = $(\bm{\gamma}^{-},\bm{\Pi}^{-})$. In fact, this will not only solve for $\lambda \neq 0$ but also $(\bm{\gamma}^{+},\bm{\Pi}^{+})$, and this fact will be used to optimize our algorithm in the end of this section.
	
	Lastly, we solve for the next set of states at time $t_{i+1}$ by computing the next set of discrete Legendre transforms,
	\begin{align*}
		(\bm{\gamma}_{i},S(\bm{\Pi}_{i}) \cdot_S (\delta \tilde{\bm{x}},\tilde{\eta}) 
			&= -D_1 L_d(q_i, q_{i+1}, h) \cdot \delta q_i,	\\
		(\bm{\gamma}_{i+1},S(\bm{\Pi}_{i+1})) \cdot_S (\delta \bm{x}_i, \eta_i) 
			&= D_2 L_d(q_i, q_{i+1}, h)) \cdot \delta q_{i+1}.
	\end{align*}
	However, this yields the same discrete Hamiltonian map $\tilde{F}_{L_{d}}[h]$, which is unsurprising because the discrete flow from $(\bm{x}_i,R_{i},\bm{\gamma}_{i},\bm{\Pi}_{i}) \mapsto (\bm{x}_{i+1},R_{i+1},\bm{\gamma}_{i+1},\bm{\Pi}_{i+1})$ on the time interval $[t_i,t_{i+1}]$ is given by $\tilde{F}_{L_{d}}[h]$.
	
	%----------------------------------------------%
	%Section 5.4
	%----------------------------------------------%
	\subsection{Multiple Impacts}
	\label{sect:MultipleImpacts}
	%----------------------------------------------%		
	Suppose that multiple impacts occur in the interval $(t_{i-1},t_i)$. For concreteness, we assume that there are $l$ impacts. From Section \ref{sect:discreteImpactPointTime}, we determined that the first impact occurs at $\tilde{t} = t_{i-1} + \alpha h$ where $\alpha \in (0,1)$. Let $\alpha_1 = \alpha$ and introduce 
	\[
		\alpha_k \in (0, 1 - \alpha_{\Sigma_{k}}),
		\qquad 
		\alpha_{\Sigma_{k}} = \sum_{j=1}^{k} \alpha_j,
	\]
	where $k = 1,2, \ldots, l$. Then, denote the configurations of impact by $\tilde{q}_k= (\tilde{\bm{x}}_k, \tilde{R}_k)$ which occurs at the time $\tilde{t}_k = t_{i-1} + \alpha_{\Sigma_{k}}h$ for each assumed collision; we also write $\tilde{R}_{k+1} = \tilde{R}_k \tilde{F}_k$. In addition, $\Phi(\tilde{\bm{x}}_k, \tilde{R}_k) = 0$.
	
		%----------------------------------------------%
		%Section 5.4.1
		%----------------------------------------------%
		\subsubsection{Lagrangian Form}
		\label{sect:discreteConfigImpactPLF}
		%----------------------------------------------%	
	For our next integration step, we combine the conservation of discrete energies and 
	\begin{equation*}
		0 = [D_2L_d(q_{i-1},\tilde{q}_1,\alpha_1 h) + D_1L_d(\tilde{q}_1,\tilde{q}_2,\alpha_2 h)] \cdot \delta \tilde{q}_1,
	\end{equation*}
	so we arrive at $F_{L_d}^{\text{Imp} +}[\lambda][\alpha_1 h, \alpha_2 h]$ in Table \ref{tab:LGVCI_F_Ld}. Observe that $\lambda \neq 0$, and the partial derivatives are computed similarly to \eqref{eqn:ellipsoidpartials}, but is evaluated at $(\tilde{\bm{x}}_{2},\tilde{R}_{2})$. We solve for the next discrete configuration $(\tilde{\bm{x}}_2, \tilde{R}_2)$ by solving for $\lambda$ first in the same way as $F_{L_d}^i[\alpha h, (1- \alpha)h]$. Next, $\alpha_2 \in (0,1-\alpha_{\Sigma_1})$ is determined using the bisection method and using $\Phi(\tilde{\bm{x}}_2, \tilde{R}_2)$ for the stopping criteria. Given both $\lambda$ and $\alpha_2$, proceed to solve for $(\tilde{\bm{x}}_{2},\tilde{R}_{2})$ using the middle three equations.
	
	In general, we use the same map $F_{L_d}^{\text{Imp} +}[\alpha_k h, \alpha_{k+1} h][\lambda]$ to find the subsequent impact configurations for $k = 2, \ldots, l-1$. Once we have determined the last collision, we can find the configuration $(\bm{x}_i,R_i)$ using $F_{L_d}^i [\alpha_l h, (1-\alpha_{\Sigma_{l}}) h][\lambda]$ from Section \ref{sect:SingleImpact}. Finally, we use $F_{L_d}^{i+1} [(1-\alpha_{\Sigma_{l}}) h, h]$ to determine the configuration $(\bm{x}_{i+1},R_{i+1})$.
	
		%----------------------------------------------%
		%Section 5.4.2
		%----------------------------------------------%
		\subsubsection{Hamiltonian Form}
		\label{sect:discreteConfigImpactPHF}
		%----------------------------------------------%
	Denote the linear and angular momentum for the configurations at impact by $(\tilde{\bm{\gamma}}_k,\tilde{\bm{\Pi}}_k)$, respectively, where $1 \leq k \leq l$. We compute the Legendre transforms,
	\begin{align*}
		(\tilde{\bm{\gamma}}_1,S(\tilde{\bm{\Pi}}_1)) \cdot_S(\delta \tilde{\bm{x}}_1,\tilde{\eta}_1 ) &= -D_1L_d(\tilde{q}_1,\tilde{q}_2,\alpha_2 h) \cdot \delta \tilde{q}_1,	\\
		(\tilde{\bm{\gamma}}_2,S(\tilde{\bm{\Pi}}_2)) \cdot_S(\delta \tilde{\bm{x}}_2,\tilde{\eta}_2 ) &= D_2L_d(\tilde{q}_1,\tilde{q}_2,\alpha_2 h) \cdot \delta \tilde{q}_2.
	\end{align*}
	Therefore, we have $\tilde{F}_{L_d}^{\text{Imp+}}[\lambda][\alpha_2 h]$ in Table \ref{tab:LGVCI_F_tilde_Ld}. Note that 
	$
		S(\tilde{\bm{\chi}}_1) = \asym \left( \tilde{R}_1^T\frac{\partial \tilde{\Phi}_1}{\partial R}\right).
	$
	Again, $\lambda$ is solved for first by using $\tilde{F}_{\jump}$ from Corollary \ref{cor:jumpCondsH}, and $\alpha_2 \in (0,1-\alpha_{\Sigma_1})$ is subsequently determined using the bisection method. As a result, $(\tilde{\bm{x}}_2,\tilde{R}_2,\tilde{\bm{\gamma}}_2,\tilde{\bm{\Pi}}_2)$ can be calculated using the middle five equations.
	
	Now, we use $\tilde{F}_{L_d}^{\text{Imp+}}[\lambda][\alpha_k h]$ to determine the next set of impact states at time $\tilde{t}_k$ for $k = 3, \ldots, l$. After determining the previous collision states, we can find the next states $(\bm{x}_i,R_{i},\bm{\gamma}_{i},\bm{\Pi}_{i})$ using $\tilde{F}_{L_d}^i[(1-\alpha_{\Sigma_l})h][\lambda]$. Lastly, the discrete Hamiltonian map $\tilde{F}_{L_d}[h]$ is used to determine $(\bm{x}_{i+1},R_{i+1},\bm{\gamma}_{i+1},\bm{\Pi}_{i+1})$ away from the last point of impact.
	
	\begin{sidewaystable}
		\begin{subtable}[t]{0.515 \textwidth}
		\captionsetup{singlelinecheck=false}
		\caption{\textbf{Lagrangian Forms}}
		\label{tab:LGVCI_F_Ld}
		\centering
		
			\begin{myframedeq}
			$F_{L_d}[h,h] : (\bm{x}_{k-1},R_{k-1},\bm{x}_k,R_k) \mapsto (\bm{x}_k,R_k,\bm{x}_{k+1},R_{k+1})$
			\begin{subequations}
			\label{eqn:discreteLM0}
				\begin{align*}
						0 & =\tfrac{m}{h}(\bm{x}_{k+1} - 2 \bm{x}_{k} + \bm{x}_{k-1}) + mgh\bm{e_3},\\
						0 & = \tfrac{1}{h} \asym(J_d F_{k-1} - F_k J_d), 
							\qquad  R_{k+1}  = R_k F_k.
				\end{align*}
			\end{subequations}
			\end{myframedeq}
			
			\begin{myframedeq}
			${F_{L_d}^\text{Imp}}[h,\alpha h] : (\bm{x}_{i-2},R_{i-2},\bm{x}_{i-1},R_{i-1}) \mapsto (\bm{x}_{i-1},R_{i-1},\tilde{\bm{x}},\tilde{R})$
			\begin{subequations}
			\label{eqn:discreteLMeq1}
				\begin{align*}
					0 &= \tfrac{m}{h}(\bm{x}_{i-1} - \bm{x}_{i-2}) - \tfrac{m}{\alpha h}(\tilde{\bm{x}} - \bm{x}_{i-1}) - \tfrac{mg(1+\alpha)h}{2}\bm{e_3}, \\
					0 &= \asym\left( \tfrac{1}{h} J_d F_{i-2} - \tfrac{1}{\alpha h}F_{i-1} J_d \right), \\
					 \tilde{R} &= R_{i-1} F_{i-1},
							\qquad 0 = \Phi(\tilde{\bm{x}},\tilde{R}).
				\end{align*}
			\end{subequations}
			\end{myframedeq}
			
			\begin{myframedeq}
			$F_{L_d}^i[\alpha h, (1-\alpha)h][\lambda] : (\bm{x}_{i-1},R_{i-1},\tilde{\bm{x}},\tilde{R}) \mapsto (\tilde{\bm{x}},\tilde{R},\bm{x}_{i},R_{i})$
			\begin{subequations}
			\label{eqn:discreteLM2s}
				\begin{align*}
					0 &= E_d(\tilde{q},q_i,(1-\alpha)h) - E_d(q_{i-1},\tilde{q},\alpha h),  \\
					0 &= \tfrac{m}{\alpha h}(\tilde{\bm{x}} - \bm{x}_{i-1}) - \tfrac{m}{(1-\alpha) h}(\bm{x}_i - \tilde{\bm{x}}) - \tfrac{mgh}{2}\bm{e_3} + \lambda \tfrac{\partial \tilde{\Phi}}{\partial \bm{x}}, \\
					0 &= \asym\left( \tfrac{1}{\alpha h} J_d F_{i-1} - \tfrac{1}{(1-\alpha) h}\tilde{F}J_d\right) + \lambda \asym\left(\tilde{R}^T \tfrac{\partial \tilde{\Phi}}{\partial R}\right),  \\
					R_i &= \tilde{R} \tilde{F}.
				\end{align*}
			\end{subequations}
			\end{myframedeq}
			
			\begin{myframedeq}
			$F_{L_d}^{i+1}[(1-\alpha)h,h] : (\tilde{\bm{x}},\tilde{R},\bm{x}_{i},R_{i}) \mapsto (\bm{x}_{i},R_{i},\bm{x}_{i+1},R_{i+1})$
			\begin{subequations}
			\label{eqn:discreteLM3s}
				\begin{align*}
					0 &= \tfrac{m}{(1-\alpha) h}(\bm{x}_{i}- \tilde{\bm{x}}) - \tfrac{m}{h}(\bm{x}_{i+1} - \bm{x}_{i}) - \tfrac{mg(2-\alpha)h}{2} \bm{e_3}, \\
					0 &= \asym\left( \tfrac{1}{(1-\alpha)h} J_d \tilde{F} - \tfrac{1}{h}F_i J_d \right), 
						\qquad R_{i+1} = R_i F_i.
				\end{align*}
			\end{subequations}
			\end{myframedeq}
			
			\begin{myframedeq}
			$F_{L_d}^{\text{Imp} +}[\lambda][\alpha_1 h, \alpha_2 h] : (\bm{x}_{i-1},R_{i-1},\tilde{\bm{x}}_{1},\tilde{R}_{1}) \mapsto (\tilde{\bm{x}}_{1},\tilde{R}_{1},\tilde{\bm{x}}_{2},\tilde{R}_{2})$
			\begin{subequations}
			\label{eqn:discreteLM2m}
				\begin{align*}
					0 &= E_d(\tilde{q}_1,\tilde{q}_2,\alpha_2 h) - E_d(q_{i-1},\tilde{q}_1,\alpha_1 h), \\
					0 &= \tfrac{m}{\alpha_1 h}(\tilde{\bm{x}}_1 - \bm{x}_{i-1}) - \tfrac{m}{\alpha_2 h}(\tilde{\bm{x}}_2 - \tilde{\bm{x}}_1) - \tfrac{mg(\alpha_1 + \alpha_2)}{2}\bm{e_3} +  \lambda \tfrac{\partial \tilde{\Phi}_1}{\partial \bm{x}}, \\
					0 &= \asym\left( \tfrac{1}{\alpha_1 h} J_d F_{i-1} - \tfrac{1}{\alpha_2h}\tilde{F}_1J_d \right) + \lambda \asym \left( \tilde{R}_1^T\tfrac{\partial \tilde{\Phi}_1}{\partial R}\right), \\
					\tilde{R}_2 &= \tilde{R}_1 \tilde{F}_1, 
						\qquad 0 = \Phi(\bm{x}_2, \tilde{R}_2).
				\end{align*}
			\end{subequations}
			\end{myframedeq}

		\end{subtable}
		\hfill
		\begin{subtable}[t]{0.475 \textwidth}
		\captionsetup{singlelinecheck=false}
		\caption{\textbf{Hamiltonian Forms}}
		\label{tab:LGVCI_F_tilde_Ld}
		\centering
		
			\begin{myframedeq}
			$\tilde{F}_{L_d}[h] : (\bm{x}_k,R_k,\bm{\gamma}_k,\bm{\Pi}_k) \mapsto (\bm{x}_{k+1},R_{k+1},\bm{\gamma}_{k+1},\bm{\Pi}_{k+1})$
			\begin{subequations}
			\label{eqn:discreteHM0}
				\begin{align*}
					\bm{x}_{k+1} &= \bm{x}_k + \tfrac{h}{m}\bm{\gamma}_k -\tfrac{gh^2}{2}\bm{e_3}, 
						& \bm{\Pi}_{k+1} &= F_k^T \bm{\Pi}_k, \\
					\bm{\gamma}_{k+1} &= \bm{\gamma}_k - mgh\bm{e_3}, 
						& R_{k+1} &= R_kF_k. \\
					S(\bm{\Pi}_k) &= \tfrac{1}{h}\asym(F_k J_d),
				\end{align*}
			\end{subequations}
			\end{myframedeq}
			
			\begin{myframedeq}
			$\tilde{F}_{L_d}^\text{Imp}[\alpha h] : (\bm{x}_{i-1},R_{i-1},\bm{\gamma}_{i-1},\bm{\Pi}_{i-1}) \mapsto (\tilde{\bm{x}},\tilde{R},\tilde{\bm{\gamma}},\tilde{\bm{\Pi}})$
			\begin{subequations}
			\label{eqn:discreteHM1}
				\begin{align*}
					\tilde{\bm{x}} &= \bm{x}_{i-1} + \tfrac{\alpha h}{m}\bm{\gamma}_{i-1} -\tfrac{g \alpha^2 h^2}{2}\bm{e_3} ,
						& \tilde{\bm{\Pi}} &= F_{i-1}^T \bm{\Pi}_{i-1}, \\
					\tilde{\bm{\gamma}} &= \bm{\gamma}_{i-1} - mg\alpha h\bm{e_3}, 
						& \tilde{R} &= R_{i-1}F_{i-1}, \\
					S(\bm{\Pi}_{i-1}) &= \tfrac{1}{\alpha h}\asym(F_{i-1} J_d),
						& 0 &= \Phi(\tilde{\bm{x}},\tilde{R}).
				\end{align*}
			\end{subequations}
			\end{myframedeq}
			
			\begin{myframedeq}
			$\tilde{F}^i_{L_{d}}[\alpha h][\lambda] : (\tilde{\bm{x}},\tilde{R},\tilde{\bm{\gamma}},\tilde{\bm{\Pi}}) \mapsto (\bm{x}_i,R_{i},\bm{\gamma}_{i},\bm{\Pi}_{i})$
			\begin{subequations}
			\label{eqn:discreteHM2s}
				\begin{align*}
					0 &= E_d(\tilde{q},q_i,(1-\alpha)h) - E_d(q_{i-1},\tilde{q},\alpha h), \\
					\bm{x}_i &= \tilde{\bm{x}} + \tfrac{(1-\alpha) h}{m}\tilde{\bm{\gamma}} -\tfrac{g (1-\alpha)^2 h^2}{2}\bm{e_3} + \lambda \tfrac{(1-\alpha)h}{m} \tfrac{\partial \tilde{\Phi}}{\partial \bm{x}}, \\
					\bm{\gamma}_{i} &= \tilde{\bm{\gamma}} - mg(1-\alpha) h\bm{e_3} + \lambda \tfrac{\partial \tilde{\Phi}}{\partial \bm{x}}, \\
					S(\tilde{\bm{\Pi}}) &= \asym\left(\tfrac{1}{(1-\alpha) h}\tilde{F}J_d \right) - \lambda \asym \left( \tilde{R}^T \tfrac{\partial \tilde{\Phi}}{\partial R}\right), \\
				\tilde{F}\bm{\Pi}_i &= \tilde{\bm{\Pi}}	+ \lambda\tilde{\bm{\chi}},
					\qquad R_i = \tilde{R}\tilde{F}.
				\end{align*}
			\end{subequations}
			\end{myframedeq}

			\begin{myframedeq}
			$\tilde{F}_{L_d}^{\text{Imp+}}[\lambda][\alpha_2 h] : (\tilde{\bm{x}}_1,\tilde{R}_1,\tilde{\bm{\gamma}}_1,\tilde{\bm{\Pi}}_1) \mapsto (\tilde{\bm{x}}_2,\tilde{R}_2,\tilde{\bm{\gamma}}_2,\tilde{\bm{\Pi}}_2) $
			\begin{subequations}
			\label{eqn:discreteHM2m}
				\begin{align*}
					0 &= E_d(\tilde{q}_1,\tilde{q}_2,\alpha_2 h) - E_d(q_{i-1},\tilde{q}_1,\alpha_1 h), \\
					\tilde{\bm{x}}_2 &= \tilde{\bm{x}}_1 + \tfrac{\alpha_2 h}{m}\tilde{\bm{\gamma}}_1 -\tfrac{g \alpha_2^2 h^2}{2}\bm{e_3} + \lambda \tfrac{\alpha_2 h}{m} \tfrac{\partial \tilde{\Phi}_1}{\partial \bm{x}}, \\
					\tilde{\bm{\gamma}}_2 &= \tilde{\bm{\gamma}}_1 - mg\alpha_2 h\bm{e_3} + \lambda \tfrac{\partial \tilde{\Phi}_1}{\partial \bm{x}}, \\
					S(\tilde{\bm{\Pi}}_1) &= \asym\left(\tfrac{1}{\alpha_2 h}\tilde{F}_1 J_d \right)  - \lambda \asym \left( \tilde{R}_1^T\tfrac{\partial \tilde{\Phi}_1}{\partial R}\right), \\
					\tilde{F}_1\tilde{\bm{\Pi}}_2 &= \tilde{\bm{\Pi}}_1	+ \lambda \tilde{\bm{\chi}}_1, 
						\qquad \tilde{R}_2  = \tilde{R}_1\tilde{F}_1,
							\qquad 0 = \Phi(\bm{x}_2, \tilde{R}_2).
				\end{align*}
			\end{subequations}
			\end{myframedeq}
			
		\end{subtable}
	\caption{Variational integrators in Lagrangian and Hamiltonian Forms}
	\label{tab:variationalIntegrators}
	\end{sidewaystable}
	
	%----------------------------------------------%
	%Section 5.5
	%----------------------------------------------%
	\subsection{The Algorithm and the Zeno Phenomenon: A Summary}
	\label{sect:Algorithm}
	%----------------------------------------------%
	
	The collection of integrators discuss previously shall be called \textit{Lie group variational collision integrators} (LGVCI), which have been written in both Lagrangian and Hamiltonian forms. In this section, we provide two algorithms that summarize the LGVCI in the Hamiltonian form: The first algorithm recalls all of the necessary integrators in Table \ref{tab:LGVCI_F_tilde_Ld}. The second algorithm is streamlined by utilizing only the discrete Hamiltonian map $\tilde{F}_{L_d}$, the bisection method for $\alpha_j$, and the jump map $\tilde{F}_{\jump}$. We only provide the summary on the Hamiltonian side because it is common to describe hybrid systems in this way. However, it is also straightforward to write the algorithms on the Lagrangian side by imitating the ones presented here.
	
		%----------------------------------------------%
		%Section 5.5.1
		%----------------------------------------------%
		\subsubsection{Algorithm 1}
		\label{sect:Algorithm1}
		%----------------------------------------------%
	Let us consider Table \ref{alg:algLong} (Algorithm 5.1): For the inputs, we have the initial states, $(\bm{x}_{\text{init}}, R_{\text{init}}, \bm{\gamma}_{\text{init}}, \bm{\Pi}_{\text{init}})$, and the number of discrete timesteps to be taken, $M$. The algorithm returns $(t_i^j,\bm{x}_i^j,R_i^j,\bm{\gamma}_i^j,\bm{\Pi}_i^j)$ for $0 \leq i \leq M$. This is the set of all the discrete states $(\bm{x}_i^j,R_i^j,\bm{\gamma}_i^j,\bm{\Pi}_i^j)$ occurring at time $t_i^j$. In particular, $t_i^0 = ih$, where $h$ is the fixed timestep. In the case where impacts occur within the interval $[t_i^0,t_{i+1}^0]$, $t_i^j$ represents the times of the collisions, where $1 \leq j \leq c_i$ and $c_i$ is number of collisions in the interval. The following parameters are required for the algorithm:
	\begin{itemize}
		\item $m$ is the mass of the rigid-body.
		\item $g$ is the gravitational acceleration.
		\item $J$ is the standard inertia matrix.
		\item $I_\epsilon$ is the diagonal matrix defined in \eqref{eqn:fE} for the ellipsoid.
		\item $\bm{n}$ is the normal vector of the plane $\PP$.
		\item $\epsilon_{\text{tol}}$ is the tolerance for bisection and Newton's method (see Section 3.4 in \cite{{Lee2007orbital}}).
	\end{itemize}
	
	The algorithm has the following brief structure: Variables are initialized, and the initial distance between the ellipsoid and plane is stored in $\dist$, which is our indicator for collision and/or interpenetration. Next, a for-loop for $1 \leq i \leq M$ has statements falling into one of the three conditions: $\dist > 0$, $\dist < 0$, and $\dist = 0$. The first condition uses $\tilde{F}_{L_d}[h]$ to determine states away from the point of impact. The second condition discards the states with interpenetration and determine the states of the first impact using $\tilde{F}_{L_d}^{\imp}[\alpha_j h]$. The last condition takes care of the cases where there could be multiple impacts using $\tilde{F}_{L_d}^{\imp +}[\lambda][\alpha_j h]$ or single impact using $\tilde{F}_{L_d}^{i}[(1-\alpha_{\tot})h][\lambda]$.
	
	Note that it is possible to encounter the Zeno phenomenon in our hybrid systems due to the form of the ellipsoid and other factors; these include the sectional curvature at the point of impact and the relative magnitude of the rotational and translational energies. As a result, a footnote to indicate where to include exception handling for Zeno phenomenon is made in the algorithm. The Zeno phenomenon would manifest itself in the algorithm by arbitrarily large values of $j$ in the while-loop, and $\alpha_j$ that are all approximately machine precision. To remedy this, we recommend either breaking out of all the loops or returning the outputs when $j$ exceeds a user-defined threshold.
	 
		%----------------------------------------------%
		%Section 5.5.2
		%----------------------------------------------%
		\subsubsection{Algorithm 2}
		\label{sect:Algorithm2}
		%----------------------------------------------%
	We present a more streamlined algorithm, which only uses $\tilde{F}_{L_d}$, the bisection method, and $\tilde{F}_{\jump}$. This is based on the observation that the other discrete Hamiltonian maps are equivalent to some combination of these three subroutines. Namely, $\tilde{F}_{L_d}^{\imp}[\alpha_j h] = \tilde{F}_{L_d}[\alpha_j h]$ after $\alpha_j \in (0,1-\alpha_{\tot})$ is determined using bisection method. There is also $\tilde{F}_{L_d}^{i}[(1-\alpha_{\tot})h][\lambda] = \tilde{F}_{L_d}[(1-\alpha_{\tot})h]$ after $\lambda$ is determined and the instantaneous states before impact are updated with the momenta after the impact from $\tilde{F}_{\jump}$. Finally, $\tilde{F}_{L_d}^{\imp +}[\lambda][\alpha_j h] = \tilde{F}_{L_d}[\alpha_j h]$ following both the updates from $\tilde{F}_{\jump}$ and bisection method for $\alpha_j$. Consequently, we have the more streamlined Table \ref{alg:algShort} (Algorithm 5.2) with the same parameters, inputs, and return.
	
	\begin{sidewaystable}
	\begin{subtable}[t]{0.5 \textwidth}
	\captionsetup{singlelinecheck=false}
	\caption{\textbf{Algorithm 5.1}}
	\label{alg:algLong}
	\centering
	\begin{algorithmic} 
		\STATE \textbf{input:} $(\bm{x},R,\bm{\gamma},\bm{\Pi})_\text{init}, M$
		\STATE $i= j=0; \quad \alpha_j= \alpha_{\tot}=0;$
		\STATE $(t,\bm{x},R,\bm{\gamma},\bm{\Pi})_i^j \gets (0,\bm{x},R,\bm{\gamma},\bm{\Pi})_\text{init}; \quad \dist \gets \Phi(\bm{x},R)_i^j;$
		\FOR{$i = 1 : M$}
			\IF{$\dist > 0$}
				\STATE $j=0; \quad t_i^j = t_{i-1}^j + h;$
				\STATE $(\bm{x},R,\bm{\gamma},\bm{\Pi})_i^j \gets \tilde{F}_{L_d}[h](\bm{x},R,\bm{\gamma},\bm{\Pi})_{i-1}^j; \dist \gets \Phi(\bm{x},R)_i^j;$
			\ELSIF{$\dist < 0$}
				\STATE $i \minuseq 1; \quad j\pluseq 1;$
				\STATE $(\bm{x},R,\bm{\gamma},\bm{\Pi})_i^j \gets \tilde{F}_{L_d}^{\imp}[\alpha_j h]  (\bm{x},R,\bm{\gamma},\bm{\Pi})_{i}^{j-1};$
				\STATE $\dist \gets \Phi(\bm{x},R)_i^j; \quad t_i^j = t_i^{j-1} + \alpha_j; \quad \alpha_{\tot} \pluseq \alpha_j;$
			\ELSE 
				\STATE $i \minuseq 1;$
				\WHILE{$\dist \leq 0$}
					\IF{$\dist = 0$}
						 \STATE $(\bm{x},R,\bm{\gamma},\bm{\Pi})_{\tmp} \gets \tilde{F}_{L_d}^{i}[(1-\alpha_{\tot})h][\lambda](\bm{x},R,\bm{\gamma},\bm{\Pi})_{i}^j;$
						 \STATE $\dist \gets \Phi(\bm{x},R)_{\tmp}$
						 \IF{$\dist > 0$}
							\STATE $i \pluseq 1; \hspace{2ex} j = 0; \hspace{2ex} \alpha_j = \alpha_{\tot}=0; \hspace{2ex} t_i^j = t_{i-1}^j + h;$
							\STATE $(\bm{x},R,\bm{\gamma},\bm{\Pi})_i^j \gets (\bm{x},R,\bm{\gamma},\bm{\Pi})_{\tmp}$;
						\ENDIF
					\ELSE
						\STATE $j \pluseq 1;^\dag$
						\STATE $(\bm{x},R,\bm{\gamma},\bm{\Pi})_i^j \gets \tilde{F}_{L_d}^{\text{Imp}+}[\lambda][\alpha_j h](\bm{x},R,\bm{\gamma},\bm{\Pi})_{i}^{j-1};$
						\STATE $\dist \gets \Phi(\bm{x},R)_i^j; \quad t_i^j = t_i^{j-1} + \alpha_j; \quad \alpha_{\tot} \pluseq \alpha_j;$
					\ENDIF
				\ENDWHILE			
			\ENDIF
		\ENDFOR
		\RETURN $(t_i^j,\bm{x}_i^j,R_i^j,\bm{\gamma}_i^j,\bm{\Pi}_i^j)$
	\end{algorithmic}
	\end{subtable}
	\rulesep
	\begin{subtable}[t]{0.49 \textwidth}
	\captionsetup{singlelinecheck=false}
	\caption{\textbf{Algorithm 5.2}}
	\label{alg:algShort}
	\centering
	\begin{algorithmic} 
		\STATE \textbf{input:} $(\bm{x},R,\bm{\gamma},\bm{\Pi})_\text{init}, M$
		\STATE $i= j=0; \quad \alpha_j= \alpha_{\tot}=0;$
		\STATE $(t,\bm{x},R,\bm{\gamma},\bm{\Pi})_i^j \gets (0,\bm{x},R,\bm{\gamma},\bm{\Pi})_\text{init}; \quad \dist \gets \Phi(\bm{x},R)_i^j;$
		\FOR{$i = 1 : M$}
			\IF{$\dist \geq 0$}
				\IF{$\dist > 0$}
					\STATE $j=0; \quad t_i^j = t_{i-1}^j + h;$
					\STATE $(\bm{x},R,\bm{\gamma},\bm{\Pi})_i^j \gets \tilde{F}_{L_d}[h](\bm{x},R,\bm{\gamma},\bm{\Pi})_{i-1}^j;$
					\STATE $\dist \gets \Phi(\bm{x},R)_i^j;$
				\ELSE
					\STATE $i \minuseq 1;$
					\STATE $(\lambda,\bm{\gamma},\bm{\Pi})_i^j \gets \tilde{F}_{\jump}(\bm{x},R,\bm{\gamma},\bm{\Pi})_{i}^{j};$
					\STATE $(\bm{x},R,\bm{\gamma},\bm{\Pi})_{\tmp} \gets \tilde{F}_{L_d}[(1-\alpha_{\tot})h](\bm{x},R,\bm{\gamma},\bm{\Pi})_i^j;$
					\STATE $\dist \gets \Phi(\bm{x},R)_{\tmp};$
					\IF{$\dist > 0$}
						\STATE $i \pluseq 1; \quad j = 0; \quad \alpha_j=\alpha_{\tot}=0; \quad t_i^j = t_{i-1}^j + h$
						\STATE $(\bm{x},R,\bm{\gamma},\bm{\Pi})_i^j \gets (\bm{x},R,\bm{\gamma},\bm{\Pi})_{\tmp};$
					\ENDIF
				\ENDIF
			\ELSE
				\STATE $i \minuseq 1; \quad j \pluseq 1;^\dag \quad \alpha_j \gets \text{Bisection}(0,1-\alpha_{\tot});$
				\STATE $(\bm{x},R,\bm{\gamma},\bm{\Pi})_i^j \gets \tilde{F}_{L_d}[\alpha_j h] (\bm{x},R,\bm{\gamma},\bm{\Pi})_{i}^{j-1};$
				\STATE $\dist \gets \Phi(\bm{x},R)_i^j; \quad t_i^j = t_i^{j-1} + \alpha_j;\quad \alpha_{\tot} \pluseq \alpha_j;$
			\ENDIF
		\ENDFOR
		\RETURN $(t_i^j,\bm{x}_i^j,R_i^j,\bm{\gamma}_i^j,\bm{\Pi}_i^j)$
	\end{algorithmic}
	\end{subtable}
	\label{fig:alg}
	\caption{LGVCI algorithms. $^\dag$Zeno phenomenon: $j$ is arbitrarily large, so break/return is recommended.}
	\end{sidewaystable}
	
%----------------------------------------------%
%----------------------------------------------%
	%SECTION 6: Extensions: Tilted Planes and Other Rigid-Bodies
%----------------------------------------------%
%----------------------------------------------%
\section{Extensions: Tilted Planes and Other Rigid-Bodies}
\label{sect:Extensions}
	We extend our problem of the bouncing ellipsoid by considering tilted planes and/or other rigid-bodies. This is a natural next step because we want to describe the collision dynamics of different hybrid systems. Note that the theories discussed in both the continuous and discrete cases remain the same with the exception of the collision detection function $\Phi$. Furthermore, it is sufficient that the partial derivatives exists for all ``probable'' configurations of impact, which will be discussed further in this section. Hence, the different collision detection functions and their partial derivatives are discussed here for different hybrid systems.
	
	General planes are defined as $\PP(\tilde{\bm{n}},D) = \{ \bm{z} \in \R^3 \mid \tilde{\bm{n}}^T \bm{z} + D = 0 \}$ for some $\tilde{\bm{n}} \in S^2$ and $D \in \R$. We define the signed distance function for an arbitrary half-plane 
	\begin{equation}
	\label{eqn:halfPlane}
		\mathcal{HP}(\tilde{\bm{n}},D) = \{ \bm{z} \in \R \mid \tilde{\bm{n}}^T\bm{z} + D \leq 0\},
	\end{equation}
	which is the region opposite from the direction of the normal vector $\tilde{\bm{n}}$ of the plane: 
	\begin{proposition}
	\label{prop:SDFhalfPlane}
	The signed distance function $\psi_{\PP} : \R^3 \to \R $ for the half plane $\mathcal{HP}(\tilde{\bm{n}},D) $ is defined by $\psi_{\PP}(\bm{z}) = \tilde{\bm{n}}^T\bm{z} + D $.
	\end{proposition}
	
	%----------------------------------------------%
	%Section 6.1
	%----------------------------------------------%
	\subsection{Tilted Planes}
	\label{sect:TiltedPlanes}
	%----------------------------------------------%
	Recall that our system describing the bouncing ellipsoid involves the horizontal plane denoted as $\PP = \{ \bm{z} \in \R^3 \mid \bm{n}^T \bm{z} = 0 \}$ where $ \bm{n} = (0,0,1)^T$. For tilted planes, it suffices to consider planes with a normal vector
	$
		\tilde{\bm{n}} \in S^2_+ =  \{ \bm{z} \in \R^3 \mid \| \bm{z} \| = 1 \text{ and } z_3 > 0 \}
	$
	that pass through the origin because shifted planes with the same normal vector $\tilde{\bm{n}}$ are equivariant with respect to translations in the gravitational direction. Furthermore, we require $\tilde{n}_3 > 0$ so that the orientation is preserved. The case $\tilde{\bm{n}}=(0,0,1)^T \in S^2_+$ gives the horizontal plane in our original discussion.
	
	Consider the tilted plane
	$
		\PP_{\tilde{\bm{n}}} = \{ \bm{z} \in \R^3 \mid \tilde{\bm{n}}^T \bm{z} = 0 \text \}.
	$
	Now, suppose $(\bm{x},R) \in SE(3)$ so that the arbitrary ellipsoid, $\EE' = T_{(\bm{x},R)}(\EE)$, is above the tilted plane; the collision detection function is written as
	\begin{equation}
	\label{eqn:distTiltedPlane}
		d_2(\EE',\PP_{\tilde{\bm{n}}}) = \Phi(\bm{x},R) = \tilde{\bm{n}}^T \bm{x} -  \|I_\EE R^T \tilde{\bm{n}} \|.
	\end{equation}
	Then, $\EE' \cap \PP_{\tilde{\bm{n}}} = \emptyset$ if and only if $\tilde{\bm{n}}^T \bm{x} > \|I_\EE R^T \tilde{\bm{n}} \|$. The partial derivatives are computed in the same way as before:
	\begin{equation}
	\label{eqn:distTiltedPlane_partials}
		\left( \frac{\partial \Phi}{\partial \bm{x}}, \frac{\partial \Phi}{\partial R} \right) = \left( \tilde{\bm{n}}, -\frac{\tilde{\bm{n}}\tilde{\bm{n}}^T R I_\epsilon^2}{\| I_\epsilon R^T \tilde{\bm{n}} \|}\right).
	\end{equation}
	
	%----------------------------------------------%
	%Section 6.2
	%----------------------------------------------%
	\subsection{Convex Rigid-Bodies and Interface Implicit Representations}
	\label{sect:ConvexRigidBodies}
	%----------------------------------------------%
	
	We propose a way to construct the collision detection function $\Phi$ for some convex domains. 
	
	Suppose the domain of the rigid-body $\BB  \subset \R^3$ is convex and compact, and $\rho : \BB \to \R$ is the mass density function such that the center of mass of $\BB $ coincides with the origin. The function also provides the respective standard and nonstandard inertia matrix $J$ and $J_d$ for our numerical implementation. Now, let the \textit{interface} $\partial \BB $ be $C^k$ where $k \geq 2$. This interface divides $\R^3$ into two separate regions where the interior $\BB^\mathrm{o}$ is the \textit{inside} and the complement $\BB^\mathrm{c}$ is the \textit{outside} of the domain. In addition, the assumption on smoothness guarantees that there are no \textit{edges} -- a curve where two surfaces meet -- and no \textit{vertices}, points where two or more edges meet; also, it guarantees that the partial derivatives of $\Phi$ will be continuous. 
	
	Now, the interface can be represented in two ways: In an \textit{explicit} representation, one explicitly writes all the points that belong to the interface. On the other hand, \textit{implicit} representation of the interface is given by the zero level set of some implicit function $\phi$ after a constant shift if necessary \citep{osher2003signed}. In fact, the \textit{signed distance functions} (SDFs) $\psi : \R^3 \to \R$ are a subset of such implicit functions defined by
	\begin{equation}
	\label{eqn:sdf}
		\psi(\bm{z}) =
		\left\{
			\begin{array}{ll}
				d_2(\bm{z},\partial \BB),	&	\bm{z} \in \BB ^\mathrm{c}\\
				-d_2(\bm{z},\partial \BB),	&	\bm{z}\in \BB ,
			\end{array}
		\right.
	\end{equation}
	but it also satisfies $\|\nabla \psi\| = 1$ everywhere except on the zero level set. We reserve $\psi$ for SDFs and $\phi$ as the general class of functions whose zero level set is $\partial \BB$. In particular, these functions are positive for $\bm{z}$ outside of the interface, negative for $\bm{z}$ inside of the interface, and zero otherwise. Of course, $|\psi(\bm{z})|$ is the distance from $\bm{z}$ to the interface. Lastly, note that
	$
		\nabla \psi(\bm{z}) = N(\bm{z}),
	$
	where $N$ is the outward normal vector field. Essentially, the collision detection function and its partial derivatives between the body and plane will be constructed using the SDF or some other implicit representation of the interface.
		%----------------------------------------------%
		%Section 6.2.1
		%----------------------------------------------%
		\subsubsection{Theory: The Collision Detection}
		\label{sect:ConvexRigidBodies1}
		%----------------------------------------------%
	Suppose the hybrid system involves a convex rigid-body $\BB$ with $C^k$ boundary ($k\geq2$), and let the SDF $\psi$ or some implicit representation $\phi$ of the rigid-body be given. The collision detection problem is described first in the body-fixed frame: Suppose $\PP_{\tilde{\bm{n}}} = \PP(\tilde{\bm{n}},0)$ is a plane containing the origin and $\BB' = T_{(\bm{x},R)}(\BB)$ is the translated, rotated body above the plane for some $(\bm{x},R) \in SE(3)$. Suppose $\BB' \cap \PP_{\tilde{\bm{n}}} = \emptyset$, and so the distance between them can be written equivalently as
	\[
		d_2(\BB',\PP_{\tilde{\bm{n}}}) = d_2(\BB, \PP_{\tilde{\bm{n}}}'),
	\]
	where
	$
		\PP_{\tilde{\bm{n}}}' 
		= \PP(R^T \tilde{\bm{n}}, \tilde{\bm{n}}^T \bm{x}) 
		= \{ \bm{z} \in \R^3 \mid \tilde{\bm{n}}^T R \bm{z} + \tilde{\bm{n}}^T\bm{x} = 0 \}
	$
	since $T_{(\bm{x},R)}$ is an isometry. This equivalent view of the distance is constructed from the body-fixed frame, so one can imagine a shifted and rotated plane $\PP_{\tilde{\bm{n}}}' = {T_{(\bm{x},R)}}^{-1}(\PP_{\tilde{\bm{n}}})$ in this frame. Furthermore, the construction is made in this frame because the body elements in $\BB$, and especially its boundary, are fixed; this allows us to easily construct $\Phi$ using a constrained optimization problem, where the constraint is $\bm{\rho} \in \partial \BB$.
	
	\begin{figure}
	\centering
		\begin{subfigure}[b]{0.45 \textwidth}
		\includegraphics{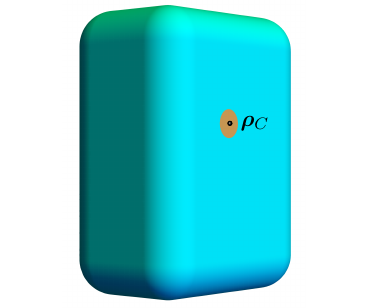}
			\caption{The neighborhood of $\bm{\rho}_C \in \partial \BB$ is locally flat where all the points are also closest to the plane.}
			\label{fig:exception1}
		\end{subfigure}
		\hspace{1 cm}
		\begin{subfigure}[b]{0.45 \textwidth}
		\includegraphics{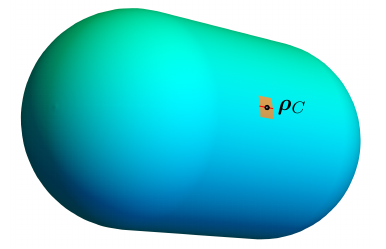}
			\caption{The neighborhood of $\bm{\rho}_C \in \partial \BB$ with a straight red curve where all the points are also closest to the plane.}
			\label{fig:exception2}
		\end{subfigure}
		\caption{Examples of neighborhood with non-unique closest points on convex bodies with $C^\infty$ boundary.}
		\label{fig:exceptions}
	\end{figure}
	
	Let us denote $\bm{\rho}_C \in \partial \BB$ as \textit{a closest point to the plane} $\PP_{\tilde{\bm{n}}}'$ in the body-fixed frame, meaning
	\[
		d_2(\bm{\rho}_C,\PP_{\tilde{\bm{n}}}') = d_2(\BB, \PP_{\tilde{\bm{n}}}').
	\]
	For most configurations $(\bm{x},R) \in SE(3)$, there is a unique $\bm{\rho}_C$ because $\BB$ is convex and has $C^k$ ($k\geq 2$) boundary. The exception occurs for configurations that have zero Gaussian curvature at the point $\bm{\rho}_C \in \partial \BB$, which can be categorized into two cases (see Figure \ref{fig:exceptions}): In the first case, both principal curvatures at $\bm{\rho}_C$ are zeros, and the neighborhood of $\bm{\rho}_C$ is a plane which is parallel to the $\PP_{\tilde{\bm{n}}}'$. This implies that all the points in this locally planar neighborhood is also closest to the plane (e.g., see Figure \ref{fig:exception1}). In the second case, one of the principal curvatures at $\bm{\rho}_C$ is zero; in particular, a neighborhood of $\bm{\rho}_C$ in $\partial \BB$ will have a curve with points which are all closest to the plane, and this curve is parallel to the plane as well, e.g., see Figure \ref{fig:exception2}. However, these cases are \textit{improbable} in numerical computations because the discrete configuration $(\bm{x},R)$ must be in a very particular and isolated state for the two cases to arise. Therefore, it is also improbable for these cases to arise at the point of impact, so we only consider the \textit{probable} configurations where $\bm{\rho}_C$ is unique. This notion of probable configurations is analogous to the notion of general position or genericity that arises in computational geometry and algebraic geometry.
	
	The improbable configurations are discussed here because we do not want to dismiss them entirely, since they are still possible configurations of hybrid systems. However, the LGVCI algorithm would not capture these configurations, which is actually a good representation of the hybrid systems in real life. In particular, the set of improbable configurations is a set of measure zero in all the configurations away from impacts. Furthermore, the improbable configurations will not occur during impact in our implementation as well, so $\BB \cap \PP_{\tilde{\bm{n}}}' = \{ \bm{\rho}_C \}$ at the configurations of impact.
	
	Given this unique point $\bm{\rho}_C$, the collision detection and its partial derivatives can now be discussed. Let $(\bm{x},R)$ be a probable configuration so that $\PP_{\tilde{\bm{n}}}'$ is given in the body-fixed frame. In theory, $\bm{\rho}_C$ is dependent on $(\bm{x},R)$ and can be determined since $\partial \BB$ is compact. Namely, $\bm{\rho}_C(\bm{x},R)$ is determined by the $\argmin$ of all possible distances from the boundary points to the plane:
	\begin{equation}
	\label{eqn:rhoC}
		\bm{\rho}_C(\bm{x},R) 
		= \argmin_{\bm{\rho} \in \partial \BB} \, \psi_{\PP_{\tilde{\bm{n}}}'}(\bm{\rho})	\\
		= \argmin_{\{\bm{\rho} \in \R^3 \mid \phi(\bm{\rho}) = 0\}} \tilde{\bm{n}}^T R \bm{\rho} + \tilde{\bm{n}}^T \bm{x},
	\end{equation}
	using the signed distance function for the plane defined in Proposition \ref{prop:SDFhalfPlane}. Also, recall that $\phi^{-1}(0) = \partial \BB$ for an implicit representation $\phi$, and this remains true if the representation is an SDF. As a result, the collision detection function is defined simply as 
	\begin{equation}
	\label{eqn:PhiConvex1}
		\Phi(\bm{x},R)
			= \min_{\{\bm{\rho} \in \R^3 \mid \phi(\bm{\rho}) = 0\}} \tilde{\bm{n}}^T R \bm{\rho} + \tilde{\bm{n}}^T \bm{x}	\\
			=  \tilde{\bm{n}}^T R \bm{\rho}_C(\bm{x},R) + \tilde{\bm{n}}^T \bm{x}.	
	\end{equation}
	Indeed, this is desirable because this function also encodes interpenetration. Namely, if the plane intersects the body with interpenetration for a given configuration $(\bm{x},R) \in SE(3)$, then $\tilde{\bm{n}}^T R \bm{\rho} + \tilde{\bm{n}}^T \bm{x} < 0$ for some $\bm{\rho} \in \partial \BB$, and so $\Phi(\bm{x},R) < 0$. In addition, $\Phi(\bm{x},R) < 0$ for configurations where the body completely passes through the plane because all the boundary points would be in the interior of the half-plane $\mathcal{H}\PP_{\tilde{\bm{n}}}'$ defined in \eqref{eqn:halfPlane}.
	
	Finally, $\bm{\rho}_C(\bm{x},R) \in \partial \BB$ can be solved using constrained optimization in  \eqref{eqn:PhiConvex1}. Recall that possible solutions to the problem are the critical points for the \textit{Lagrangian function}
	\[
		\mathcal{L}(\bm{x},R;\bm{\rho}) = \tilde{\bm{n}}^T R \bm{\rho} + \tilde{\bm{n}}^T \bm{x} + \lambda \phi(\bm{\rho}),
	\]
	where $\lambda$ is the Lagrange multiplier. Observe that $\frac{\partial \mathcal{L}}{\partial \bm{\rho}} = \bm{0}$ implies that $\lambda \nabla \phi(\bm{\rho}) = -R^T \tilde{\bm{n}}$. Hence we have the following remark:
	\begin{remark}
	\label{rem:tangentplaneRhoC}
		Given a probable configuration $(\bm{x},R) \in SE(3)$,  $\bm{\rho}_C \in \partial \BB$ is unique and $\bm{z} = \bm{\rho}_C$ is a solution to $\nabla \psi(\bm{z}) = \pm R^T \tilde{\bm{n}}$ or $\lambda \nabla \phi(\bm{z}) = -R^T \tilde{\bm{n}}$ for some $\lambda \in \R$.
	\end{remark}
	Intuitively, this means that the tangent plane at $\bm{\rho}_C \in \partial \BB$ is parallel to $\PP_{\tilde{\bm{n}}}'$ since they share the same normal vector up to a scalar multiplier. Furthermore, as a consequence of Remark \ref{rem:tangentplaneRhoC}, one might conclude that the dependencies should be changed, $\bm{\rho}_C(\bm{x},R) \longrightarrow \bm{\rho}_C(R)$.
		%----------------------------------------------%
		%Section 6.2.1
		%----------------------------------------------%
		\subsubsection{Theory: The Partial Derivatives of \texorpdfstring{$\Phi$}{Phi}}
		\label{sect:ConvexRigidBodies2}
		%----------------------------------------------%
	We continue our discussion with the partial derivatives of the collision detection function. By the smoothness of the interface, its partial derivatives exist and are continuous. Specifically, they should exist for probable configurations at impact since the partial derivatives are computed during collision. Given the setup for the constrained optimization problem, we can simply compute the partial derivatives as
	\begin{equation}
	\label{eqn:partialsGen}
		\left( \frac{\partial \Phi}{\partial \bm{x}}, \biggl.\frac{\partial \Phi}{\partial R} \right)\biggr\lvert_{(\bm{x},R)} = \left( \tilde{\bm{n}}, \tilde{\bm{n}}\bm{\rho}_C^T +  \tilde{\bm{n}}^T R \frac{\partial \bm{\rho}_C}{\partial R} \right).
	\end{equation}
	Intuitively, $\frac{\partial \Phi}{\partial \bm{x}}$ makes sense because given a fixed $R$, the gradient of $\Phi$ with respect to $\bm{x}$, the position of the center of mass of the body, should point towards the direction of greatest increase for $\Phi$; this is, indeed, the normal vector of the plane $\tilde{\bm{n}}$. Hence, given any admissible configurations and configurations at impact, $\frac{\partial \Phi}{\partial \bm{x}} = \tilde{\bm{n}}$.
	
	We may also check this against the example of the ellipsoid. Again, using inversive geometry and/or constrained optimization, the closest point on the transformed ellipsoid to the plane $\PP_{\tilde{\bm{n}}}'$ is
	$
		\bm{\rho}_C(R) = - \frac{I_\epsilon^2 R^T \tilde{\bm{n}}}{\| I_\epsilon R^T \tilde{\bm{n}} \|}.
	$ 
	One can show that $\tilde{\bm{n}}^T R \frac{\partial \bm{\rho}_C}{\partial R} = 0$ and then get $\frac{\partial \Phi}{\partial R}(\bm{x},R) = \tilde{\bm{n}} \bm{\rho}_C(\bm{x},R)^T = - \frac{\tilde{\bm{n}} \tilde{\bm{n}}^T R I_\epsilon^2}{\| I_\epsilon R^T \tilde{\bm{n}} \|}$ for all admissible configurations and configurations at impact. Indeed, this was already derived in \eqref{eqn:distTiltedPlane_partials}.

	Lastly, $\frac{\partial \Phi}{\partial R}$ depends on the uniqueness of $\bm{\rho}_C$ which holds by assumption, since we have restricted ourselves to probable configurations. Suppose instead that we are given an improbable configuration $(\bm{x},R) \in SE(3)$ and $\bm{\rho}_C$ is not unique, and we decide to choose one of the closest points as a representation. Then, the partial derivative $\frac{\partial \Phi}{\partial R}$ will not be well-defined, and it will be different for each representation. This is unsurprising, since the directional derivative along any fixed rotation will be drastically different for the different elements of the closest point set. 	
	%----------------------------------------------%
	%Section 6.3
	%----------------------------------------------%
	\subsection{Convex Polyhedra}
	\label{sect:ConvexPoly}
	%----------------------------------------------%
	
	In this section, we extend our theory for hybrid systems by considering rigid-bodies that are convex polyhedra. In particular, a convex polyhedron $\BB$ is defined by the convex hull of a collection of vertices $\{\bm{v}_j\}_{j=1}^l$ where $l \geq 4$.  Its centroid and moment of inertia may be computed using the Mirtich's algorithm in \cite{mi1996}. Otherwise, one may utilize built-in functions \texttt{RegionCentroid} and \texttt{MomentOfInertia} in Mathematica 12 to compute the centroid and inertia matrix, respectively. Of course, the centroid of the body can be made to coincide with the origin by shifting the vertices. 
	
	Note that $\BB$ has a $C^0$ interface, which consists of faces, edges, and vertices.  In Section \ref{sect:ConvexRigidBodies}, we discussed the improbable configurations where the closest point to the plane is not unique, and this case would arise when any one of the faces or edges of the polyhedron is closest and parallel to the plane. However again, these configurations remain improbable due to finite numerical precision and numerical roundoff. In fact, suppose one of the edges is closest and parallel to the plane; a small perturbation in the configuration of the body will leave one of the two vertices of the edge as the closest point to the plane. Similarly, if one of the faces is closest and parallel to the plane, a small perturbation will leave one of the vertices of the face as the closest point. As a result, the closest point is almost always a vertex, and it is unique; this also ensures that a vertex is always the singleton at the point of impact. 
	
	Despite this simplification, we face a challenge applying our jump conditions at the points of impact, which are sharp corners of the polyhedron. One possible approach is to replace our continuous equations of motion with \textit{differential inclusions}. This will yield a set of possible momenta that lie in the normal cone (in the sense of convex analysis) at the configuration of impact $(\bm{x}_i,R_i)$. To compute a realization, we would have to choose a momenta within the normal cone to determine the state after the collision, which adds an element of randomness to the simulation; this is undesirable. We avoid these issues altogether by considering a regularization of the rigid-body. Using this method, the modified convex polyhedron has a $C^\infty$ boundary.  Hence, its impact point on the plane will be a singleton, and the partial derivatives will exist.
	
	Let us begin by considering the SDF  $\psi$ for the convex polyhedron $\BB$, then the interface is given by the zero level set, i.e., $\psi^{-1}(0) = \partial \BB$. We modify the SDF, so that the zero level set has smoothness near the vertices, and this is done using \textit{$\epsilon$-rounding} where $\epsilon$ is a sufficiently small parameter: Define the new SDF $\psi_\epsilon : \R^3 \to \R$ by
	\begin{equation}
	\label{eqn:SDFepsilon}
		\psi_\epsilon(\bm{x}) = \psi(\bm{x}) - \epsilon.
	\end{equation}
	Then, the zero level set of $\psi_\epsilon$ is similar to $\BB$, but its surfaces and edges are $\epsilon$-distance further away from each respective surfaces and edges. Furthermore, the corners are now rounded with a radius of curvature bounded from below by $\epsilon$, for example in Figure \ref{fig:e-rounding}. Hence, define $\partial\BB^\epsilon = \psi_\epsilon^{-1}(0)$ as the new interface of interest, and let $\BB^\epsilon$ be the new rigid-body of interest defined by the convex hull of $\partial\BB^\epsilon$.
	
	\begin{figure}
	\centering
		\begin{subfigure}[b]{0.45 \textwidth}
			\includegraphics[width = .8\textwidth]{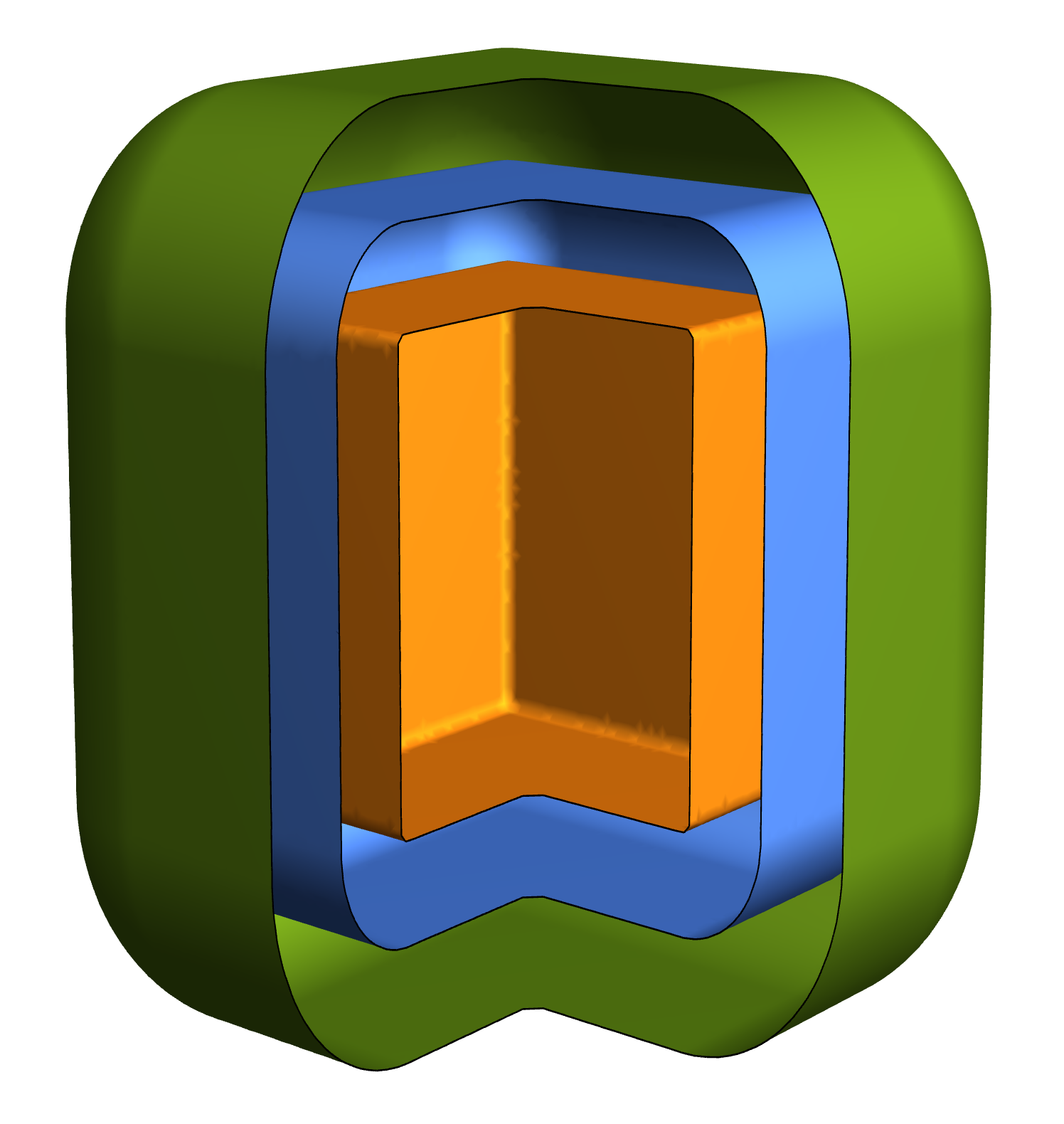}
			\caption{Examples of exaggerated $\epsilon$-rounding in blue and green of the orange cuboid interface.}
			\label{fig:e-rounding}
		\end{subfigure}
		\hspace{1 cm}
		\begin{subfigure}[b]{0.45 \textwidth}
			\includegraphics[width=0.9 \textwidth]{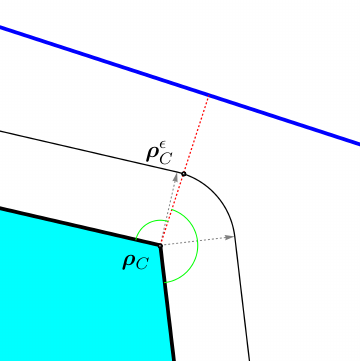}
			\caption{In 2D, $\bm{\rho}^\epsilon_C \in \partial \BB^\epsilon$ is always on the rounded corner forming obtuse angles in green.}
			\label{fig:roundedcorner}
		\end{subfigure}
		\caption{$\epsilon$-rounding figures}
		\label{fig:rounding}
	\end{figure}
	
	For a probable configuration $(\bm{x},R) \in SE(3)$, we are given the plane $\PP_{\tilde{\bm{n}}}'$ in the body-fixed frame. Now, the unique, closest point on the boundary $\partial \BB^\epsilon$ will always be on the rounded portion of a particular corner. This can be shown by first determining the closest point $\bm{\rho}_C \in \partial \BB$, which is one of the vertices $\{\bm{v}_j\}_{j=1}^l$. Then the closest point on $\partial \BB^\epsilon$ is $\epsilon$-distance along the normal vector towards the plane, i.e., 
	$
		\bm{\rho}^\epsilon_C (\bm{x},R) = \bm{\rho}_C(\bm{x},R) - \epsilon R^T \tilde{\bm{n}}.
	$
	In particular, this point $\bm{\rho}^\epsilon_C \in \partial \BB^\epsilon$, represented as a vector, always forms obtuse angles with all the surfaces and edges at the vertex $\bm{\rho}_C$ (see Figure \ref{fig:roundedcorner}). This means that $\bm{\rho}^\epsilon_C$ is unique and always on the rounded corner. 
	
	Finally, following the definition in \eqref{eqn:PhiConvex1}, define the collision detection function for the rounded convex polyhedron $\BB^\epsilon$ as
	\begin{equation}
	\label{eqn:PhiRoundedConvexPoly}
		\Phi(\bm{x},R) 
		=\tilde{\bm{n}}^T R \bm{\rho}^\epsilon_C(\bm{x},R) + \tilde{\bm{n}}^T \bm{x}
		=  \tilde{\bm{n}}^T R \bm{\rho}_C(\bm{x},R) + \tilde{\bm{n}}^T \bm{x} - \epsilon.
	\end{equation}
	Also, recall that $\bm{\rho}_C(\bm{x},R)$ is the vertex of the convex polyhedron $\BB$ that is closest to the plane. Then the partial derivatives are computed to be
	\begin{equation}
	\label{eqn:partialsRoundedConvexPoly}
		\left( \frac{\partial \Phi}{\partial \bm{x}}, \biggl.\frac{\partial \Phi}{\partial R} \right)\biggr\lvert_{(\bm{x},R)}
		= \left( \tilde{\bm{n}},  \tilde{\bm{n}} \bm{\rho}_C^T + \tilde{\bm{n}}^T R \frac{\partial \bm{\rho}_C}{\partial R}  - \epsilon \tilde{\bm{n}} \tilde{\bm{n}}^T R\right).
	\end{equation}
	
	%----------------------------------------------%
	%Section 6.4
	%----------------------------------------------%
	\subsection{Union and Intersection of Convex Rigid-Bodies}
	\label{sect:UnionAndIntersection}
	%----------------------------------------------%
	
	Recall that the signed distance function $\psi$ is a subset of the functions that can represent the interface implicitly. In this representation, we have many accessible geometric tools including boolean operations for advanced constructive solid geometry (CSG).  These boolean operations include the union, intersection, and complement, and the resulting rigid-bodies will be composed of the convex rigid-bodies in the previous discussion. However, they will no longer be convex in general after the boolean operations. Therefore, $\bm{\rho}_C \in \partial \BB$, the closest point to the plane, is not necessarily unique for certain configurations including configurations of impact. Nevertheless, we include this discussion because the LGVCI will still generally capture the dynamics where the collision set $\BB \cap \PP_{\tilde{\bm{n}}}'$ is a singleton and the CSG is minimal.
	
	In this section, we only consider two rigid-bodies for the union and intersection and discuss the collision detection and its partial derivatives for the resulting, constructed body. Notably, the new rigid-body must have its center of mass coinciding with the origin. Additional union/intersection of bodies can be considered but it should be minimal. This should be done with care due to the center of mass and its geometry which could significantly increase the chance of non-singleton intersection at the point of impact. 
		
	Suppose $\psi_1$ and $\psi_2$ are two SDFs whose corresponding convex bodies, $\BB_1$ and $\BB_2$, do not necessarily have the centers of mass at the origin. Let $\BB_1 \cap \BB_2 \neq \emptyset$, and suppose $\BB_{\cup} = \BB_1 \cup \BB_2$ is their union with a center of mass coinciding with the origin. Then the corresponding SDF $\psi_\cup : \R^3 \to \R$ is defined by
	\begin{equation}
	\label{eqn:SDFUnion}
		\psi_\cup(\bm{x}) = \min\{\psi_1(\bm{x}),\psi_2(\bm{x})\}.
	\end{equation}
	Furthermore, its gradient is given by
	\begin{equation}
	\label{eqn:GradSDFUnion}
		\nabla \psi_\cup(\bm{x})
		=
		\left\{
		\begin{array}{lr}
			\nabla \psi_1(\bm{x}) 	&	\text{if } \psi_1(\bm{x}) < \psi_2(\bm{x}),	\\
			\nabla \psi_2(\bm{x}) &	\text{if } \psi_2(\bm{x}) < \psi_1(\bm{x}),	\\
			\nabla \psi_1(\bm{x})	& 	\text{if } \psi_1(\bm{x}) = \psi_2(\bm{x}) \text{ and } \nabla \psi_1(\bm{x}) = \nabla \psi_2(\bm{x}), \\
			\text{Undefined}		&	\text{if } \psi_1(\bm{x}) = \psi_2(\bm{x}) \text{ and } \nabla \psi_1(\bm{x}) \neq \nabla \psi_2(\bm{x}).
		\end{array}
		\right.
	\end{equation}
	Recall that the gradient is needed to resolve the impact, so we would like to avoid the last case, which does not arise in most CSG. The case arise when $\bm{x} \in \partial \BB_1 \cap \partial \BB_2$ is the singleton for the intersection at impact and $\nabla \psi_1(\bm{x}) \neq \nabla \psi_2(\bm{x})$. However, this could arise in unusual examples of union such as a small ellipsoid completely sitting inside a larger ellipsoid where their boundaries share a point. 
	
	For the intersection of the bodies, $\BB_\cap = \BB_1 \cap \BB_2$, we also assume that its center of mass is at the origin. The corresponding SDF $\psi_\cap : \R^3 \to \R$ is defined as
	\begin{equation}
	\label{eqn:SDFIntersection}
		\psi_\cap(\bm{x}) = \max\{\psi_1(\bm{x}),\psi_2(\bm{x})\}.
	\end{equation}
	The gradient is given by
	\begin{equation}
	\label{eqn:GradSDFIntersection}
		\nabla \psi_\cap(\bm{x})
		=
		\left\{
		\begin{array}{lr}
			\nabla \psi_1(\bm{x}) 	&	\text{if } \psi_1(\bm{x}) > \psi_2(\bm{x}),	\\
			\nabla \psi_2(\bm{x}) &	\text{if } \psi_2(\bm{x}) > \psi_1(\bm{x}),	\\
			\nabla \psi_1(\bm{x})	& 	\text{if } \psi_1(\bm{x}) = \psi_2(\bm{x}) \text{ and } \nabla \psi_1(\bm{x}) = \nabla \psi_2(\bm{x}), \\
			\text{Undefined}		&	\text{if } \psi_1(\bm{x}) = \psi_2(\bm{x}) \text{ and } \nabla \psi_1(\bm{x}) \neq \nabla \psi_2(\bm{x}).
		\end{array}
		\right.
	\end{equation}
	For this gradient, the last case is as similarly unlikely as the improbable configurations discussed for the general convex rigid-body. Put another way, the gradient is generically well-defined.
	
	Lastly, the complement is also introduced since it is useful for CSG. Given a convex rigid-body $\BB$ and the SDF $\psi$, the SDF for the complement $\BB^{\mathrm{C}}$ is given by
	\begin{equation}
	\label{eqn:SDFComplement}
		\psi_{\mathrm{C}}(\bm{x}) = -\psi(\bm{x}).
	\end{equation}
	
%----------------------------------------------%
%----------------------------------------------%
	%SECTION 7: Numerical Experiments
%----------------------------------------------%
%----------------------------------------------%
\section{Numerical Experiments}
\label{sect:NumericalExp}
	
	 Numerical experiments are performed using Algorithm \ref{alg:algShort} for the following four hybrid systems consisting of different rigid-bodies and planes:
	\begin{enumerate}[label={(\Roman*):}]
		\item Triaxial ellipsoid over the horizontal plane.
		\item Triaxial ellipsoid over tilted plane.
		\item Union of ellipsoids over the horizontal plane.
		\item A cube over the horizontal plane.
	\end{enumerate}
	
	Recall that the horizontal plane has the normal vector $\bm{n}^T = (0, 0, 1)$ and passes through the origin. In Case II, the tilted plane is a two-degree counterclockwise rotation of the horizontal plane about the $y$-axis. Furthermore, the rigid-bodies are described by the parameters given in Table \ref{tab:params}.
	
	\begin{table}[hb]
	\begingroup
	\begin{tabular}{c|ccc}
					&		Rigid-Body Properties 	&	Center(s) &	Inertia Matrix ($J$) \\ \hline
		Case I 	&
			$I_\epsilon = \text{diag}(2,3,4)$ 	&
  			$(0,0,0)$ &
  			$\text{diag}(5,4,2.6)$  \\ \hline
		Case II 	&
  			$I_\epsilon = \text{diag}(2,3,4)$	&
  			$(0,0,0)$ &
 			$\text{diag}(5,4,2.6)$ \\ \hline
		Case III	&
  			\begin{tabular}{@{}c@{}}$I_{\epsilon_{1}} = \text{diag}(3,4,5)$ \\ $I_{\epsilon_{2}} = \text{diag}(6,1,1)$	\end{tabular}		&
  			\begin{tabular}{@{}c@{}}$(1.5,0,0) + \bm{c}$ \\ $(-4.5,0,0) + \bm{c}$ \end{tabular} &
  			$\begin{bsmallmatrix}  7.5932718 & 6 \mathrm{e}{-7} & -4 \mathrm{e}{-7}\\- 6 \mathrm{e}{-7} & 9.9326434 & 0 \\4 \mathrm{e}{-7} & 0 & 8.2731252 \end{bsmallmatrix}$ \\ \hline
		Case IV &
			$s = 2\sqrt{3}, \epsilon = 10^{-13}$ &
			$(0,0,0)$ &
			$\text{diag}(2,2,2)$
	\end{tabular}
	\endgroup
	\caption{Parameters for the hybrid systems where $\bm{c} = (-0.9937128,0,0)$.}
	\label{tab:params}
	\end{table}
	
	In Case III, the body is a union of two ellipsoids whose centers are shifted by $\bm{c}$ so that the centroid coincide with the origin. Lastly, in Case IV, the rigid-body is a cube whose side-length is denoted by $s$, and its $\epsilon$-rounding parameter is $\epsilon = 10^{-13}$. Table \ref{tab:params} also includes the standard inertia matrices $J$, and these matrices in Case I, II, and IV are computed using standard formulas assuming constant uniform density: Namely, 
	\begin{align*}
		J_{\text{Ellipsoid}} 
			&=\frac{1}{5} m \,\text{diag}(b^2+c^2,a^2+c^2,a^2+b^2),
		& J_{\text{Cube}}
			&= \frac{1}{6}ms^2 I_3,
	\end{align*}
	where $I_\epsilon = \text{diag}(a,b,c)$ give the parameters for $J_{\text{Ellipsoid}}$, and $m$ is the total mass of the rigid-body. In all of the cases, $m = 1$ including Case III where the standard matrix is computed and scaled reasonably using the built-in function \texttt{MomentOfInertia} in Mathematica 12. We fixed the following parameters for all cases including $g = 9.80665$ for the gravitational acceleration, $h=0.01$ for the timestep, and $\epsilon_{\text{tol}} = 10^{-15}$ for tolerance.
	
	For all four cases, the initial values are 
	\begin{align*}
		\bm{x}_0
			&= (0,0,10)^T, 
		&\bm{\gamma}_0
			&= (2,2,10)^T, 
		&\bm{\Pi}_0
			&= (4,-4,4)^T, 
		&R_0 
			&= I_3.
	\end{align*}
	
	Lastly, note that the Zeno phenomenon does not arise in our four simulations because all the rigid-bodies are $C^\infty$ smooth, including the cube which was regularized using $\epsilon$-rounding. However, the Zeno phenomenon could be observed if we considered a non-smooth rigid-body whose boundary is $C^k$ for some finite $k \geq 2$.
	
	%----------------------------------------------%
	%Section 7.1
	%----------------------------------------------%
	\subsection{Snapshots}
	\label{sect:Snapshots}
	%----------------------------------------------%
	
	\begin{figure}[htbp]
		\centering
		\begin{subfigure}[b]{0.475 \textwidth}
			\centering
			\includegraphics[width=\textwidth]{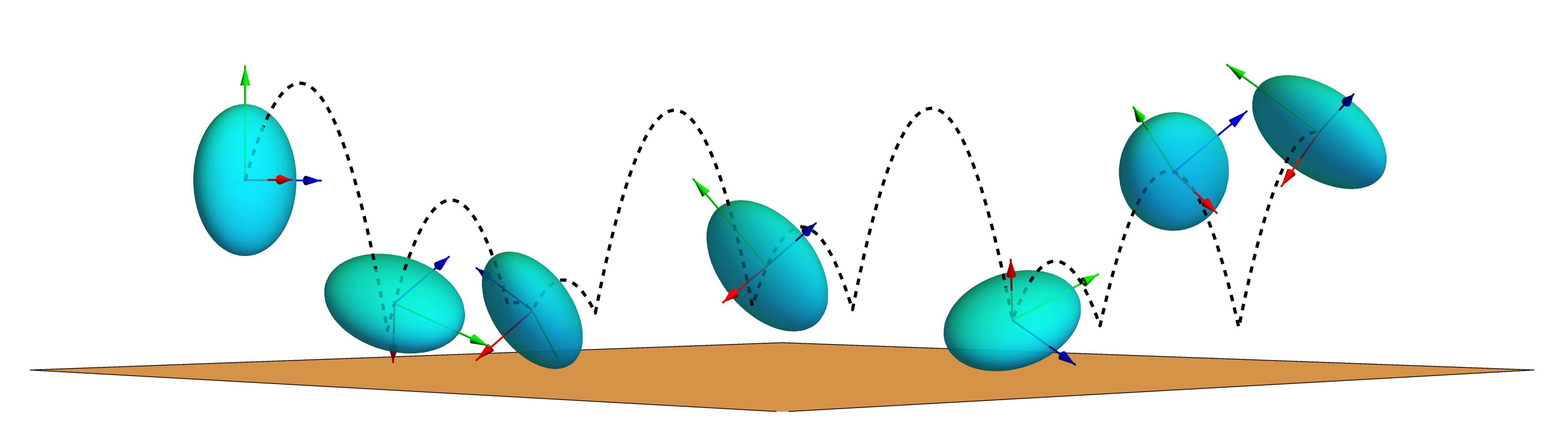}
			\caption{Triaxial ellipsoid over the horizontal plane}
			\label{fig:triaxial_snaps}
		\end{subfigure}
		\hfill
		\begin{subfigure}[b]{0.475 \textwidth}
			\centering
			\includegraphics[width=\textwidth]{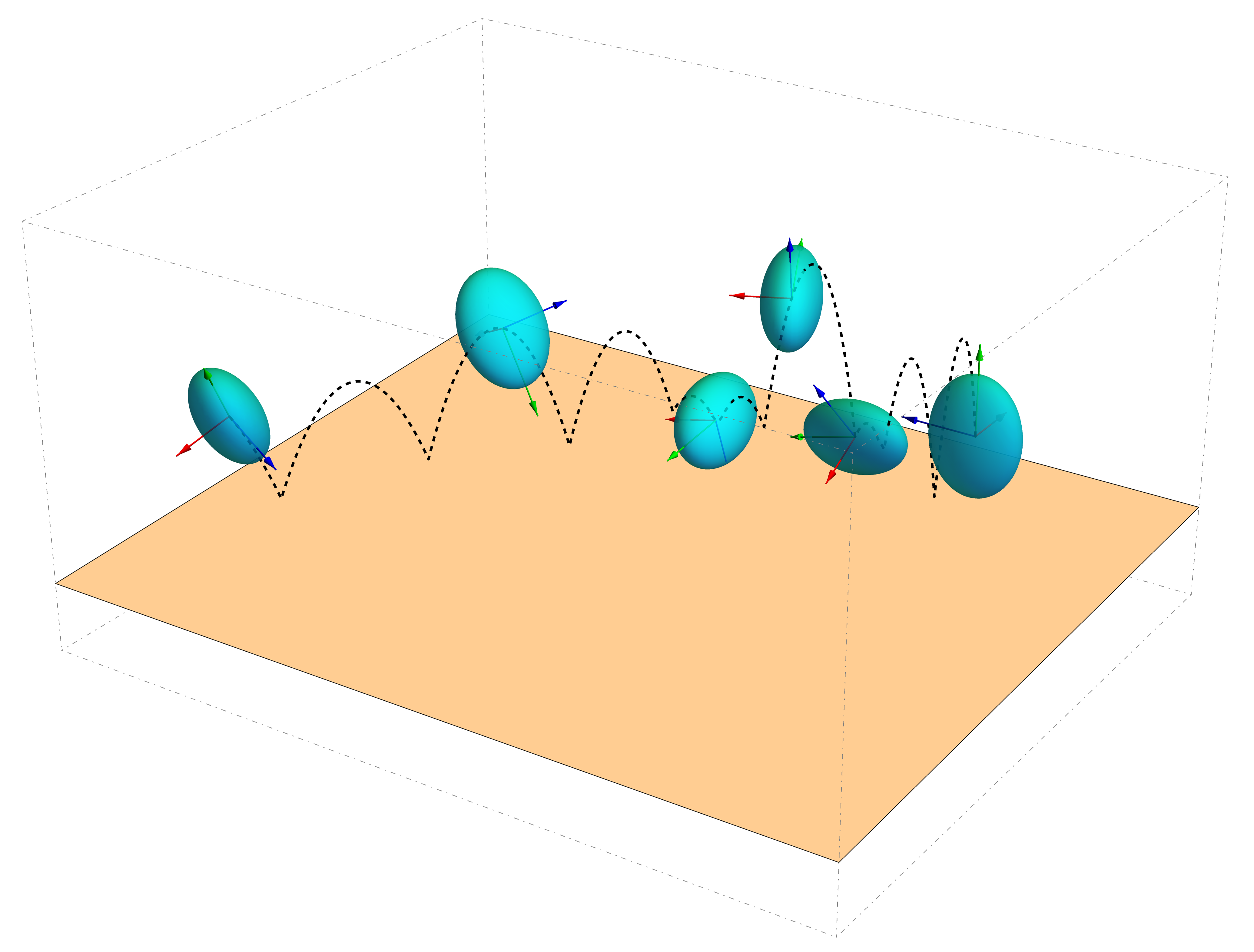}
			\caption{Triaxial ellipsoid over a tilted plane}
			\label{fig:tilted_snaps}
		\end{subfigure}
		\vskip\baselineskip
		\begin{subfigure}[b]{0.475 \textwidth}
			\centering
			\includegraphics[width=\textwidth]{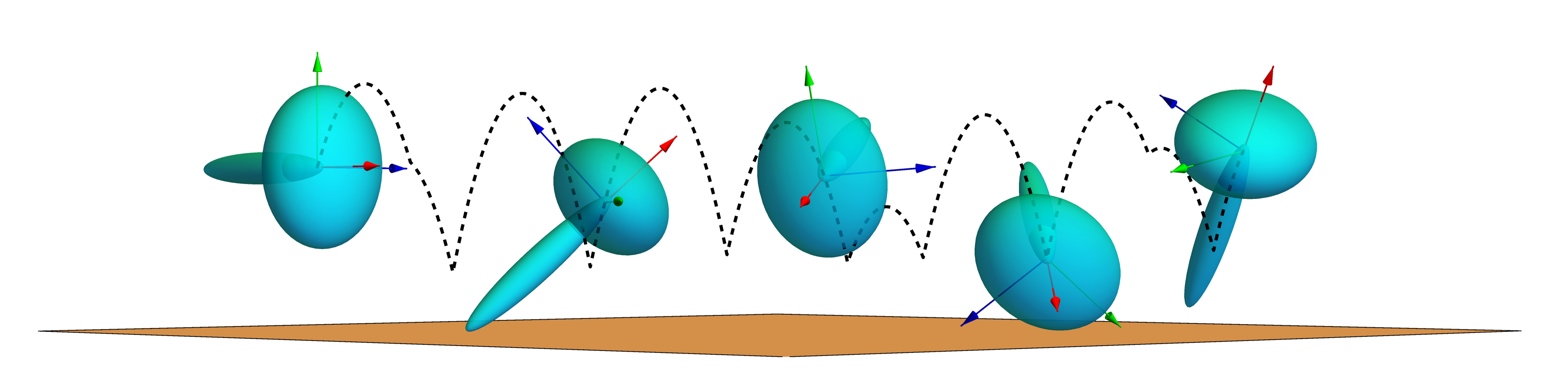}
			\caption{Union of ellipsoids over the horizontal plane}
			\label{fig:union_snaps}
		\end{subfigure}
		\hfill
		\begin{subfigure}[b]{0.475 \textwidth}
			\centering
			\includegraphics[width=\textwidth]{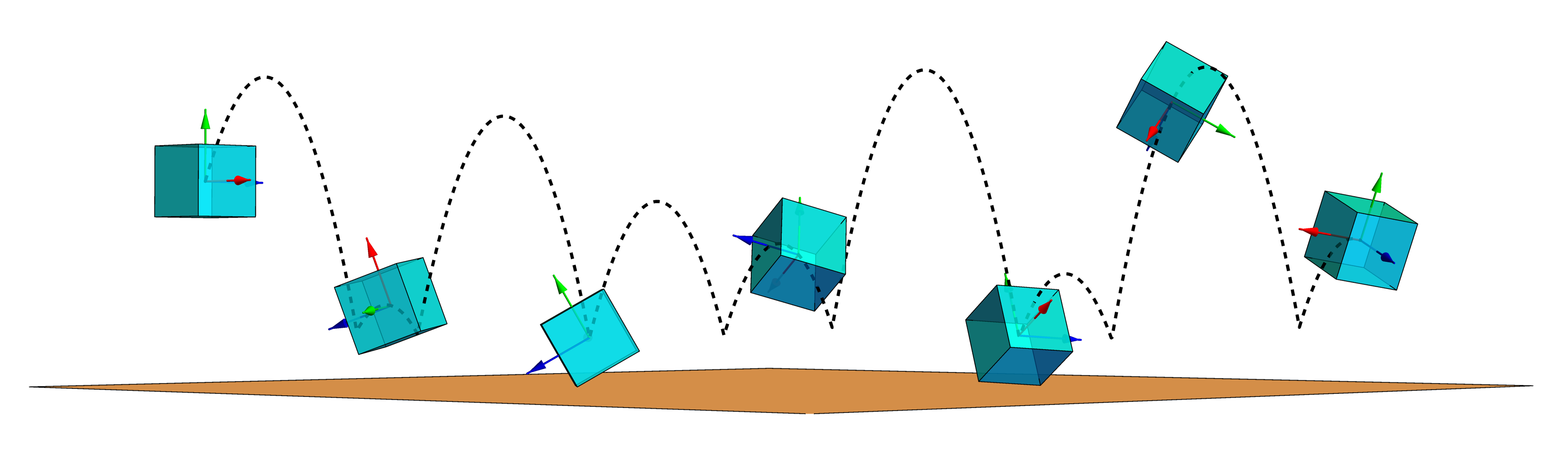}
			\caption{A cube over the horizontal plane}
			\label{fig:cube_snaps}
		\end{subfigure}
		\caption{Snapshots of configurations, including some collisions, of the four hybrid systems.}
		\label{fig:snapshots}
	\end{figure}
	
	Fortunately, our numerical experiments for the hybrid systems can be visualized as a sequence of snapshots, as shown in Figure \ref{fig:snapshots}. We can make a number of collective observations for Case I, II, and IV, which share the horizontal plane in common. In particular, the top-down view of the path of the center of mass for each rigid-body is a straight line. This is immediate because the updates on the linear momenta away from collisions are only affected by the gravitational direction, the $z$-component. Also, the instantaneous update on the linear momenta after the collision is only dependent on $\frac{\partial \Phi}{\partial \bm{x}}^T = \bm{n}^T = (0,0,1)$, affecting only the $z$-component again. Hence, only Case III with a tilted plane exhibits a curved path for the center of mass when viewed top-down.
	
	%----------------------------------------------%
	%Section 7.2
	%----------------------------------------------%
	\subsection{Transfer of Energy}
	\label{sect:TransferofEnergy}
	%----------------------------------------------%
	\begin{figure}[p]
		\centering
		\begin{subfigure}[b]{1.\textwidth}
		\includegraphics{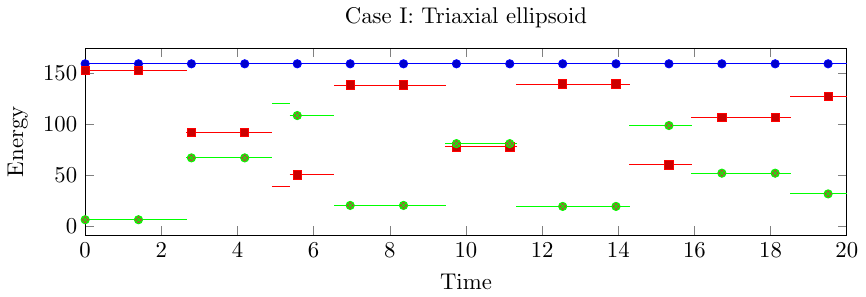}
		\end{subfigure}
		\begin{subfigure}[b]{1.\textwidth}
		\includegraphics{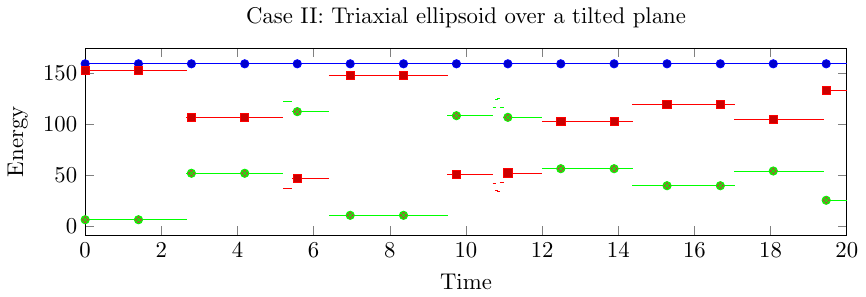}
		\end{subfigure}
		\begin{subfigure}[b]{1.\textwidth}
		\includegraphics{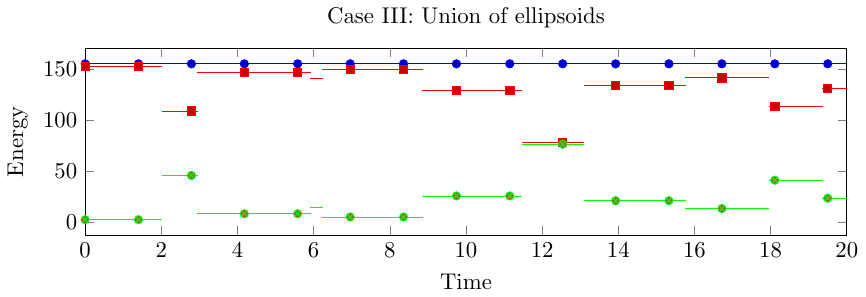}
		\end{subfigure}
		\begin{subfigure}[b]{1.\textwidth}
		\includegraphics{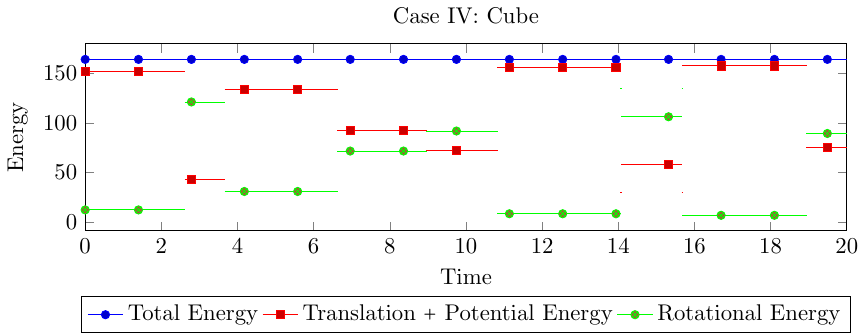}
		\end{subfigure}
		\caption{Plots of total energy (T.E.), translational $+$ potential energy (T.P.E), and rotational energy (R.E) of different hybrid systems demonstrating the exchange of energies between T.P.E and R.E. after each collision.}
		\label{fig:energy_plots}
	\end{figure}
	
	Moreover, each hybrid system is closed, and their collisions are elastic. Therefore, the total energy within each system is conserved, which can be observed in Figure \ref{fig:energy_plots}, where the data is taken over a short time interval. By conservation of angular momentum, the rotational kinetic energy (R.E.) is constant in the time intervals that are away from the collisions, and so the translational kinetic $+$ potential energy (T.P.E.) is also constant. However, we observe a transfer of energies between T.P.E and R.E. after each collision because the normal force, perpendicular to the plane, is no longer passing through the center of mass; as a result, this not only imparts an instantaneous impulse on the center of mass, changing its linear momentum but also an instantaneous angular impulse on the body, changing the angular momentum after the collision. By the conservation of energy in the system, the sum of T.P.E. and R.E. still gives the total energy, but these changes in momenta induce the transfer of energy between T.P.E. and R.E., which are the jumps in Figure \ref{fig:energy_plots}. 
	
	%----------------------------------------------%
	%Section 7.3
	%----------------------------------------------%
	\subsection{Short-Term \& Long-Term Behaviors}
	\label{sect:ShortLongTermBehaviors}
	%----------------------------------------------%
	\begin{figure}[p]
		\centering
		\begin{subfigure}[b]{1.\textwidth}
		\includegraphics{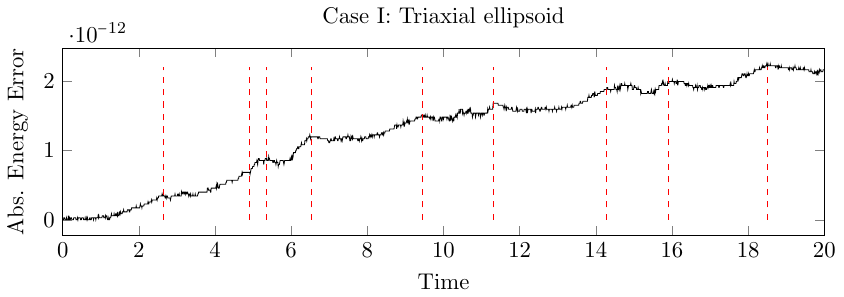}
		\end{subfigure}
		\begin{subfigure}[b]{1.\textwidth}
		\includegraphics{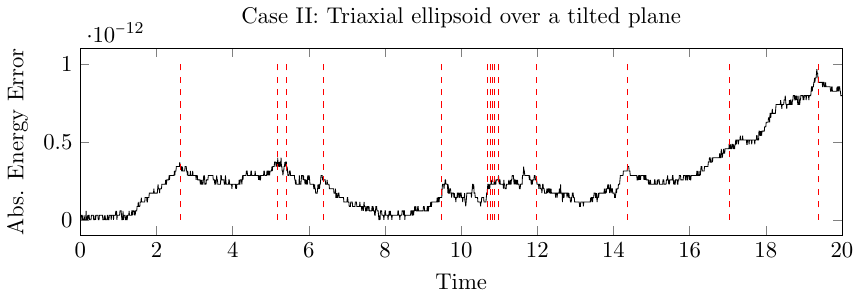}
		\end{subfigure}
		\begin{subfigure}[b]{1.\textwidth}
		\includegraphics{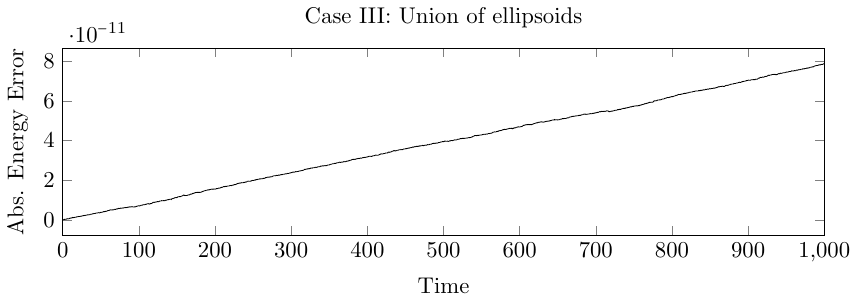}
		\end{subfigure}
		\begin{subfigure}[b]{1.\textwidth}
		\includegraphics{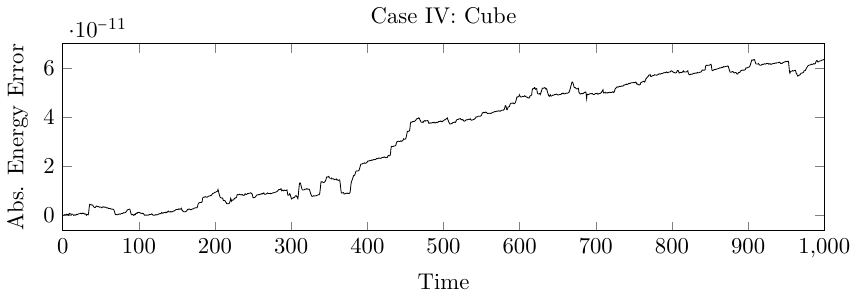}
		\end{subfigure}
		\caption{Short-term absolute energy errors are shown for the systems of triaxial ellipsoid and its counterpart with the tilted plane. Long-term absolute energy errors are shown for the systems of union of ellipsoids and cube, demonstrating energy drift after many collisions.}
		\label{fig:energyerr_plots}
	\end{figure}
	
	Both short-term and long-term energy behaviors are shown for our hybrid systems in Figure \ref{fig:energyerr_plots}. For Case I and Case II, short-term energy behaviors are shown to illustrate how collisions affect the conservation properties. In particular, the red-dashed lines in both figures indicate the times of the impacts, and we see that the trend of the absolute energy errors shift after each impact; however, the magnitude of the error about each collision is essentially the same, so the jump conditions $\tilde{F}_{\jump}$ still produce near energy conservation in the same manner as the discrete flow away from impacts. 
	
	We also explore the long-term energy behaviors in Case III and Case IV with $10^5$ integration steps; there are roughly $806$ collisions that occur in Case III and roughly $652$ collisions in Case IV. Overall, the LGCVI exhibits good long-time near energy conservation; however, there is a drift in both absolute energy errors, which is attributed to the numerous collisions in each hybrid system. This is the case because the discrete Lagrangian variational mechanics away from the collisions uses a fixed timestep $h$, and preserves a modified Hamiltonian up to an exponentially small error for exponentially long times by virtue of backward error analysis. However, the modified Hamiltonian is an asymptotic expansion in $h$, and a different $h$ is taken during collisions in order to resolve the collision time accurately. Hence, a slightly different modified Hamiltonian is preserved, resulting in a small energy drift after each collision.
		
	Interestingly, we observe that this drift in energy appears to be roughly monotonic; otherwise, the absolute energy error plots would have negative trends for long periods of time as well. This could be the consequence of our use of the bisection method to resolve the collision time and our bias in choosing a configuration that is admissible, so we choose an approximation such that $\Phi(\bm{x}_i,R_i) \approx 0$ is always positive.
	
	%----------------------------------------------%
	%Section 7.4
	%----------------------------------------------%
	\subsection{Sensitivity to Initial Conditions}
	\label{sect:SensitiveToInitials}
	%----------------------------------------------%
	
	The hybrid systems involving rigid-bodies that are not spherical are, in general, sensitive to initial conditions by nature. As a result, we expect our collision algorithm to capture this when we apply slight changes in the initial position or attitude of the rigid-body, while maintaining the same linear and angular momentum. By following the path of the center of mass height (C.M. height) over the horizontal plane, we observe this sensitivity in Figure \ref{fig:CMsensitive}. The black dashed line is the path of C.M. height for the original initial conditions $(\bm{x_0},R_0)$. The path of C.M. height of the slightly perturbed position $(\bm{x_0}+\bm{\delta x}, R_0)$ and attitude $(\bm{x_0}, R_0\delta R)$ are plotted in square-blue and circle-red lines, respectively; the perturbations are set as
		$
		\bm{\delta x} = (0,0,10^{-8})^T
		$
		and 
		$\delta R = 
		\begin{bsmallmatrix}
			\cos \theta	&		0		&		\sin \theta	\\
			0					&		1		&		0				\\
			-\sin \theta	&		0		&		\cos \theta
		\end{bsmallmatrix}
		$
	where $\theta = 10^{-8}\pi$. By observing the C.M. heights, the discrete flows of the two perturbed systems appear to be the same as the unperturbed system for the first few collisions; however, the flows noticeably diverge after the eighth collision with the plane, and the plots in Figure \ref{fig:CMsensitive} clearly illustrate this.
	
	\begin{figure}[htbp]
		\centering
		\includegraphics{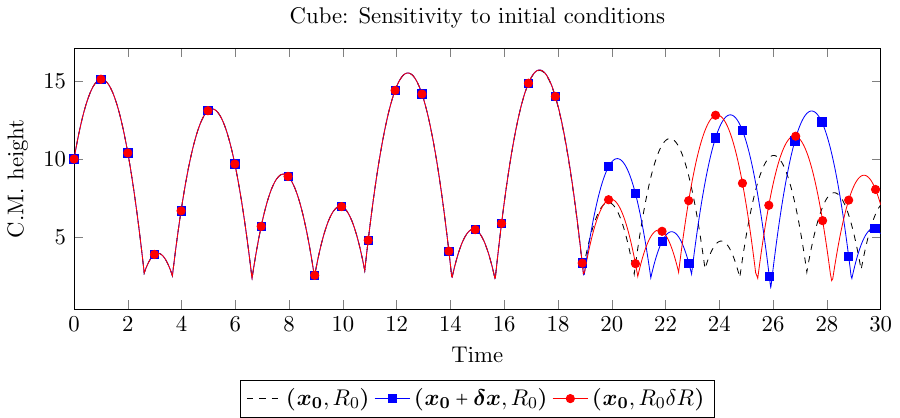}
		\caption{Plot of the height of the center of mass of the cube, for a given initial condition and two slight perturbations in position and attitude, to illustrate that the sensitivity to initial conditions is captured by our proposed collision algorithm.}
		\label{fig:CMsensitive}
	\end{figure}
	
	To ensure that the divergence of these trajectories is intrinsic to the nature of the system instead of being a consequence of the design of the algorithm, we look at the absolute errors: 
	\[
		\textrm{Err}_{\textrm{per}}(t) = \|\bm{x}_\textrm{per}(t) - \bm{x}(t)\|_2 + \|R_\textrm{per}(t)- R(t)\|_2,
	\]
	where $(\bm{x}_\textrm{per}(t), R_\textrm{per}(t))$ represent the discrete flow for the perturbed initial conditions either in position $(\bm{x}_\textrm{pos}(t),R_\textrm{pos}(t))$ or in attitude $(\bm{x}_\textrm{att}(t),R_\textrm{att}(t))$. Of course $(\bm{x}(t),R(t))$ is the discrete flow for the system with unperturbed initial conditions $(\bm{x}_0,R_0)$. Note also that the matrix norm above is induced by the 2-norm. From the small perturbations, the initial absolute errors are 
	$
		\textrm{Err}_{\textrm{pos}}(0) = 10^{-8}
	$
	and
	$
		\textrm{Err}_{\textrm{att}}(0) \approx 1.75 \cdot 10^{-8},
	$
	and these errors change in magnitude after each collision as seen in Figure \ref{fig:CMsensitiveErr}, illustrating that the differences between the trajectories are amplified after each collision. Eventually, the absolute errors saturate because the perturbed systems become unrelated to the original system beside conserving approximately the same total energy.
	
		\begin{figure}[htbp]
	\centering
	\includegraphics{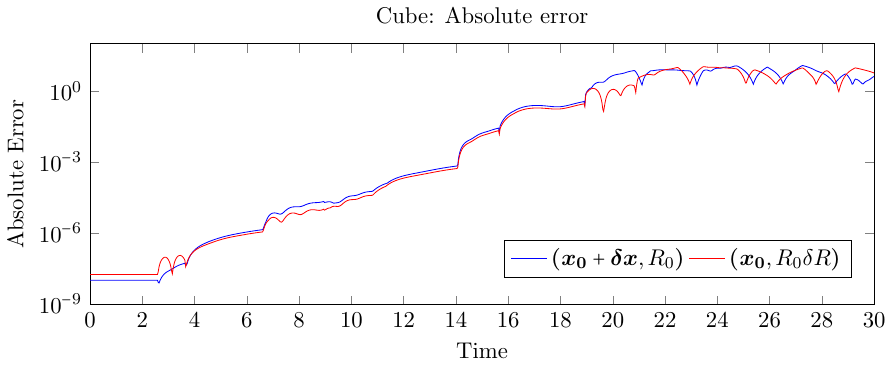}
		\caption{Absolute Error $\|\bm{x}_\text{per}(t) - \bm{x}(t)\|_2 + \|R_\text{per}(t)- R(t)\|_2$ plots for perturbed initial position and attitude in blue and red, respectively. The matrix norm is the induced 2-norm.}
		\label{fig:CMsensitiveErr}
	\end{figure}

%----------------------------------------------%
%----------------------------------------------%
	%SECTION 8: Conclusion
%----------------------------------------------%
%----------------------------------------------%
\section{Conclusions and Future Directions}
\label{sect:Conclusion}
	We have developed an algorithm to simulate a class of hybrid systems that is the extension of the classical bouncing ball hybrid system. Now, this class of hybrid systems in 3-dimensions comprise of a general convex rigid-body bouncing elastically on a horizontal or tilted plane. The resulting algorithm is called a Lie group variational collision integrator (LGVCI) which is based on a combination of the work done in nonsmooth Lagrangian mechanics for collision variational integrators and Lie group variational integrators. Consequently, the LGVCI provide a combination of discrete flow maps for the hybrid systems away from the points of collision and jump conditions to update the instantaneous state after each collision. They are also symplectic and momentum-preserving by design, and they are heavily dependent on the collision detection function $\Phi$.
	
	Initially, we developed the algorithm for a model problem involving an ellipsoid and a horizontal plane. However, our theory easily extended to more general systems involving tilted planes, unions and/or intersections of convex rigid-bodies by modifying and constructing the collision detection function. Furthermore, we introduced a convenient and straightforward regularization using $\epsilon$-rounding for the collision responses of convex rigid-bodies with corners. In general, these bodies are convex polyhedra. This development avoids the need for complicated nonlinear convex analysis of the corner impacts involving different inclusions and the computation of normal/tangent cones at the configuration of impact. Consequently, the regularization provides a unique deterministic response after the collision using the jump conditions we have derived, while still exhibiting the full range of potential outcomes that the formulation involving differential inclusions and normal/tangent cones would exhibit by varying the initial conditions slightly.
		
	For future research,  we intend to extend our variational collision integrators approach to hybrid systems with external and contact forces, as briefly discussed in \citep{fetecau2003nonsmooth} for the continuous setting. Moreover, we can consider elastic bodies (e.g., hyperelastic materials) which can be formalized in the discrete variational method using the appropriate elastic potential energies. From here, we may consider the analysis of collisions for convex-nonconvex rigid-bodies and then construct numerical integrators for these extensions. This is naturally a topic of interest as they more closely resemble real-world, complex systems that exhibit energy dissipation during collisions.
		
	There are also interesting applications to geometric control and optimal control on the Special Euclidean group $SE(3)$ and its submanifolds with boundary, which is the setting of this paper. We propose to study these geometric control problems for two reasons: First, the configuration space $SE(3)$ is global, and its elements represent configurations of a real-world object uniquely, which is advantageous for describing dynamical systems analytically and for constructing numerical methods based on variational principles. This is particularly critical when considering systems that exhibit large rotational motions that cannot be effectively described using local coordinate based approaches. Second, novel algorithms that are both efficient and geometrically exact can be developed for systems on $SE(3)$ with unilateral constraints, which would extend the results demonstrated in \citep{taeyoung2006optimal} for bilateral constraints.

\bibliographystyle{siamplain}
\bibliography{collision_ML}

\end{document}